\documentclass{article}
\usepackage[utf8]{inputenc}
\usepackage{amsfonts}
\usepackage[T1]{fontenc}
\usepackage{amssymb,amsthm,amsmath,mathtools,amsrefs,MnSymbol}
\usepackage{amscd,amstext,units,amsfonts,amsbsy}
\usepackage[utf8]{inputenc}
\usepackage[colorlinks=true]{hyperref}
\usepackage{color,xcolor}
\usepackage{url}
\usepackage{comment}
\usepackage{tikz-cd}

\usepackage[toc,page]{appendix}
\usepackage{tikz-cd}
\usepackage{graphicx}
\graphicspath{ {images/} }
\usepackage{fancyhdr}
\newcommand{\bp}{\begin{prop}}
\newcommand{\ep}{\end{prop}}
\newcommand{\bd}{\begin{definicion}}
\newcommand{\ed}{\end{definicion}}
\newcommand{\bl}{\begin{lema}}
\newcommand{\el}{\end{lema}}
\newcommand{\bh}{\begin{hecho}}
\newcommand{\eh}{\end{hecho}}
\newcommand{\bpreg}{\begin{preg}}
\newcommand{\epreg}{\end{preg}}
\newcommand{\bo}{\begin{obs}}
\newcommand{\eo}{\end{obs}}
\newcommand{\bcon}{\begin{conj}}
\newcommand{\econ}{\end{conj}}
\newcommand{\brmk}{\begin{rmk}}
\newcommand{\ermk}{\end{rmk}}
\newcommand{\bc}{\begin{corol}}
\newcommand{\ec}{\end{corol}}
\newcommand{\bconst}{\begin{const}}
\newcommand{\econst}{\end{const}}

\newcommand{\bitem}{\begin{itemize}}
\newcommand{\eitem}{\end{itemize}}
\newcommand{\bt}{\begin{teor}}
\newcommand{\et}{\end{teor}}
\newcommand{\be}{\begin{ejem}}
\newcommand{\ee}{\end{ejem}}
\newcommand{\bnot}{\begin{nota}}
\newcommand{\enot}{\end{nota}}

\newtheorem*{thm}{Theorem}

\newtheorem*{conj}{Conjecture}

\newtheorem{theorem}{Theorem}[section]
\newtheorem{claim}{Claim}[theorem]

\newtheorem{corollary}[theorem]{Corollary}

\newtheorem{definition}[theorem]{Definition}

\newtheorem{lemma}[theorem]{Lemma}
\newtheorem{notation}[theorem]{Notation}

\newtheorem{proposition}[theorem]{Proposition}

\newtheorem{fact}[theorem]{Fact}

\newtheorem{remark}[theorem]{Remark}
\usepackage{fancyvrb }
\newcommand{\forkindep}[1][]{%
  \mathrel{
    \mathop{
      \vcenter{
        \hbox{\oalign{\noalign{\kern-.3ex}\hfil$\vert$\hfil\cr
              \noalign{\kern-.7ex}
              $\smile$\cr\noalign{\kern-.3ex}}}
      }
    }\displaylimits_{#1}
  }
}

\numberwithin{equation}{section}

\newcommand{\tp}{\operatorname{tp}}
\newcommand{\dcl}{\operatorname{dcl}}
\newcommand{\acl}{\operatorname{acl}}
\newcommand{\VS}{\operatorname{VS}}
\newcommand{\val}{\operatorname{val}}
\newcommand{\res}{\operatorname{res}}
\newcommand{\primes}{\operatorname{Primes}}
\newcommand{\Pres}{\operatorname{Pres}}
\newcommand{\cof}{\operatorname{cof}}
\newcommand{\red}{\operatorname{red}}
\newcommand{\germ}{\operatorname{germ}}
\newcommand{\Stab}{\operatorname{Stab}}
\title{Elimination of imaginaries in $\mathbb{C}((\Gamma))$ }
\author{Mariana Vicar\'ia }
\begin{document}
    \maketitle
\begin{abstract}
In this paper we study elimination of imaginaries in henselian valued fields of equicharacteristic zero and residue field algebraically closed. The results are sensitive to the complexity of the value group. We focus first on the case where the ordered abelian group has finite spines, and then prove a better result for the dp-minimal case. In \cite{Vicaria} it was shown that an ordered abelian with finite spines weakly eliminates  imaginaries once one adds sorts for the quotient groups $\Gamma/ \Delta$  for each definable convex subgroup $\Delta$, and sorts for the quotient groups $\Gamma/(\Delta+ l\Gamma)$ where $\Delta$ is a definable convex subgroup and $l \in \mathbb{N}_{\geq 2}$. We refer to these sorts as the \emph{quotient sorts}. In \cite{dpminimal} F. Janke, P. Simon and E. Walsberg characterized $dp$-minimal ordered abelian groups as those without singular primes, i.e. for every prime number $p$ $[\Gamma:p\Gamma]< \infty$. \\

 We prove the following two theorems: 
 \begin{thm}Let $K$ be a henselian valued field of equicharacteristic zero with residue field algebraically closed and value group of finite spines. Then $K$ admits weak elimination of imaginaries once one adds codes for all the definable $\mathcal{O}$-submodules of $K^{n}$ for each $n \in \mathbb{N}$, and the quotient sorts for the value group. 
 \end{thm}

 \begin{thm}Let $K$ be a henselian valued field of equicharacteristic zero, residue field algebraically closed and dp-minimal value group. Then $K$ eliminates imaginaries once one adds codes for all the definable $\mathcal{O}$-submodules of $K^{n}$ for each $n \in \mathbb{N}$, the quotient sorts for the value group and constants to distinguish the elements of each of the finite groups $\Gamma/\ell\Gamma$, where $\ell \in \mathbb{N}$.
\end{thm}
\end{abstract}
\newpage
  \tableofcontents 
\newpage
\section{Introduction}
%\textcolor{red}{RE-DO: mejorar la motivacion, la escritura y por ultimo una introduccion para el no teorista de modelos.}
The model theory of henselian valued fields has been a major topic of study during the last century, it was initiated by Robinson's model completeness results for algebraically closed valued fields  in \cite{robinson}.  Remarkable work has been achieved by Haskell, Hrushovski and Macpherson to understand the model theory of algebraically closed valued fields. In a sequence of papers \cite{HHM} and \cite{HHM2} they developed the notion of stable domination, that rather than being a new form of stability should be understood as a way to apply techniques of stability in the setting of valued fields. Further work of Ealy, Haskell and Ma\v{r}\'{i}cov\'{a} in  \cite{Clif} for the setting of real closed convexly valued fields, suggested that the notion of having a stable part of the structure was not fundamental to achieve domination results and indicated that the right notion should be residue field domination or domination by the sorts internal to the residue field. Our main motivation for the present document arises from the natural question of how much further a notion of residue field domination could be extended to broader classes of valued fields to gain a deeper model theoretic insight of henselian valued fields, and the first step is finding a reasonable language where the valued field will eliminate imaginaries. \\

The starting point in this project relies on the Ax-Kochen theorem, which states that the first order theory of a henselian valued field of equicharacteristic zero or unramified mixed characteristic with perfect residue field is completely determined by the first order theory of its valued  group and its residue field. A natural principle follows from this theorem: model theoretic questions about the valued field itself can be understood by reducing them to its residue field, its value group and their interaction in the field.\\

A fruitful application of this principle has been achieved to describe the class of definable sets. For example, in \cite{Pas} Pas proved field quantifier elimination relative to the residue field and the value group once angular component maps are added in the equicharacteristic case. Further studies of Basarab and F.V. Kuhlmann show a quantifier elimination relative to the $RV$ sorts  [see \cite{Basarab}, \cite{Kuhlmann} respectively]. \\

The question of whether a henselian valued field eliminates imaginaries in a given language is of course subject to the complexity of its value group and its residue field as both are interpretable structures in the valued field itself.  The case for algebraically closed valued fields was finalized by Haskell, Hrushovski and Macpherson in their important work \cite{HHM2} , where elimination of imaginaries for $ACVF$ is achieved once the  geometric sorts $S_{n}$ (codes for the $\mathcal{O}$-lattices of rank $n$ ) and $T_{n}$ (codes for the residue classes of the elements in $S_{n}$)  are added.  This proof was later significantly simplified by Will Johnson in  \cite{will} by using a criterion isolated by Hurshovski [ see \cite{criteria}]. \\

Recent work has been done to achieve elimination of imaginaries in some other examples of henselian valued fields, as the case of separably closed valued fields  in \cite{separablyclosed}, the $p$-adic case in \cite{padics} or  enrichments of ACVF in \cite{silvancito}.\\
 
 However, the above results are all obtained for particular instances of henselian valued fields while the more general approach of obtaining a relative statement for broader classes of henselian valued fields  is still a very interesting open question.\\
 Following the Ax-Kochen style principle, it seems natural to first attempt to solve this question by looking at the problem in two orthogonal directions: one by making the residue field as docile as possible and studying which troubles would the value group bring into the picture, or by making the value group  tame and understanding the difficulties that the residue field would contribute to the problem. \\
 Hils and Rideau \cite{hablando} had proved that under the assumption of having a definably complete value group and requiring that the residue field eliminates the $\exists^\infty$ quantifier, then any definable set admits a code once the geometric sorts and the linear sorts are added to the language. Any definably complete ordered abelian group is either divisible or a $\mathbb{Z}$-group (i.e. a model of Presburger Arithmetic).\\
 
 This paper is addressing the first approach in the setting of henselian valued fields of equicharacteristic zero. We suppose the residue field to be algebraically closed and we obtain results which are sensitive to the complexity of the value group. We first analyze the case where the value group has finite spines. An ordered abelian with finite spines weakly eliminates imaginaries once we add sorts for the quotient groups $\Gamma/ \Delta$  for each definable convex subgroup $\Delta$, and sorts for the quotient groups $\Gamma/ \Delta+ l\Gamma$ where $\Delta$ is a definable convex subgroup and $l \in \mathbb{N}_{\geq 2}$. We refer to these sorts as the quotient sorts. The first result that we obtain is: 
 \begin{theorem} Let $K$ be a valued field of equicharacteristic zero, residue field algebraically closed and value group with finite spines. Then $K$ admits weak elimination of imaginaries once we add codes for all the definable $\mathcal{O}$-submodules of $K^{n}$ for each $n \in \mathbb{N}$, and the quotient sorts for the value group. 
 \end{theorem}
Later, we prove a better result for the dp-minimal case, this is: 
 \begin{theorem}Let $K$ be a henselian valued field of equicharacteristic zero, residue field algebraically closed and dp-minimal value group. Then $K$ eliminates imaginaries  once we add codes for all the definable $\mathcal{O}$-submodules of $K^{n}$ for each $n \in \mathbb{N}$, the quotient sorts for the value group and constants to distinguish the elements of the finite groups $\Gamma/ \ell \Gamma$ where $l \in \mathbb{N}_{\geq 2}$. 
\end{theorem}

This document is organized as follows:
\begin{itemize}
\item \emph{Section \ref{preliminaries}:} We introduce the required background, including quantifier elimination statements, the state of the model theory of ordered abelian groups and some results about valued vector spaces. 
\item \emph{Section \ref{modules}:} We study definable $\mathcal{O}$-modules of $K^{n}$. 
\item \emph{Section \ref{stabilizer}: }We start by presenting Hrushovski's criterion to eliminate imaginaries. We introduce the \emph{stabilizer sorts}, where the $\mathcal{O}$-submodules of $K^{n}$ can be coded. 
\item \emph{Section \ref{density}:} We prove that each of the conditions of Hrushovski's criterion hold. This is the density of definable types in definable sets in $1$-variable $X \subseteq K$ and finding canonical basis for definable types in the stabilizer sorts and $\Gamma^{eq}$. We conclude this section proving the weak elimination of imaginaries of any henselian valued field of equicharacteristic zero, residue field algebraically closed and value group with finite spines down to the stabilizer sorts. 
\item \emph{Section \ref{dpminimal}: }We show a complete elimination of imaginaries statement when the value group is $dp$-minimal. We prove that any finite set of tuples in the stabilizer sorts can be coded. 
\end{itemize}
\textbf{Aknowledgements:}  The author would like to thank Pierre Simon and Thomas Scanlon for many insightful mathematical conversations, their time, their constant support and finally for the elegance of their mathematical approach. The author would like to thank particularly Scanlon, for having introduce her to the model theory of valued fields with a huge mathematical generosity and curiosity. The author would like to express her gratitude for the financial support given by the NSF grant 1600441, that allowed her to attend the conference on Model Theory of Valued fields in Paris. \\
The author would like to express her gratitude to the anonymous referee, for a very careful reading and for pointing out a gap in the original draft. The provided list of comments helped significantly to improve the presentation of the paper, and we extend a warm message of gratitude to them.
\newpage
\section{Preliminaries}\label{preliminaries}
\subsection{Quantifier Elimination for valued fields of equicharacteristic zero and residue field algebraically closed }
In this section we recall several results relevant for our statement. In particular we state a quantifier elimination relative to the value group in the canonical three sorted language $\mathcal{L}_{val}$ for the class of valued fields of equicharacteristic zero and residue field algebraically closed. 
\subsubsection{The three-sorted language $\mathcal{L}_{\val}$}
We consider valued fields as  three sorted structures $(K, k, \Gamma)$. The first two sorts are equipped with the language of fields  $\mathcal{L}_{fields}= \{0, 1, +, \cdot, (\cdot)^{-1}, - \}$, we refer to the first one as the \emph{main field sort} while we call the second one as the \emph{residue field sort}. The third sort is supplied with the language of ordered abelian groups $\mathcal{L}_{OAG}=\{ 0, <, +, -\}$, and we refer to it as the \emph{value group sort}. We also add constants $\infty$ to the second sort and the third sort. 
 We introduce a function symbol $v: K \rightarrow \Gamma \cup \{ \infty\}$, interpreted as the valuation and, we add a map $res: K \rightarrow k \cup \{ \infty\}$, where 
 $res: \mathcal{O} \rightarrow k$ is interpreted as a surjective homomorphism of rings, while for any element $x \in K \backslash \mathcal{O}$ we have $\res(x)=\infty$ . We denote this language as $\mathcal{L}_{\val}$. 
\subsubsection{The Extension Theorem}
Let $\mathcal{K}=(K,k,\Gamma)$ be a valued field and $\mathcal{O}$ its valuation ring. A triple $\mathcal{E}=(E,k_{E}, \Gamma_{E})$ is a substructure if $E$ is a subfield of $K$, $k_{E}$ is a subfield of $k$, $\Gamma_{E}$ is a subgroup of $\Gamma$, $v(E^{\times}) \subseteq \Gamma_{E}$ and $\res(O_{E}) \subseteq k_{E}$ where $O_{E}= O \cap E$. 

\begin{definition} Let $\mathcal{K}_{1}=(K_{1},k_{1},\Gamma_{1})$ and  $\mathcal{K}_{2}=(K_{2},k_{2},\Gamma_{2})$ be valued fields of equicharacteristic zero with residue field algebraically closed. \\
Let  $\mathcal{E}=(E,k_{E}, \Gamma_{E})$ be a substructure of $\mathcal{K}_{1}$ a triple $(f,f_{r},f_{v})= \mathcal{E} \rightarrow \mathcal{K}_{2}$ is said to be an \emph{admissible embedding} if it is a $\mathcal{L}_{val}$ isomorphism and $f_{v}= \Gamma_{E} \rightarrow \Gamma_{2}$ is a partial elementary map between $\Gamma_{1}$ and $\Gamma_{2}$, i.e for every $\mathcal{L}_{OAG}$ formula $\phi(x_{1},\dots, x_{n})$ and tuple $e_{1}, \dots, e_{n} \in \Gamma_{E}$ we have that
\begin{align*}
    \Gamma_{1} \vDash \phi(e_{1},\dots,e_{n})& \ \text{if and only if} \ \Gamma_{2} \vDash \phi(f_{v}(e_{1}),\dots, f_{v}(e_{n})). 
\end{align*}

Let $\kappa=\max \{ |k_{E}|, |\Gamma_{E}| \}$. If $\mathcal{K}_{2}$ is $\kappa^{+}$-saturated we say that $(f,f_{r},f_{v})$ is an \emph{admissible map with small domain}.
\end{definition}

\begin{theorem} \label{extension}[The Extension Theorem] The theory of henselian valued fields of equicharacteristic zero with residue field algebraically closed admits quantifier elimination relative to the value group in the language $\mathcal{L}_{val}$. That is given $\mathcal{K}_{1}=(K_{1},k_{1},\Gamma_{1})$ and $\mathcal{K}_{2}=(K_{2},k_{2},\Gamma_{2})$ henselian valued fields of equicharacteristic zero with residue field algebraically closed, a substructure $\mathcal{E}$ of $\mathcal{K}_{1}$ and $(f,f_{r},f_{v}): \mathcal{E} \rightarrow \mathcal{K}_{2}$ and admissible map with small domain, for any $b \in K_{1}$ there is an admissible map $\hat{f}$ extending $f$ whose domain contains $b$.
\end{theorem}
\begin{proof}
This is straightforward using the standard techniques to obtain elimination of field quantifiers already present in the area. We refer the reader for example to \cite[Theorem 5.21]{dries}. The unique step that requires the presence of an angular component map is 
when for a subfield $E\subseteq K_{1}$ we want to add an element $\gamma$ to $v(E^{\times})$ and there is some prime number  $p$ such that $p\gamma \in v(E^{\times})$. For this,  take $a \in E$ and $c \in K_{1}$ be such that $v(a)=p\gamma$ and $v(c)=\gamma$. We first aim to find  $b_{1} \in K_{1}$  that is a root of the polynomial $Q(x)\in O_{E}(x)$, where $Q(x)=x^{p}-\frac{a}{c^{p}}$. Let $d=\res(\frac{a}{c^{p}})$, because $k_{1}$ is algebraically closed there is some $z \in k_{1}$ such that $z^{p}-d=0$. Let $\alpha \in K_{1}$ be such that $\res(\alpha)=z$, then $v(Q(\alpha))>0$ while $v(Q'(\alpha))=0$, because $p \neq char(k_{1})$. Indeed, $Q'(x)= px^{p-1}$ and $z \neq 0$ because $d \neq 0$. \\
 By henselianity we can find $b_{1} \in K_{1}$ that is a root of $Q(x)$. Then $x_{1}= (b_{1} \cdot c)$ is a $p$th-root of $a$. \\

\end{proof}
The following is an immediate consequence of relative quantifier elimination. 
\begin{corollary}\label{stableembeddedness} The residue field and the value group are both purely stably embedded and orthogonal to each other. 
\end{corollary}

\subsection{Some results on the model theory of ordered abelian groups}\label{background}
In this Subsection we summarize many interesting results about the model theory of ordered abelian groups. We start by recalling the following folklore fact.
\begin{fact} Let $(\Gamma,\leq ,+, 0)$ be a non-trivial ordered abelian group. Then the topology induced by the order in $\Gamma$ is discrete if and only if $\Gamma$ has a minimum positive element. In this case we say that $\Gamma$ is \emph{discrete}, otherwise we say that it is \emph{dense}. 
\end{fact}
The following notions were isolated in the sixties by Robinson and Zakon in \cite{RZ} to understand some model complete extensions of the theory of ordered abelian groups.
\begin{definition} Let $\Gamma$ be an ordered abelian group and $n \in \mathbb{N}_{\geq 2}$.
\begin{enumerate}
\item  Let $\gamma \in \Gamma$. We say that $\gamma$ is \emph{$n$-divisible} if there is some $\beta \in \Gamma$ such that $\gamma=n\beta$. 
\item We say that $\Gamma$ is \emph{$n$-divisible} if every element $\gamma \in \Gamma$ is $n$-divisible. 
\item $\Gamma$ is said to be \emph{$n$-regular} if any interval with at least $n$ points contains an $n$-divisible element. 
\end{enumerate}
\end{definition}
\begin{definition}\label{regular} An ordered abelian group $\Gamma$ is said to be \emph{regular} if it is $n$-regular for all $n \in \mathbb{N}$.
\end{definition}

Robinson and Zakon in their seminal paper \cite{RZ} completely characterized the possible completions of the theory of regular groups, obtained by extending the first order theory of ordered abelian groups with axioms asserting that for each $n \in \mathbb{N}$ 
if an interval contains at least $n$-elements then it contains an $n$-divisible element. 
The following is \cite[Theorem 4.7]{RZ}. 
\begin{theorem}\label{charreg}
The possible completions of the theory of regular groups, are: 
\begin{enumerate}
    \item the theory of discrete regular groups, and
    \item  the completions of the theory of dense regular groups $T_{\chi}$ where
    \begin{align*}
      \chi&=:\primes \rightarrow \mathbb{N} \cup \{\infty\},\\
    \end{align*}
   is a function specifying the index $\chi(p)= [\Gamma : p\Gamma]$.
\end{enumerate}
\end{theorem}
Robinson and Zakon proved as well that each of these completions is the theory of some archimedean group. In particular, any discrete regular group is elementarily equivalent to $(\mathbb{Z},\leq, +, 0)$.\\

The following definitions were introduced by Schmitt in \cite{supersch}.
\begin{definition}  We fix an ordered abelian group $\Gamma$ and $n\in \mathbb{N}_{\geq 2}$. Let $\gamma \in \Gamma$. We define:
\begin{itemize}
    \item $A(\gamma)=$ the largest convex subgroup of $\Gamma$ not containing $\gamma$.
    \item $B(\gamma)=$ the smallest convex subgroup of $\Gamma$ containing $\gamma$.
    \item $C(\gamma)=B(\gamma)/A(\gamma)$.
    \item $A_{n}(\gamma)=$ the smallest convex subgroup $C$ of $\Gamma$ such that $B(g)/C$ is $n$-regular. 
    \item $B_{n}(g)=$ the largest convex subgroup $C$ of $\Gamma$ such that $C/A_{n}(\gamma)$ is $n$-regular.
\end{itemize}
\end{definition}
In \cite[Chapter 2]{supersch}, Schmitt shows that the groups $A_{n}(\gamma)$ and $B_{n}(\gamma)$ are definable in the language of ordered abelian groups $\mathcal{L}_{OAG}=\{ +,-, \leq, 0\}$ by a first order formula using only the parameter $\gamma$.\\
We recall that the set of convex subgroups of an ordered abelian group is totally ordered by inclusion.
\begin{definition}
Let $\Gamma$ be an ordered abelian group and $n \in \mathbb{N}_{\geq 2}$, we define the \emph{$n$-regular rank} to be the order type of: \\
\begin{align*}
\big( \{ A_{n}(\gamma) \ | \ \gamma \in \Gamma \backslash \{ 0\}\}, \subseteq \big).
    \end{align*}
\end{definition}
The $n$-regular rank of an ordered abelian group $\Gamma$ is a linear order, and when it is finite we can identify it with its cardinal. In \cite{Farre}, Farr\'e emphasizes that we can characterize the $n$-regular rank without mentioning the subgroups $A_{n}(\gamma)$. The following is \cite[Remark 2.2]{Farre}. 
\begin{definition}
Let $\Gamma$ be an ordered abelian group and $n \in \mathbb{N}_{\geq 2}$, then: 
\begin{enumerate}
\item $\Gamma$ has $n$-regular rank equal to $0$ if and only if $\Gamma=\{0\}$,
\item $\Gamma$ has $n$-regular rank equal to $1$ if and only if $\Gamma$ is $n$-regular and not trivial, 
\item $\Gamma$ has $n$-regular rank equal to $m$ if there are $\Delta_{0}, \dots, \Delta_{m}$ convex subgroups of $\Gamma$, such that:
\begin{itemize}
\item $\{0\}=\Delta_{0} < \Delta_{1} < \dots < \Delta_{m}=\Gamma$,
\item for each $0 \leq i< m$, the quotient group $\Delta_{i+1}/\Delta_{i}$ is $n$-regular,
 \item the quotient group $\Delta_{i+1}/\Delta_{i}$ is not $n$-divisible for $0 < i < m$. 
\end{itemize}
In this case we define $RJ_{n}(\Gamma)=\{ \Delta_{0}, \dots, \Delta_{m-1}\}$. The elements of this set are called the \emph{$n$-regular jumps}.
\end{enumerate}
\end{definition}
\begin{definition}\label{polyreg} Let $\Gamma$ be an ordered abelian group. We say that it is \emph{poly-regular} if it is elementarily equivalent to a subgroup of the lexicographically ordered group $(\mathbb{R}^{n},+, \leq_{lex},0)$.
\end{definition}
In \cite{Belegradek} Belegradek studied poly-regular groups and proved that an ordered abelian group is poly-regular if and only if it has finitely many proper definable convex subgroups, and all the proper definable subgroups are definable over the empty set. In \cite[Theorem 2.9]{poly} Weispfenning obtained quantifier elimination for the class of poly-regular groups in the language of ordered abelian groups extended with predicates to distinguish the subgroups $\Delta+\ell \Gamma$ where $\Delta$ is a convex subgroup and $\ell \in \mathbb{N}_{\geq 2}$.
\begin{definition}
Let $\Gamma$ be an ordered abelian group. We say that it has \emph{bounded regular rank} if it has finite $n$-regular rank for each $n \in \mathbb{N}_{\geq 2}$.  For notation, we will use $\displaystyle{RJ(\Gamma)= \bigcup_{n \in \mathbb{N}_{\geq 2}} RJ_{n}(\Gamma)}$. 
\end{definition}

The class of ordered abelian groups of bounded regular rank  extends the class of poly-regular groups and regular groups. The terminology of \emph{bounded regular rank} becomes clear with the following Proposition (item $3$). 
\begin{proposition}\label{charbounded} Let $\Gamma$ be an ordered abelian group. The following are all equivalent: 
\begin{enumerate}
\item $\Gamma$ has finite $p$-regular rank for each prime number $p$.
\item $\Gamma$ has finite $n$-regular rank for each $n \geq 2$.
\item There is some cardinal $\kappa$ such that for any $H \equiv \Gamma$, $|RJ(H)|\leq \kappa$.
\item For any $H \equiv \Gamma$, any definable convex subgroup of $H$  has a definition without parameters.
\item There is some cardinal $\kappa$ such that for any $H \equiv \Gamma$, $H$ has at most $\kappa$  definable convex subgroups. 
\end{enumerate}
Moreover, in this case $RJ(\Gamma)$ is the collection of all proper definable convex subgroups of $\Gamma$ and all are definable without parameters. In particular, there are only countably many definable convex subgroups. 
\end{proposition}
\begin{proof}
This is \cite[Proposition 2.3]{Farre}. 
\end{proof}
The first results about the model completions of ordered abelian groups appear in \cite{RZ} (1960), where the notion of $n$-regularity was isolated.  
\begin{definition}
\begin{enumerate}
\item Let $\Gamma$ be an ordered abelian group and $\gamma \in \Gamma$, we say that $\gamma$ is $n$-divisible if there is some $\beta \in \Gamma$ such that $\gamma=n\beta$. 
\item  Let $n \in \mathbb{N}_{\geq 2}$. An ordered abelian group $\Gamma$ is said to be \emph{$n$-regular} if any interval with at least $n$-points contains an $n$-divisible element. 
\item Let $\Gamma$ be an ordered abelian group, we say that it is \emph{regular} if it is $n$-regular for all $n \in \mathbb{N}_{\geq 2}$. 
\end{enumerate}
\end{definition}

\subsubsection{Quantifier elimination and the quotient sorts}\label{OAGBR}
 In \cite{Immi-Cluckers} Cluckers and Halupczok introduced a language $\mathcal{L}_{qe}$  to obtain  quantifier elimination for ordered abelian groups relative to the \emph{auxiliary sorts} $S_{n}$, $T_{n}$ and $T_{n}^{+}$, whose precise description can be found in \cite[Definition 1.5]{Immi-Cluckers}. This language is similar in spirit to the one introduced by Schmitt in  \cite{supersch}, but has lately been preferred by the community as it is more in line with the many-sorted language of Shelah's imaginary expansion $\mathfrak{M}^{eq}$. Schmitt does not distinguish between the sorts $S_{n}$, $T_{n}$ and $T_{n}^{+}$. Instead for each $n \in \mathbb{N}$ he works with a single sort $Sp_{n}(\Gamma) $ called the \emph{$n$-spine} of $\Gamma$, whose description can be found in  \cite[Section 2]{gurevich}. In  \cite[Section 1.5]{Immi-Cluckers} it is explained how the auxiliary sorts of Cluckers and Halupczok are related to the $n$-spines $Sp_{n}(\Gamma)$ of Schmitt. In \cite[Section 2]{Farre}, it is shown that an ordered abelian group $\Gamma$ has bounded regular rank if and only if all the $n$-spines are finite, and $Sp_{n}(\Gamma)=RJ_{n}(\Gamma)$. In this case, we define the regular rank of $\Gamma$ as the cardinal $|RJ(\Gamma)|$, which is either finite or $\aleph_{0}$. Instead of saying that $\Gamma$ is an ordered abelian group with finite spines, we prefer to use the classical terminology of bounded regular rank, as it emphasizes the relevance of the $n$-regular jumps and the role of the divisibilities to describe the definable convex subgroups.\\

We define the \emph{Presburger Language}  $\mathcal{L}_{\Pres}=\{ 0,1,+,  -, <, (P_{m})_{m \in \mathbb{N}_{\geq 2}}\}$. Given an ordered abelian group $\Gamma$ we naturally see it as a $\mathcal{L}_{\Pres}$-structure. The symbols $\{0,+, - ,<\}$ take their obvious interpretation. If $\Gamma$ is discrete, the constant symbol $1$ is interpreted as the least positive element of $\Gamma$,  and by $0$ otherwise. For each $m \in \mathbb{N}_{\geq 2}$ the symbol $P_{m}$ is a unary predicate interpreted as $m\Gamma$.

\begin{definition}\label{lanbounded} [The language $\mathcal{L}_{b}$]
Let $\Gamma$ be an ordered abelian group with bounded regular rank, we view $\Gamma$ as a multi-sorted structure where: 
\begin{enumerate}
\item We add a sort for the ordered abelian group $\Gamma$, and we equip it with a copy of the language $\mathcal{L}_{\Pres}$ extended with predicates to distinguish each of the convex subgroups $\Delta \in RJ(\Gamma)$. We refer to this sort as the \emph{main sort}. 
\item We add a sort for each of the ordered abelian groups $\Gamma/\Delta$, equipped with a copy of the language $\mathcal{L}_{\Pres}^{\Delta}=\{ 0^{\Delta}, 1^{\Delta}, +^{\Delta}, -^{\Delta}, <^{\Delta}, (P_{m}^{\Delta})_{m \in \mathbb{N}_{\geq 2}}\}$. We add as well a map $\rho_{\Delta}: \Gamma \rightarrow \Gamma/\Delta$, interpreted as the natural projection map. 
\end{enumerate}
\end{definition}
\begin{remark} To keep the notation as simple and clear as possible, for each $\Delta \in RJ(\Gamma)$ and $n\in \mathbb{N}_{\geq 2}$ and $\beta \in \Gamma/\Delta$ we will write $\beta \in n(\Gamma/\Delta)$ instead of $P_{n}^{\Delta}(\beta)$.
\end{remark}
The following statement is a direct consequence of  \cite[Proposition 3.14]{distal}.
\begin{theorem} \label{QEbounded} Let $\Gamma$ be an ordered abelian group with bounded regular rank. Then $\Gamma$ admits quantifier elimination in the language $\mathcal{L}_{b}$. 
\end{theorem}
We will consider an extension of this language that we will denote as  $\mathcal{L}_{bq}$, where for each natural number $n \geq 2$ and $\Delta \in RJ(\Gamma)$ we add a sort for the quotient group $\Gamma/(\Delta+n\Gamma)$ and a map $\pi_{\Delta}^{n}: \Gamma \rightarrow \Gamma/(\Delta+n\Gamma)$.  We will refer to the sorts in the language $\mathcal{L}_{bq}$ as \emph{quotient sorts}. \\
The following is  \cite[Theorem 5.1]{Vicaria}. 
\begin{theorem}\label{WEIbounded} Let $\Gamma$ be an ordered abelian group with bounded regular rank. Then $\Gamma$ admits weak elimination of imaginaries in the language $\mathcal{L}_{bq}$, i.e. once one adds all the quotient sorts.
\end{theorem}
\paragraph{Definable end-segments in ordered abelian groups with bounded regular rank}
\begin{definition}
\begin{enumerate}
\item A non-empty set $S \subset \Gamma$ is said to be an end-segment if for any $x \in S$ and $y \in \Gamma$, $x < y$ we have that $y \in S$. 
\item Let $n \in \mathbb{N}$, $\Delta \in RJ(\Gamma)$, $\beta \in \Gamma \cup \{ -\infty\}$ and $\square \in \{ \geq, >\} $. The set: 
\begin{align*}
S_{n}^{\Delta}(\beta):=\{ \eta \in \Gamma \ | \ n\eta + \Delta \square  \beta +\Delta\} 
\end{align*}
is an \emph{end-segment} of $\Gamma$. We call any of the end-segments of this form as \emph{divisibility end-segments}.
\item Let $S \subseteq \Gamma$ be a definable end-segment and $\Delta \in RJ(\Gamma)$. 
We consider the projection map $\rho_{\Delta}: \Gamma \rightarrow \Gamma/\Delta$, 
and we write $S_{\Delta}$ to denote $\rho_{\Delta}(S)$. This is a definable end-segment of $\Gamma/\Delta$.
\item Let $\Delta \in RJ(\Gamma)$ and $S \subseteq \Gamma$ an end-segment. We say that $S$ is \emph{$\Delta$-decomposable} if it is a union of $\Delta$-cosets. 
\item We denote as $\Delta_{S}$ the \emph{stabilizer of $S$}, i.e. $\displaystyle{\Delta_{S}:= \{ \eta \in \Gamma \ | \ \eta+S=S \}.}$ 
\end{enumerate}
\end{definition}

\begin{definition} Let $\Gamma$ be an ordered abelian group.  Let  $S, S' \subseteq \Gamma$ be definable end-segments. We say that $S$ is a \emph{translate} of $S'$ if  there some $\beta \in \Gamma$  such that $S=\beta+S'$. Given a family $\mathcal{S}$ of definable end-segments we say that $\mathcal{S}$ is \emph{complete} if every definable end-segment is a translate of some $S' \in \mathcal{S}$. 
\end{definition}
\begin{fact}\label{Mari} Let $\Gamma$ be an ordered abelian group with bounded regular rank. Let $\beta, \gamma \in \Gamma$, $\Delta \in RJ(\Gamma)$ and $n \in \mathbb{N}_{\geq 2}$. If $\beta - \gamma \in \Delta+ n\Gamma$ then $S_{n}^{\Delta}(\gamma)$ is a translate of $S_{n}^{\Delta}(\beta)$. 
\end{fact}
The following is \cite[Proposition 3.3]{Vicaria}. 
\begin{proposition}\label{charend} Let $\Gamma$ be an ordered abelian group of bounded regular rank. Any definable end-segment is a divisibility end-segment. 
\end{proposition}
\begin{remark}\label{suff} Let $\Gamma$ be an ordered abelian group and $\Delta$ be a convex subgroup. Any complete set of representatives in $\Gamma$ modulo $k\Gamma$ for $k\in \mathbb{N}$ is also a complete set of representative of $\Gamma$ modulo $\Delta+k\Gamma$. Moreover, there is and $\emptyset$-definable surjective function $f: \Gamma/ k\Gamma \rightarrow \Gamma/(\Delta+k\Gamma)$. 
\end{remark}
\begin{proof}
For the first part of the statement take $\gamma, \beta \in \Gamma$, if $\gamma-\beta \in k\Gamma$ then $\gamma-\beta \in \Delta+k\Gamma$. For the second part, consider the $\emptyset$-definable function:
\begin{align*}
    f: \Gamma/k\Gamma &\rightarrow \Gamma/ (\Delta+k\Gamma)\\
    \gamma+k\Gamma &\rightarrow \gamma+(\Delta+k\Gamma).
\end{align*}
This function is surjective by the first part of the statement. 
\end{proof}
\begin{corollary} \label{complete} Let $\Gamma$ be an ordered abelian group with bounded regular rank. For each $n \in \mathbb{N}_{\geq 2}$ let $\mathcal{C}_{n}$ be a complete set of representatives of the cosets $n\Gamma$ in $\Gamma$. 
Define $\mathcal{S}_{n}^{\Delta}:= \{ S_{n}^{\Delta}(\beta) \ | \ \beta \in \mathcal{C}_{n}\}$. Then $\displaystyle{\mathcal{S}=\bigcup_{\Delta \in RJ(\Gamma), n \in \mathbb{N}_{\geq 2}} \mathcal{S}_{n}^{\Delta}}$ is a complete family.  
\end{corollary}
\begin{proof}
It is an immediate consequence of Proposition \ref{charend}, Fact \ref{Mari} and Remark \ref{suff}.
\end{proof}
The following is \cite[Fact 4.1]{Vicaria}.
\begin{fact}\label{stab1}Let $S \subseteq \Gamma$ be a definable end-segment. Then $\Delta_{S}$ is a definable convex subgroup of $\Gamma$, therefore $\Delta_{S} \in RJ(\Gamma)$. Furthermore, $\displaystyle{\Delta_{S}= \bigcup_{\Delta \in \mathcal{C}} \Delta}$, where $\mathcal{C}=\{ \Delta \in RJ(\Gamma) \ | \ S \text{\ is $\Delta$-decomposable} \}$. 
\end{fact}
\begin{definition} Let $S \subseteq \Gamma$ be a definable end-segment. Let 
\begin{align*}
\Sigma_{S}^{gen}(x):&= \{ x \in S \} \cup \{ x \notin B \ | \ B \subsetneq S\  \text{and $B$ is a definable end-segment }\}.
\end{align*}
 We refer to this partial type as the \emph{generic type in $S$}. This partial type is $\ulcorner S \urcorner$-definable.
\end{definition}

\subsubsection{The dp-minimal case}
In $1984$ the classification of the model theoretic complexity of ordered abelian groups was initiated by Gurevich and Schmitt, who proved that no ordered abelian group has the independence property. During the last years finer classifications have been achieved, in particular dp-minimal ordered abelian groups have been characterized in \cite{dpminimal}. 
\begin{definition} Let $\Gamma$ be an ordered abelian group and let $p$ be a prime number. We say that $p$ is a \emph{singular prime} if $[\Gamma: p\Gamma]= \infty$.
\end{definition}
The following result corresponds to \cite[Proposition 5.1]{dpminimal}.
\begin{proposition}\label{dpmin} Let $\Gamma$ be an ordered abelian group, the following conditions are equivalent:
\begin{enumerate}
\item $\Gamma$ does not have singular primes,
\item $\Gamma$ is $dp$-minimal.
\end{enumerate}
\end{proposition}
\begin{definition}\label{landp}[The language $\mathcal{L}_{dp}$]  Let $\Gamma$ be a dp-minimal ordered abelian group. We consider the language extension $\mathcal{L}_{dp}$ of  $\mathcal{L}_{bq}$ [see Definition \ref{lanbounded}] where for each $n \in \mathbb{N}_{\geq 2}$ we add a set of constant for the elements of the finite group $\Gamma/n\Gamma$.
\end{definition}

The following is  \cite[Corollary 5.2]{Vicaria}. 
\begin{corollary}\label{EIdpmin} Let $\Gamma$ be a $dp$-minimal ordered abelian group. Then $\Gamma$ admits elimination of imaginaries in the language $\mathcal{L}_{dp}$.
\end{corollary}
The following will be a very useful fact.
\begin{fact}\label{genericodp} Let $\Gamma$ be a dp-minimal ordered abelian group and let $S \subseteq \Gamma$ be a definable end-segment. Then any complete type $q(x)$ extending $\Sigma_{S}^{gen}(x)$ is $\ulcorner S \urcorner$-definable.
\end{fact}
\begin{proof}
Let $\Sigma_{S}^{gen}(x)$ be the generic type of $S$ and $q(x)$ be any complete extension. $\Sigma_{S}^{gen}(x)$ is $\ulcorner S\urcorner$-definable, and by Theorem  \ref{QEbounded} $q(x)$ is completely determined by the quantifier free formulas. It is sufficient to verify that for each $\Delta \in RJ(\Gamma)$, $k\in \mathbb{Z}$ and $n \in \mathbb{N}$ the set:\\
\begin{align*}
   Z= \{ \beta \in \Gamma \ | \ \big(\rho_{\Delta}(x)-\rho_{\Delta}(\beta)+k^{\Delta}\in n (\Gamma/\Delta)\big) \in q(x)\}
\end{align*}
is $\ulcorner S \urcorner$-definable. First, we note that there is a canonical one to one correspondence
\begin{align*}
    g:&= (\Gamma/\Delta)/ n\big(\Gamma/\Delta\big) \rightarrow \Gamma/(\Delta+n\Gamma).
\end{align*}
Let $c= g(k^{\Delta}+ n(\Gamma/\Delta)) \in \dcl^{eq}(\emptyset)$. Take $\mu \in \Gamma/n\Gamma$ be such that $\pi_{k}(x)=\mu \in q(x)$. Let $f$ be the $\emptyset$-definable function given by Remark \ref{suff}.
 Then $\beta \in Z$ if and only if $\displaystyle{\vDash \pi_{\Delta}^{n}(\beta)= f(\mu)+c}$, and $f(\mu)+c \in \dcl^{eq}(\emptyset)$.
\end{proof}
We conclude this subsection with the following Remark, that simplifies the presentation of a complete family in the dp-minimal case. 
\begin{remark} \label{completedp} Let $\Gamma$ be a dp-minimal ordered abelian group. For each $n \in \mathbb{N}_{\geq 2}$ let $\Omega_{n}$ be a finite set of constants in $\Gamma$ to distinguish representatives for each of the cosets of $n\Gamma$ in $\Gamma$.\\
Let $\mathcal{S}_{n}^{\Delta}:= \{ S_{n}^{\Delta}(d) \ | \ d \in \Omega^{n}\}$. The set $\displaystyle{\mathcal{S}_{dp}=\bigcup_{\Delta \in RJ(\Gamma), n \in \mathbb{N}_{\geq 2}} \mathcal{S}_{n}^{\Delta}}$ is a complete family whose elements are all definable over $\emptyset$. 
\end{remark}
%The following is Theorem $3.13$ in \cite{distal}, 
%\begin{proposition} Let $\Gamma$ be an ordered abelian group with finite spines (i.e. each $Sp_{n}$ is finite). The following statements are equivalent:
%\begin{enumerate}
%\item $\Gamma$ is distal,
%\item $\Gamma$ is $dp$-minimal. 
%\end{enumerate}
%\end{proposition}
%The following statement was independently achieved in \cite{Farre}, \cite{Dolich-Goodrick} and \cite{Halevi-Hasson}. 
%\begin{proposition} Let $\Gamma$ be an ordered abelian group, the following conditions are all equivalent:
%\begin{enumerate}
%\item $\Gamma$ is strongly dependent,
%\item $\Gamma$ has finite $dp$-rank,
%\item $\Gamma$ has bounded regular rank and finitely many singular primes. 
%\end{enumerate}
%Moreover,  let $\mathcal{P}= \{ p \in \mathbb{N} \ | \ p$ is a singular prime $\}$. Then
%\begin{align*}
%dp-rank(\Gamma) \leq 1+ \sum_{p \in \mathcal{P}} |RJ_{p}(G)|. 
%\end{align*}
%\end{proposition}

\subsection{Henselian valued fields of equicharacteristic zero with residue field algebraically closed and value group with bounded regular rank}
The main goal of this section is to describe the $1$-definable subsets $X \subseteq K$, where $K$ is a valued field with residue field algebraically closed and with value group of bounded regular rank.  
\subsubsection{The language $\mathcal{L}$}\label{language}
Let $(K,v)$ be a valued field of equicharactieristic zero, whose residue field is algebraically closed and whose value group is of bounded regular rank. We will view this valued field as an $\mathcal{L}$-structure, where $\mathcal{L}$ is the language extending $ \mathcal{L}_{\val}$ in which the value group sort is equipped with the language $\mathcal{L}_{b}$ described in Definition \ref{lanbounded}. Let $T$ be the complete $\mathcal{L}$-first ordered theory of $(K,v)$. (In particular, we are fixing a complete theory for the value group) 
\begin{corollary}\label{QEregular}  The first order theory $T$ admits quantifier elimination in the language $\mathcal{L}$. 
\end{corollary}
\begin{proof}
This is a direct consequence of Theorem \ref{extension} and Theorem \ref{QEbounded}.
\end{proof}
\subsubsection{Description of definable sets in $1$-variable}
In this Subsection we give a description of the definable subsets in $1$-variable $X \subseteq K$, where $K \vDash T$. We denote as $\mathcal{O}$ its valuation ring. 
\begin{definition} Let $(K,\mathcal{O})$ be a henselian valued field of equicharacteristic zero and let $\Gamma$ be its value group. Let $\Delta$ be a convex subgroup of $\Gamma$ then the map:
\begin{center}$v_{\Delta}: \begin{cases}
K &\rightarrow \Gamma/\Delta \\
x & \mapsto v(x)+\Delta,
\end{cases}$\end{center}
is a henselian valuation on $K$ and it is commonly called as \emph{the coarsened valuation induced by $\Delta$}. Note that $v_{\Delta}= \rho_{\Delta} \circ v$. 
\end{definition}
The following is a folklore fact. 
\begin{fact}\label{endsegement}There is a one-to-one correspondence between the $\mathcal{O}$-submodules of $K$ and the end-segments of $\Gamma$. Given $M \subseteq K$ an $\mathcal{O}$-submodule, we have that $S_{M}:= \{ v(x) \ | \ x \in M\}$ is an end-segment of $\Gamma$. We refer to $S_{M}$ as the end-segment induced by $M$. And given an end-segment $S \subseteq \Gamma$, the set $M_{S}:= \{ x \in K \ | \ v(x) \in S\}$ is an $\mathcal{O}$-submodule of $K$.
\end{fact}
\begin{definition}
 \begin{enumerate}
 \item Let $M$ and $N$ be $\mathcal{O}$-submodules of $K$, we say that $M$ is a \emph{scaling} of $N$ if there is some $b \in K$ such that $M= bN$. 
\item A family $\mathcal{F}$ of definable $\mathcal{O}$-submodules of $K$ is said to be \emph{complete} if any definable submodule $M \subseteq K$ is a scaling of some $\mathcal{O}$-submodule $N \in \mathcal{F}$. 
\end{enumerate}
\end{definition}
\begin{fact}\label{completemodules} Let $\mathcal{F}= \{ M_{S} \ | \ S \in \mathcal{S}\}$, where $\mathcal{S}$ is the complete family of definable end-segments described in Corollary \ref{complete}. Then $\mathcal{F}$ is a complete family of $\mathcal{O}$-submodules of $K$. 
\end{fact}
\begin{definition}
\begin{enumerate}
\item Let $w: K \rightarrow \Gamma_{w}$ be a valuation, $\gamma \in \Gamma_{w}$ and $a \in K$. The \emph{closed ball of radius $\gamma$ centered at $a$ according to the valuation $w$} is the set of the form $\bar{B}_{\gamma}(a)=\{ x \in K \ | \ \gamma \leq w(x-a) \}$, and the \emph{open ball of radius $\gamma$ centered at $a$ according to the valuation $w$} is the set of the form $\displaystyle{B_{\gamma}(a)=\{x \in K \ | \ \gamma< w(x-a)\}}$. 
\item  A \emph{swiss cheese according to the valuation $w$} is a set of the form $\displaystyle{A \backslash (B_{1}\cup \dots \cup B_{n})}$ where for each $i \leq n$, $B_{i} \subsetneq A$ and the $B_{i}$ and $A$ are balls according to the original valuation $w: K \rightarrow \Gamma_{w}$. 
\item A $1$-torsor of $K$ is a set of the form $a+bI$ where $a,b \in K$ and $I \in \mathcal{F}$.

\item A \emph{generalized swiss cheese} is either a singleton element in the field $\{a\}$ or a set of the form   $\displaystyle{A \backslash (B_{1}\cup \dots \cup B_{n})}$ where $A$ is a $1$-torsor for each $i \leq n$, $B_{i} \subsetneq A$ and the $B_{i}$ is either a $1$-torsor or a singleton element $\{b_{i}\}$ of the field. 
\item A \emph{basic positive congruence formula in the valued field} is a formula of the form $\displaystyle{zv_{\Delta}(x- \alpha)- \rho_{\Delta} (\beta) + k_{\Delta} \in n (\Gamma/\Delta)}$, where $k, z \in \mathbb{Z}$, $\alpha \in K$, $\beta \in \Gamma$, $n \in \mathbb{N}_{\geq 2}$ and $k_{\Delta}=k \cdot 1^{\Delta}$, where $1^{\Delta}$ is the minimum positive element of $\Gamma/\Delta$ if it exists. 
\item A \emph{basic negative congruence formula in the valued field} is a formula of the form $\displaystyle{zv_{\Delta}(x- \alpha)- \rho_{\Delta} (\beta) + k_{\Delta} \notin n(\Gamma/\Delta)}$, where $k, z \in \mathbb{Z}$, $\alpha \in K$, $\beta \in \Gamma$, $n \in \mathbb{N}_{\geq 2}$ and $k_{\Delta}=k \cdot 1^{\Delta}$, where $1^{\Delta}$ is the minimum positive element of $\Gamma/\Delta$ if it exists. 
\item A \emph{basic congruence formula in the valued field} is either a \emph{basic positive congruence formula in the valued fied} or a  \emph{basic negative congruence formula in the valued field}.
\item A \emph{finite congruence restriction in the valued field} is a finite conjunction of basic congruence formulas in the valued field. 
\item A \emph{nice set} is a set of the form $S \cap C$ where $S$ is a generalized swiss cheese and $C$ is the set defined by a finite congruence restriction in the valued field. 
\end{enumerate}
\end{definition}

To describe completely the definable subsets of $K$ we will need the following lemmas, which permit us to reduce the valuation of a polynomial into the valuation of linear factor of the form $v(x-a)$. We recall a definition and some results present in \cite{Flenner} that will be useful for this purpose. 
\begin{definition} Let $(K,w)$ be a henselian valued field,  $\alpha \in K$ and $S$ a swiss cheese. Let $p(x) \in K[x]$, we define: 
\begin{align*}
m(p,\alpha, S):= \max \{ i \leq d \ | \ \exists x \in S \ \forall j \leq d \ \big( w(a_{i}(x-\alpha)^{i}) \leq w\big( a_{j} (x-\alpha)^{j} \big) \},
\end{align*}
where the $a_{i}$ are the coefficients of the expansion of $p$ around $\alpha$, i.e. $\displaystyle{p(x)= \sum_{i \leq d} a_{i}(x-\alpha)^{i}}$.\\
Thus $m(p,\alpha, S)$ is the highest order term in $p$ centered at $\alpha$ which can have minimal valuation (among the other terms of $p$) in $S$. 
\end{definition}
The following is \cite[Proposition 3.4]{Flenner}. 
\begin{proposition} \label{decompositionFlenner}Let $\mathcal{K}$ be a valued field of characteristic zero. Let $p(x) \in K[x]$ and $S$ be a \emph{swiss cheese} in $K$. Then there are (disjoint) sub-swiss cheeses $T_{1}, \dots, T_{n} \subseteq S$ and $\alpha_{1}, \dots, \alpha_{n} \in K$ such that $\displaystyle{S = \bigcup_{1\leq i \leq n} T_{i}}$, where for all $x \in T_{i}$ 
$\displaystyle{w(p(x))= w \big( a_{im_{i}} (x-\alpha_{i})^{m_{i}}\big)}$, where $\displaystyle{p(x)= \sum_{n=0}^{d} a_{in}(x-\alpha_{i})^{n}}$ and $m_{i}= m(p,\alpha_{i}, T_{i})$. Furthermore, $\alpha_{1},\dots, \alpha_{k}$ can be taken algebraic over the subfield of $K$ generated by the coefficients of $p(x)$. 
\end{proposition}
Though the preceding proposition is stated for a single polynomial, the same result will hold for any finite number of polynomials $\Sigma$. To obtain the desired decomposition, simply apply the proposition to each $p(x) \in \Sigma$, then intersect the resulting partitions to get one that works for all $p(x) \in \Sigma$, using the fact that intersection of two swiss cheeses is again a swiss cheese. 

\begin{fact}\label{difpol}Let $(K,w)$ be a henselian valued field of equicharacteristic zero, and $Q_{1}(x), Q_{2}(x) \in K[x]$ be two polynomials in a single variable. Let $R=\{ x \in K \ | \ Q_{2}(x)=0\}$. There is a finite union of swiss cheeses $\displaystyle{K=\bigcup_{i \leq k}T_{i}}$, coefficients $\epsilon_{i} \in K$, elements $\gamma_{i} \in \Gamma$ and integers $z_{i} \in \mathbb{Z}$ such that for any $x \in T_{i} \backslash R$:
\begin{align*}
w(Q_{1}(x))-w(Q_{2}(x))=\gamma_{i}+ z_{i}w(x-\epsilon_{i}).
\end{align*}
\end{fact}
\begin{proof}
The statement is a straightforward computation after applying Proposition \ref{decompositionFlenner}, and it is left to the reader. 
\end{proof}

\begin{proposition}\label{des} Let $K \vDash T$, for each $\Delta \in RJ(\Gamma)$ let $v_{\Delta}: K \rightarrow \Gamma/\Delta$ be the coarsened valuation induced by $\Delta$. Let  $Q_{1}(x),Q_{2}(x) \in K[x]$ and $R=\{ x \in K \ | \ Q_{1}(x)=0$ or $Q_{2}(x)=0\}$. 
Let $X \subseteq K \backslash R$ be the set defined by a formula of the form:
\begin{align*}
\gamma \leq^{\Delta} v_{\Delta}(Q_{1}(x))- v_{\Delta}(Q_{2}(x)) \ \text{or} \ v_{\Delta}\left(\frac{Q_{1}(x)}{Q_{2}(x)}\right)- \gamma \in n \left( \Gamma/ \Delta \right);
\end{align*}
 where $\gamma \in \Gamma/\Delta$ and $n \in \mathbb{N}$.
 Then $X$ is a finite union of nice sets.
\end{proposition}
\begin{proof}
First we observe that a swiss cheese with respect to the coarsened valuation $v_{\Delta}$ is a generalized swiss cheese with respect to $v$. The statement follows by a  straightforward computation after applying Fact \ref{difpol}, and it is left to the reader. 
\end{proof}
We conclude this section by characterizing the definable sets in $1$-variable.
\begin{theorem}\label{1def} Let $K \vDash T$ and $X \subseteq K$ be a definable set. Then $X$ is a finite union of nice sets. 
\end{theorem}
\begin{proof}
By Corollary \ref{QEregular} , $X$ is a boolean combination of sets defined by formulas of the form\\
$\displaystyle{\gamma \leq^{\Delta} v_{\Delta}(Q_{1}(x))- v_{\Delta}(Q_{2}(x))}$ or
$\displaystyle{v_{\Delta}\left(\frac{Q_{1}(x)}{Q_{2}(x)}\right)- \gamma \in n \left( \Gamma/ \Delta \right)}$,
where $\Delta \in RJ(\Gamma)$, $\gamma \in \Gamma/\Delta$ and $n \in \mathbb{N}_{\geq 2}$. By Proposition \ref{des} each of these formulas defines a finite union of nice sets. Because the intersection of two generalized swiss cheeses is again a generalized swiss cheese and the complement of a generalized swiss cheese is a finite union of generalized swiss cheeses the statement follows. 
\end{proof}

\subsection{$\mathcal{O}$-modules and homomorphisms in maximal valued fields}
In this section we recall some results about modules over maximally complete valued fields. We follow ideas of Kaplansky in \cite{Kaplansky} to characterize the $\mathcal{O}$-submodules of finite dimensional $K$-vector spaces. 
\begin{definition} 
\begin{enumerate}
\item Let $K$ be a valued field and $\mathcal{O}$ its valuation ring. We say that $K$ is \emph{maximal}, if whenever $\alpha_{r} \in K$ and (integral or fractional) ideals $I_{r}$ are such that the congruences $x-\alpha_{r} \in I_{r}$ are pairwise consistent, then there exists in $K$ a simultaneous solution of all the congruences. 
\item Let $K$ be a valued field and $M \subseteq K^{n}$ be an $\mathcal{O}$-module. We say that $M$ is \emph{maximal} if whenever ideals $I_{r} \subseteq \mathcal{O}$ and elements $s_{r} \in M$ are such that $x-s_{r} \in I_{r}M$ is pairwise consistent in $M$, then there exists in $M$ a simultaneous solution of all the congruences. 
  \item Let $N\subseteq K^{n}$ be an $\mathcal{O}$-submodule. Let $x \in N$ we say that $x$ is \emph{$\alpha$-divisible in $N$} if there is some $n \in N$ such that $x=\alpha n$. 
\end{enumerate}
\end{definition}
We start by recalling a very useful fact.
\begin{fact}\label{maxext} Let $K$ be a henselian valued field of equicharacteristic zero, then there is an elementary extension $K \prec K'$ that is maximal. 
\end{fact}
\begin{proof}
Let $K$ be a henselian valued field of equicharacteristic zero, let $T$ be its $\mathcal{L}_{\val}$-complete first order theory and $\mathfrak{C}$ the monster model of $T$. By \cite[Lemma 4.30]{dries} there is some maximal immediate extension of $K \subseteq F \subseteq \mathfrak{C}$. By \cite[Theorem 7.12]{dries} $K \prec F$. 
\end{proof}
The following is \cite[Lemma 5]{Kaplansky}.
\begin{lemma}\label{product} Let $K$ be a maximal valued field, then any (integral or fractional) ideal $I$ of $\mathcal{O}$ is maximal as an $\mathcal{O}$-submodule of $K$. Moreover, any finite direct sum of maximal $\mathcal{O}$-modules is also maximal. 
\end{lemma}
\begin{fact}\label{1dimmod} Let $N \subseteq K$ be a non-trivial $\mathcal{O}$-submodule. Let $n \in N\backslash \{0\}$ then $N=n I$ where $I$ is a copy of $K$, $\mathcal{O}$ or an (integral or fractional) ideal of $\mathcal{O}$. 
\end{fact}
\begin{definition} Let $K$ be a field and $n \in \mathbb{N}_{\geq 1}$, we say that a set $\{ a_{1}, \dots, a_{n}\}$ is an \emph{upper triangular basis} of the vector space $K^{n}$ if it is a $K$-linearly independent set and the matrix $[a_{1},\dots,a_{n}]$ is upper triangular. 
\end{definition}
  \begin{theorem} \label{characterizationmodules}
  Let $K$ be a maximal valued field and $n \in \mathbb{N}_{\geq 1}$. Let $N \subseteq K^{n}$ be an $\mathcal{O}$-submodule then $N$ is maximal, and $N$ is definably isomorphic to a direct sum of copies of $K$, $\mathcal{O}$ and (integral or fractional)  ideals of $\mathcal{O}$. 
 Moreover, if $\displaystyle{N \cong \bigoplus_{i \leq n} I_{i}}$ where each $I_{i}$ is either a copy of $K$, $\mathcal{O}$ and (integral or fractional)  ideals of $\mathcal{O}$ one can find an upper triangular basis $\{a_{1}, \dots, a_{n}\}$ of $K^{n}$ such that $\displaystyle{N= \{a_{1}x_{1}+ \dots+a_{n}x_{n}\ | \ x_{i}\in I_{i} \}}$. In this case we say that $[a_{1}, \dots, a_{n}]$ is a \emph{representation matrix for the module $N$}. \end{theorem}
  \begin{proof}
 We proceed by induction on $n$, the base case is given by Fact \ref{1dimmod} and Lemma \ref{product}. For the inductive step, let $\pi: K^{n+1} \rightarrow K$ be the projection into the last coordinate and let $M= \pi(N)$. We consider the exact sequence of $\mathcal{O}$-modules
$\displaystyle{0 \rightarrow N \cap \big(K^{n} \times \{ 0 \}\big) \rightarrow N \rightarrow M \rightarrow 0}.$\\
By induction, $N \cap (K^{n} \times \{ 0 \})$ is maximal and of the required form. And there is an upper triangular basis $\{ a_{1},\dots,a_{n}\}$ of $K^{n} \times \{ 0\}$ such that $[a_{1},\dots,a_{n}]$ is a representation matrix for $N \cap (K^{n} \times \{ 0 \})$. 
If $M=\{0\}$ we are all set, so we may take $m \in M$ such that $m \neq 0$. \\
\begin{claim}\label{c1} There is some element $x \in N$ such that  $\pi(x)=m$ and  for any $\alpha \in \mathcal{O}$, if $m$ is $\alpha$-divisible in $M$ then $x$ is $\alpha$-divisible in $N$.
\end{claim}
       \begin{proof}
   Let  $J= \{ \alpha \in \mathcal{O} \ | \ m$ is $\alpha$-divisible in $M \}$. For each $\alpha \in J$, let $m_{\alpha} \in M$ be such that $m=\alpha m_{\alpha}$ and take $n_{\alpha} \in \pi^{-1}(m_{\alpha}) \cap N$. Fix an element $y \in N$ satisfying $\pi(y)=m$ and let 
    $s_{\alpha}=y- \alpha n_{\alpha} \in N \cap (K^{n} \times \{0\})$. \\
Consider $\mathcal{S}=\{x-s_{\alpha} \in \alpha N \cap (K^{n}\times \{0\}) | \ \alpha \in J\}$ this is system of congruences in $N \cap (K^{n} \times \{0\} )$. We will argue that it is pairwise consistent. Let $\alpha, \beta \in \mathcal{O}$, then either $\frac{\alpha}{\beta} \in \mathcal{O}$ or $\frac{\beta}{\alpha} \in \mathcal{O}$ (or both). Without loss of generality assume that $\frac{\alpha}{\beta} \in \mathcal{O}$, then:
     \begin{align*}
         s_{\alpha}-s_{\beta}&= (y-\alpha n_{\alpha})-(y- \beta n_{\beta})=\beta n_{\beta}-\alpha n_{\alpha}= \beta \underbrace{\big( n_{\beta}-\frac{\alpha}{\beta} n_{\alpha}\big)}_{ \in N \cap (K^{n} \times \{0\})}
     \end{align*}
Thus  $s_{\alpha}$ is a solution to the system $\{ x-s_{\alpha} \in \alpha N \cap (K^{n} \times \{0\})\}  \cup \{  x-s_{\beta} \in \beta N \cap (K^{n} \times \{0\})\}$. By maximality of $N \cap (K^{n} \times \{0\})$ we can find an element $z \in N \cap (K^{n} \times \{0 \})$ such that $z$ is a simultaneous solution to the whole system of congruences in $\mathcal{S}$. Let $x=y-z \in N$, then $x$ satisfies the requirements. Indeed, for each $\alpha \in J$, we had chosen $z-s_{\alpha} \in \alpha  N \cap (K^{n} \times \{0\})$, so $z=s_{\alpha}+\alpha w$ for some $w \in N \cap (K^{n} \times \{0\})$. Thus, $\displaystyle{ x= y-z= y- s_{\alpha}-\alpha w=y- (y- \alpha n_{\alpha} )- \alpha w= \alpha(n_{\alpha}-w) \in \alpha N}$, as desired. 
     \end{proof}

Let $s:M \rightarrow N$ be the map sending an element $\alpha m $ to $\alpha x$, where $\alpha \in K$. As $N$ is a torsion free module, $s$ is well defined. One can easily verify that $s$ is a homomorphism such that $\pi \circ s= id_{M}$. Thus, $N$ is the direct sum of $N \cap (K^{n} \times \{0\})$ and $s(M)$, so it is maximal by Lemma \ref{product}. Moreover, $[a_{1},\dots,a_{n},x]$ is a representation matrix for $N$, as required.  
    \end{proof}
    
 \begin{proposition}\label{homomaximal}Let $K$ be a maximal valued field. Let $M,N \subseteq K$ be $\mathcal{O}$-submodules. For any $\mathcal{O}$-homomorphism $h: M \rightarrow K/N$ there is some $a \in K$ such that for any $x \in M$, 
$h(x)=ax+N.$
\end{proposition}
\begin{proof}
By Fact \ref{1dimmod} $M=bI$ where $I$ is a copy of $K$, $\mathcal{O}$ or an (integral or fractional ideal) of $\mathcal{O}$. It is sufficient to prove the statement for $b=1$. Let $S_{I}=\{ v(y) \ | \ y \in I\}$ be the end-segment induced by $I$. Let $\{ \gamma_{\alpha} \ | \ \alpha \in \kappa\}$ be  a co-initial decreasing sequence in $S_{I}$. Choose an element $x_{\alpha} \in K$ such that $v(x_{\alpha})=\gamma_{\alpha}$, then for each $\alpha < \beta < \kappa$,  $x_{\beta}\mathcal{O} \subseteq x_{\alpha} \mathcal{O}$ and $\displaystyle{I=\bigcup_{\alpha \in \kappa} x_{\alpha} \mathcal{O}}$.\\
\begin{claim}\label{1}
For each $\alpha \in \kappa$ there is an element $a_{\alpha} \in K$ such that for all  $x \in x_{\alpha} \mathcal{O}$ we have $h(x)=a_{\alpha}x +N$.
\end{claim}
For each $\alpha$ choose an element $y_{\alpha}$ such that $h(x_{\alpha})= y_{\alpha}+N$ and let $a_{\alpha}= x_{\alpha}^{-1} y_{\alpha}$. Fix an element $x \in x_{\alpha}\mathcal{O}$, then:
\begin{align*}
h(x)&= h(x_{\alpha} \underbrace{(x^{-1}_{\alpha}x)}_{\in \mathcal{O}})= \big(x^{-1}_{\alpha} x\big) \cdot h(x_{\alpha})=(x^{-1}_{\alpha} x) \cdot (a_{\alpha} x_{\alpha}+ N)=a_{\alpha}x + N.
\end{align*}
\begin{claim}\label{2} Given $\beta < \alpha  < \kappa$, then $a_{\beta}-a_{\alpha} \in x_{\beta}^{-1} N$. 
\end{claim}
Note that $x_{\beta} \in x_{\beta}\mathcal{O} \subseteq x_{\alpha}\mathcal{O}$, by Claim \ref{1} we have $h(x_{\beta})=a_{\alpha}x_{\beta} + N= a_{\beta}x_{\beta}+ N$, then $(a_{\alpha}-a_{\beta})x_{\beta} \in N$. Hence, $(a_{\alpha}-a_{\beta}) \in x_{\beta}^{-1} N$.\\
\begin{claim}\label{3} Without loss of generality we may assume that for any $\alpha < \kappa$ there is some $\alpha <\alpha'< \kappa$ such that for any $\alpha' <\alpha'' < \kappa \ $ $a_{\alpha}-a_{\alpha''} \notin x_{\alpha''}^{-1} N$.
\end{claim}
Suppose the statement is  false. Then there is some $\alpha$ such that for any $\alpha< \alpha'$ we can find $\alpha' < \alpha''$ such that $a_{\alpha}-a_{\alpha''} \in x_{\alpha''}^{-1} N$. Define:
\begin{center}
$h^{*}: \begin{cases}
I &\rightarrow K/ N\\
 x &\rightarrow a_{\alpha}x + N. 
\end{cases}$
\end{center}
We will show that for any $x \in I$, $h(x)=h^{*}(x)$. Fix an element $x \in I$, since $< \gamma_{\alpha} \ | \ \alpha \in \kappa>$ is coinitial and decreasing in $S_{I}$ we can find an element $\alpha'> \alpha$ such that $v(x)> \gamma_{\alpha'}$, so $x \in x_{\alpha'} \mathcal{O} \subseteq x_{\alpha''} \mathcal{O}$.  Then 
\begin{align*}
(a_{\alpha}- a_{\alpha'})x= \underbrace{(a_{\alpha}-a_{\alpha''})x}_{\in x_{\alpha''}^{-1}x N\subseteq N}+ \underbrace{(a_{\alpha''}-a_{\alpha'})x}_{{\in x_{\alpha'}^{-1}x N\subseteq N}}.
\end{align*}
we conclude that $(a_{\alpha}- a_{\alpha'})x \in N$. By Claim \ref{1} we have $h(x)= a_{\alpha'}x+N$, thus $h^{*}(x)=h(x)$ and $h^{*}$ 
witnesses the conclusion of the statement. \\
\begin{claim} There is a subsequence $<b_{\alpha} \ | \ \alpha \in \cof(\kappa)>$ of $<a_{\alpha} \ | \ \alpha \in \kappa>$ 
that is pseudo-convergent.
\end{claim}
\begin{proof}
Let $g:\cof(\kappa)\rightarrow \kappa$ be a cofinal function in $\kappa$ i.e. for any $\delta \in \kappa$ there is some $\alpha \in \cof(\kappa)$ such that $g(\alpha)>\delta$. We construct the desired sequence by transfinite recursion in $\cof(\kappa)$, 
building a strictly increasing function $f: \cof(\kappa) \rightarrow \kappa$ satisfying the following conditions:
\begin{enumerate}
\item for each $\alpha < \cof(\kappa)$ we have $b_{\alpha}= a_{f(\alpha)}$ and $f(\alpha)>g(\alpha)$,
\item for any $\alpha < \cof(\kappa)$ the sequence $(b_{\eta} \ | \ \eta < \alpha)$ is pseudo-convergent. 
This is, given $\eta_{1}< \eta_{2} < \eta_{3}< \alpha$ 
\begin{equation*}
v(b_{\eta_{3}}-b_{\eta_{2}})> v(b_{\eta_{2}}-b_{\eta_{1}}),    
\end{equation*}

\item for each $\alpha<\cof(\kappa)$ we have that: for any $\eta<\alpha$, and $f(\alpha)< \eta'<\kappa$ 
\begin{equation*}
    v(a_{\eta'}-b_{\alpha})> v(b_{\eta}-b_{\alpha}) 
\ \text{and} \ a_{\eta'}-b_{\alpha} \notin x_{\eta'}^{-1}N.
\end{equation*}
\end{enumerate}
For the base case, set $b_{0}= a_{0}$ and $f(0)=g(0)+1$. Suppose that for $\mu < \cof(\kappa)$, $f\upharpoonright_{\mu}$ 
has been defined and $< b_{\eta} \ | \ \eta < \mu>$ has been constructed. Let $\mu^{*}= \sup\{ f(\eta) \ | \ \eta < \mu \}$,
by Claim \ref{3} (applied to $\alpha=\max\{\mu^{*},g(\mu)\}$) there is some $\max\{\mu^{*},g(\mu)\} < v < \kappa$ satisfying the following property:\\
\begin{center}
\emph{for any $v<\eta'<\kappa$, $a_{\alpha}-a_{\eta'}\notin x_{\eta'}^{-1}N$.}\\
\end{center}
Set $f(\mu)=v$ and $b_{\mu}=a_{v}$. We continue verifying that the three conditions are satisfied. The first condition  $b_{\mu}=a_{f(\mu)}$ and $f(\mu)>g(\mu)$ follows immediately by construction. \\

We continue checking that $(b_{\eta} \ | \ \eta \leq \mu)$ is a pseudo-convergent sequence. Fix $\eta_{1}< \eta_{2}< \mu$ we must show that $\displaystyle{v(b_{\mu}- b_{\eta_{2}})> v(b_{\eta_{2}}-b_{\eta_{1}})}$. By construction $b_{\mu}=a_{f(\mu)}=a_{v}$ and $f(\mu)=v> \mu^{*}\geq f(\eta_{2})$. Since the third condition holds for $\eta_{2}$ 
we must have $v(a_{v}-b_{\eta_{2}})>v(b_{\eta_{2}}-b_{\eta_{1}})$, as required. \\

Lastly, we verify that the third condition holds for $\mu$.
Let $\eta < \mu$ and $v=f(\mu)<\eta'$ , we aim to show
$\displaystyle{v(a_{\eta'}-b_{\mu})> v(b_{\mu}-b_{\eta})}$. Suppose by contradiction that this inequality does not hold, then $\frac{b_{\mu}-b_{\eta}}{a_{\eta'}-b_{\mu}} \in \mathcal{O}$.
 Because the third condition holds for $\eta$ and by construction $v=f(\mu)> f(\eta)$, we have that $b_{\mu}-b_{\eta}=a_{v}-a_{f(\eta)} \notin x_{v}^{-1}N$.\\
 By Claim \ref{2} $a_{\eta'}-b_{\mu}=a_{\eta'}-a_{v} \in x_{v}^{-1}N$, then:
 \begin{align*}
 b_{\mu}-b_{\eta}=\underbrace{\frac{b_{\mu}-b_{\eta}}{a_{\eta'}-b_{\mu}}}_{\in \mathcal{O}} (a_{\eta'}-b_{\mu}) \in x_{v}^{-1}N,
 \end{align*}
 which leads us to a contradiction. It is only left to show that for any $f(\mu)<\eta'<\kappa$ we have that $a_{\eta'}-b_{\mu} \notin x_{\eta'}^{-1}N$. By construction, we have chose $b_{\mu}=a_{v}$ where $\alpha=\max \{ g(\mu),\mu^{*}\}< v< \kappa$ and for any $v<\eta^{'}<\kappa$ we have:
 \begin{equation*}
     a_{\alpha}-a_{\eta'} \notin x_{\eta'}^{-1} N. 
 \end{equation*}
 As $\alpha<f(\mu)$, by Claim \ref{2} $\displaystyle{a_{\alpha}-a_{f(\mu)} \in x_{\mu}^{-1}N \subseteq x_{\eta'}^{-1}N}$.\\

Fix $\eta'>v=f(\mu)$, then $a_{\eta'}-b_{\mu}=a_{\eta'}-a_{f(\mu)}\notin x_{\eta'}^{-1}N$. Otherwise,
\begin{equation*}
    a_{\eta'}-a_{\alpha}= \underbrace{(a_{\eta'}-a_{f(\mu)})}_{\in x_{\eta'}^{-1}N}+\underbrace{(a_{f(\mu)}- a_{\alpha})}_{\in x_{\eta'}^{-1}N}  \in x_{\eta'}^{-1}N \ \text{because $x_{\eta}^{-1}$N is an $\mathcal{O}$-submodule of $K$,}\
\end{equation*}
which leads us to a contradiction.\\
\end{proof}
 Since $K$ is maximal there is some $a \in K$ that is a pseudolimit of $<b_{\alpha} \ | \ \alpha \in \cof(\kappa)>$. We aim to prove that $h(x)= ax+N$ for $x\in I$. Fix an element $x \in I$. The function $f$ is cofinal in $\kappa$ because of the first condition combined with the fact that $g$ is cofinal in $\kappa$. We can find some $\alpha \in \cof(\kappa)$ such that $x \in x_{f(\alpha)} \mathcal{O} \subseteq I$. By Claim \ref{1} $h(x)=a_{f(\alpha)}x+N$, hence it is sufficient to prove that $(a-a_{f(\alpha)}) x \in N$. As $x \in x_{f(\alpha)} \mathcal{O}$ it is enough to show that $(a-a_{f(\alpha)})=(a-b_{\alpha}) \in x_{f(\alpha)}^{-1} N$. Let $\alpha< \beta< \kappa$, by Claim \ref{2} $(b_{\beta}-b_{\alpha})= (a_{f(\beta)}-a_{f(\alpha)}) \in x_{f(\alpha)^{-1}}N$.  Also, $v(a-a_{f(\alpha)})=v(a_{f(\beta)}-a_{f(\alpha)})$ thus $(a-a_{f(\alpha)})= u(a_{f(\beta)}-a_{f(\alpha)})$ for some $u \in \mathcal{O}^{\times}$, thus  $(a-a_{f(\alpha)}) \in x_{f(\alpha)}^{-1}N$, as desired. \end{proof}
    
\subsection{Valued vector spaces}
We introduce valued vector spaces and some facts that will be required through this paper. An avid and curious reader can consult \cite[Section 2.3]{BookVal} for a more exhaustive presentation. Through this section we fix $(K,\Gamma, v)$ a valued field and $V$ a $K$-vector space.
\begin{definition} A tuple $(V,\Gamma(V),val,+)$ is a \emph{valued vector space} structure if:
\begin{enumerate}
    \item $\Gamma(V)$ is a linear order,
    \item there is an  action $+: \Gamma \times \Gamma(V) \rightarrow \Gamma(V)$ which is order preserving in each coordinate,
    \item $val: V \rightarrow \Gamma(V)$ is a map such that for all $v,w \in V$ and $\alpha \in K$ we have:
    \begin{itemize}
        \item $val(v+w) \geq \min \{ val(w),val(v)\}$,
        \item $val(\alpha v)= v(\alpha)+val(v)$. 
    \end{itemize}
\end{enumerate}
\end{definition}
The following Fact is \cite[Remark 1.2]{will}. 
\begin{fact}\label{finiteorbits} Let $V$ be a finite dimensional valued vector space over $K$, then the action of $\Gamma(K)$ over $\Gamma(V)$ has finitely many orbits. In fact, $|\Gamma(V)/ \Gamma(K)| \leq dim_{K}(V)$. 
\end{fact}

\begin{definition} Let $(V,\Gamma(V), val, +)$ be a valued vector space:
\begin{enumerate}
    \item Let $a \in V$ and $\gamma \in \Gamma(V)$. A \emph{ball} in $V$ is a set of the form:
    \begin{equation*}
       \overline{Ball_{\alpha}(a)}= \{ x \in V \ | \ val(x-a)\geq \gamma\} \ \text{or} \  Ball_{\alpha}(a)=\{ x \in V \ | \ val(x-a)> \gamma\}.
    \end{equation*}
    \item We say that $(V,\Gamma(V), val, +)$ is \emph{maximal} if every nested family of balls in $V$ has non-empty intersection.
\end{enumerate}
\end{definition}

\begin{definition} Let $(V, \Gamma(V), val, +)$ be a valued vector space and let $W$ be a subspace of $V$. Then $(W,\Gamma(W),val,+)$ is also a valued vector space, where $\Gamma(W)=\{ val(w) \ | \ w \in W\}$. \\
We say that:
\begin{enumerate}
\item \emph{$W$ is maximal in $V$} if every family of nested balls 
\begin{equation*}
    \{ Ball_{\alpha}(x_{\alpha}) \ | \ \alpha \in S\}, \ \text{where} \ S\subseteq \Gamma(W) \  \text{and for each $\alpha \in S \ x_{\alpha} \in W$. }
\end{equation*}
that has non-empty intersection in $V$ has non-empty intersection in $W$.
\item $W \leq V$ has the \emph{optimal approximation property} if for any $v \in V \backslash W$ the set $ \{val(v-w) \ | \ w \in W\}$  attains a maximum.
%And if such maximum is attained in $val(v-w_{0})$ then we say that $v-w_{0}$ is a \emph{good representative} of the coset $v+W$. 
\end{enumerate}
\end{definition}
The following is a folklore fact. 
\begin{fact}\label{optclosed} Let $(V, \Gamma(V),val, +)$ be a valued vector space, and $W$ a subspace of $V$ the following statements are equivalent:
\begin{enumerate}
    \item $W$ is maximal in $V$,
    \item $W$ has the optimal approximation property in $V$. 
    \end{enumerate}
    Additionally, if $W$ is maximal then it is maximal in $V$.
\end{fact}

We conclude this subsection with the definition of separated basis. 
\begin{definition}Let $(V, \Gamma(V), val, +)$ be a valued vector space. Assume that $V$ is a $K$-vector space of dimension $n$. A basis $\{v_{1},\dots,v_{n}\} \subseteq V$ is a \emph{separated basis} if for any  $\alpha_{1},\dots,\alpha_{n} \in K$ we have that:
\begin{equation*}
val(\sum_{i\leq n} \alpha_{i}v_{i})=\min\{ val(\alpha_{i}v_{i}) \ | \ i \leq n \}.
\end{equation*}
\end{definition}

   \section{Definable modules}\label{modules}
    In this section we study definable $\mathcal{O}$-submodules in henselian valued fields of equicharacteristic zero.  
    \begin{corollary}\label{basismodule}Let $(F,v)$ be a henselian valued field of equicharacteristic zero and $N$ be a definable $\mathcal{O}$-submodule of $F^{n}$. Then $N$ is definably isomorphic to a direct sum of copies of $F$, $\mathcal{O}$, or (integral or fractional) ideals of $\mathcal{O}$. Moreover, if 
    $\displaystyle{N \cong \oplus_{i \leq n} I_{i}}$ there is some upper triangular basis $\{ a_{1}, \dots, a_{n}\}$ of $F^{n}$ such that $[a_{1}, \dots, a_{n}]$ is a representation matrix of $N$.
    \end{corollary}
    \begin{proof}
    By Fact \ref{maxext} we can find $F'$ an elementary extension of $F$ that is maximal, so we can apply Theorem \ref{characterizationmodules}. As the statement that we are trying to show is first order expressible, it must hold as well in $F$. 
    \end{proof}
   \begin{corollary}\label{homomorphism} Let $(F,v)$ be a henselian valued field of equicharacteristic zero and let $N,M \subseteq F$ be a definable $\mathcal{O}$-submodules.
Then for any definable $\mathcal{O}$-homomorphism $h: M \rightarrow K/N$. Then there is some $b \in F$ satisfying that for any $y \in M$, $h(y)=by+N$. 
%Let $I \in \mathcal{F}$, $a \in K$ and $h: aI \rightarrow K/N$. Then there is some $b \in K$ satisfying that for any $y \in aI$, $h(y)=by+N$. 
 \end{corollary}
 \begin{proof}
 By Fact \ref{maxext} we can find an elementary extension  $F \prec F'$ that is maximal. The statement follows by applying Proposition \ref{homomaximal}, because it is first order expressible.
 \end{proof}
 \subsection{Definable modules in valued fields of equicharacteristic zero with residue field algebraically closed and value group with bounded regular rank}
Let $(K,v)$ be a henselian valued field of equicharacteristic zero with residue field algebraically closed and value group with bounded regular rank. Let $\mathcal{O}$ be its valuation ring and $T$ be the complete $\mathcal{L}$-first order theory of $(K,v)$. In this section we study the definable $\mathcal{O}$-modules and torsors. Let $\mathcal{I}'$ be the complete family of $\mathcal{O}$-submodules of $K$ described in Fact \ref{completemodules}. From now on we fix a complete family $\displaystyle{\mathcal{I}=\mathcal{I}' \backslash \{ 0, K\}}$.

. %If the torsor $X$ is a coset of the submodule $A$ of $K^{n}$, then the subtorsor of $X$ is a coset (contained in $X$) of an $\mathcal{O}$-submodule $U$ of $A$. 

    \begin{remark} If  $K \vDash T$, then   $\displaystyle{N \cong \oplus_{i \leq n} I_{i}}$, where each $I_{i} \in \mathcal{F} \cup \{ 0, K\}$. This follows because $\mathcal{F}$ is a complete family of $\mathcal{O}$-modules.
    \end{remark}
   
  \begin{definition} Let $K \vDash T$. A \emph{ definable torsor} $U$ is a coset in $K^{n}$ of a definable $\mathcal{O}$-submodule of $K^{n}$, if $n=1$ we say that $U$ is a $1$-torsor. Let $U$ be a definable $1$-torsor, we say that $U$ is:
  \begin{enumerate}
\item  \emph{closed} if it is a translate of a submodule of $K$ of the form $a\mathcal{O}$. 
\item it is \emph{open} if it is either $K$ or a translate of a submodule of the form $a I$ for some $a \in K$, where $I\in \mathcal{F} \backslash \mathcal{O}$. 
\end{enumerate}
\end{definition}
 \begin{definition} Let $(I_{1}, \dots, I_{n}) \in \mathcal{F}^{n}$ be a fixed tuple. 
 \begin{enumerate} 
\item  An $\mathcal{O}$-module $M \subseteq K^{n}$ is of \emph{type $(I_{1}, \dots, I_{n})$}  if $M \cong \bigoplus_{i \leq n} I_{i}$.
\item An $\mathcal{O}$-module $M \subseteq K^{n}$ of type $(\mathcal{O}, \dots, \mathcal{O})$ is said to be an $\mathcal{O}$-lattice of rank $n$. 
%\item  Let $\Lambda_{(I_{1}, \dots, I_{n})}= \{ M \subseteq K^n \ | \ M \ $ is an $\mathcal{O}$-module of type $(I_{1},\dots, I_{n})\}$. 
\item  A torsor $Z$ is \emph{of type $(I_{1},\dots,I_{n})$}, if $Z= \bar{d}+M$ where  $M \subseteq K^{n}$ is an $\mathcal{O}$-submodule of $K^{n}$ of type $(I_{1},\dots,I_{n})$.
%\item We denote as $\Lambda_{(I_{1}, \dots, I_{n})}^{*}$ the set of all torsors of type $(I_{1},\dots, I_{n})$. This is $\displaystyle{\Lambda_{(I_{1}, \dots, I_{n})}^{*}= \bigcup_{M \in \Lambda_{(I_{1}, \dots, I_{n})}} K^{n} / M}$, the set of all cosets in of an $\mathcal{O}$-module of type $(I_{1},\dots,I_{n})$ in $K^{n}$ . 
\end{enumerate}
\end{definition}

\begin{proposition}\label{torsormodule} Let $Z$ be a torsor of type $(I_{1}, \dots, I_{n})$. Then there is some $\mathcal{O}$-module $L\subseteq K^{n+1}$ of type $(I_{1}, \dots, I_{n},\mathcal{O})$ such that $\ulcorner Z \urcorner$ and $\ulcorner L \urcorner $ are interdefinable. 
\end{proposition}
\begin{proof}
Let $N \subseteq K^{n}$ be the $\mathcal{O}$-submodule and take $\bar{d} \in K^{n}$ be such that $Z=\bar{d}+N$. Let $N_{2}=N\times \{0\}$ which is an $\mathcal{O}$ submodule of $K^{n+1}$ and let $\bar{b}=\begin{bmatrix} \bar{d}\\1 \end{bmatrix}$. Define the  $\mathcal{O}$-module of $K^{n+1}$:
\begin{align*}
L_{\bar{d}}&:=N_{2}+\bar{b}\mathcal{O}=\{\begin{bmatrix} n+\bar{d}r\\r \end{bmatrix} \ | \ r \in \mathcal{O}, n \in N\}.
\end{align*}
By a standard computation, one can verify that the definition of $L_{\bar{d}}$ is independent of the choice of $\bar{d}$, i.e. if $\bar{d}-\bar{d}' \in N$ then $L_{\bar{d}}=L_{\bar{d}'}$. 
%\begin{proof}
%We show that $L_{\bar{d}} \subseteq L_{\bar{d}'}$ and by symmetry of the argument we must have that $L_{\bar{d}} = L_{\bar{d}'}$. Fix elements $x \in \mathcal{O}$ and $n \in N_{2}$ then 
%\begin{align*}
%\begin{bmatrix} 1 \\ \bar{d} \end{bmatrix} x+n= \begin{bmatrix} 1 \\ \bar{d}' \end{bmatrix} x+ \underbrace{ \begin{bmatrix} 0 \\ \bar{d}- \bar{d}' \end{bmatrix} x +n}_{=n_{1} \in N_{2}}
%\end{align*}
%thus $\begin{bmatrix} 1 \\ \bar{d} \end{bmatrix} x+n \in L_{\bar{d}'}$, and we conclude that $L_{\bar{d}} \subseteq L_{\bar{d}'}$ as required.
%\end{proof}
So we can denote $L=L_{\bar{d}}$, and we aim to show that $L$ and $Z$ are interdefinable. It is clear that $\ulcorner L \urcorner \in dcl^{eq}(\ulcorner Z \urcorner)$, while $\ulcorner Z \urcorner \in dcl^{eq}(\ulcorner L \urcorner)$ because $Z=\pi_{2 \leq n+1 }\big(L \cap  (K^{n}\times\{1\}) \big)$ where $\pi_{2 \leq n+1 }: K^{n+1} \rightarrow K^{n}$ is the projection into the last $n$-coordinates. %Thus $\ulcorner Z \urcorner \in dcl^{eq}(\ulcorner L \urcorner)$. For the converse, suppose that $Z$ is being given and pick an element $\bar{d} \in Z$, then $N=\{ \bar{m}-\bar{d} \ | \ \bar{m} \in Z\}$, thus $\ulcorner N \urcorner \in dcl^{eq}(\ulcorner Z \urcorner)$. Once $N$ has been specified, $N_{2}$ is definable and we can define $L=L_{\bar{d}}$, by the claim the definition of $L$ is independent from the choice of $\bar{d}$ and we conclude $\ulcorner L \urcorner \in dcl^{eq}(\ulcorner Z \urcorner)$, as required. 
\end{proof}

\subsection{Definable 1-$\mathcal{O}$-modules}
In this subsection we study the quotient modules of $1$-dimensional modules. 
\begin{notation} Let $M \subseteq K$ be a definable $\mathcal{O}$-module. We denote by $S_{M}:= \{ v(x) \ | \ x \in M\}$ the \emph{end-segment induced by $M$}. We recall as well that we write $\mathcal{F}$ to denote the complete family of $\mathcal{O}$-submodules of $K$ previously fixed.
\end{notation}

\begin{definition} A \emph{definable $1$-$\mathcal{O}$-module} is an $\mathcal{O}$-module which is definably isomorphic to a quotient of a definable $\mathcal{O}$-submodule of $K$ by another, i.e. something of the form $\displaystyle{aI/bJ}$ where $a,b \in K$ and $I,J \in \mathcal{F}\cup \{ 0, K\}$. 
\end{definition}

  The following operation between $\mathcal{O}$-modules will be particularly useful in our setting.
 \begin{definition} Let $N, M$ be $\mathcal{O}$-submodules of $K$, we define the \emph{colon module} $\displaystyle{ Col(N: M)=\{ x \in K \ | \ x M \subseteq N\}}$.  It is a well known fact from Commutative Algebra that $Col(N:M)$ is also an $\mathcal{O}$-module. 
 \end{definition}

\begin{lemma}\label{homreg} Let $K \vDash T$.  Let $A$ be a $1$-definable $\mathcal{O}$-module. 
Suppose that $A=A_{1}/A_{2}$, where $A_{2} \leq A_{1}$ are $\mathcal{O}$-submodules of $K$. Then the $\mathcal{O}$-module  $Hom_{\mathcal{O}}(A,A)$ is definably isomorphic to the $1$-definable 
 $\mathcal{O}$-module 
 \begin{equation*}
\big( Col(A_{1}:A_{1})\cap Col(A_{2}:A_{2})\big)/ Col(A_{2},A_{1}). 
 \end{equation*}
\end{lemma}
\begin{proof}
By Fact \ref{maxext} without loss of generality we may assume $K$ to be maximal, because the statement is first order expressible. Let 
\begin{equation*}
\mathcal{B}=\{ f: A_{1} \rightarrow A_{1}/ A_{2} \ | \ \text{$f$ is a homomorphism and}\ A_{2}\subseteq ker(f)\}.
\end{equation*}
$\mathcal{B}$ is canonically in one-to-one correspondence with $Hom_{\mathcal{O}}(A,A)$. By Corollary \ref{homomorphism}, 
for every homomorphism $f \in \mathcal{B}$ there is some $b_{f} \in K$ satisfying that for any $x \in A_{1}$, $f(x)=b_{f}x +A_{2}$ and we say that \emph{$b_{f}$ is a linear representation of $f$}.
\begin{claim} {Let $f\in \mathcal{B}$ if $b_{f}$ is a linear representation of $f$, then $b_{f} \in Col(A_{1}:A_{1}) \cap Col(A_{2}:A_{2})$.}
\end{claim}
\begin{proof}
First we verify that $b_{f}\in Col(A_{1}:A_{1})$. Let $x \in A_{1}$, by hypothesis $f(x)=b_{f}x+A_{2}\in A_{1}/A_{2}$. Then there is some $y \in A_{1}$ such that $b_{f}x+A_{2}=y+A_{2}$ and therefore $b_{f}x-y \in A_{2} \subseteq A_{1}$. Consequently, $b_{f}x \in y+A_{1}=A_{1}$, and as $x$ is an arbitrary element we conclude that $b_{f}\in Col(A_{1}:A_{1})$. We check now that $b_{f} \in Col(A_{2}:A_{2})$, and we fix an element $x \in A_{2}$. By hypothesis, $b_{f}x+A_{2}=A_{2}$ so $b_{f}x \in A_{2}$, and as $x \in A_{2}$ is an arbitrary element we conclude that $b_{f}\in Col(A_{2}:A_{2})$. 
\end{proof}
\begin{claim}\label{wd} {Let $f\in \mathcal{B}$ if $b_{f},b_{f}'$ are linear representations of $f$, then $b_{f}-b_{f}' \in Col(A_{2}:A_{1})$}
\end{claim}
\begin{proof}
Let $x \in A_{1}$, by hypothesis $f(x)=b_{f}x+A_{2}=b_{f}'x+A_{2}$, so $(b_{f}-b_{f}')x \in A_{2}$. Because $x$ is arbitrary in $A_{1}$ we have that $(b_{f}-b_{f}') \in Col(A_{2}:A_{1})$.
\end{proof}
We consider the map  $\displaystyle{\phi: \mathcal{B} \rightarrow \big(Col(A_{1}:A_{1})\cap Col(A_{2}:A_{2})\big)/ Col(A_{2}:A_{1})}$ that sends an $\mathcal{O}$-homomorphism $f$ to the coset $b_{f}+Col(A_{1}:A_{2})$. By Claim \ref{wd} such map is well defined.\\
By a standard computation $\phi$ is an injective $\mathcal{O}$-homomorphism. To show that $\phi$ is surjective, let 
\begin{equation*}
 b \in Col(A_{1}:A_{1}) \cap Col(A_{2}:A_{2}),
\end{equation*}
 and consider $f_{b}: A_{1}\rightarrow A_{1}/A_{2}$, the map that sends the element $x$ to $bx+A_{2}$. Because $b \in Col(A_{2}:A_{2})$, for any $x \in A_{2}$ we have that $bx \in A_{2}$ thus $A_{2} \subseteq ker(f_{b})$. Consequently, $f_{b} \in \mathcal{B}$ and  $\phi(f_{b})=b+Col(A_{1}:A_{2})$.
\end{proof}
%% THIS PART WILL BE RELEVANT FOR THE CONSTRUCTION OF RESOLUTIONS, FOR THE MOMENT, TAKE IT AWAY AND MAKE THE OTHER PAPER MORE CLEAR!!!
\begin{lemma}\label{quotients} Let $n \in \mathbb{N}_{\geq 2}$ and $M \subseteq K^{n}$ be an $\mathcal{O}$-module.
 \begin{enumerate}
 \item Let $\pi^{n-1}: K^{n} \rightarrow K^{n-1}$ be the projection into the first $(n-1)$-coordinates and $B_{n-1}=\pi^{n-1}(M)$. Take $A_{1} \subseteq K$ be the $\mathcal{O}$-module such that $ker(\pi^{n-1})= M \cap (\{0\}^{n-1} \times K)= (\{ 0 \} ^{n-1} \times A_{1})$. 
 \item Let $\pi_{n}: K^{n} \rightarrow K$ be the projection into the last coordinate and $B_{1}=\pi_{n}(M)$. Let  $A_{n-1} \subseteq K^{n-1}$ be the $\mathcal{O}$-module such that $ker(\pi_{n})= M \cap ( K^{n-1}\times \{0\})= (A_{n-1} \times \{ 0\})$. 
 \end{enumerate}
 Then $A_{n-1} \leq B_{n-1}$ and both lie in $K^{n-1}$, and $A_{1} \leq B_{1}$ and both lie in $K$. The map $\phi: B_{n-1} \rightarrow B_{1}/A_{1}$ given by $b \mapsto a+A_{1} \text{ where $(b,a) \in M$}$, is a well defined homomorphism of $\mathcal{O}$-modules whose kernel is $A_{n-1}$. In particular, $B_{n-1}/ A_{n-1} \cong B_{1}/A_{1}$. Furthermore if $M$ is definable, $\phi$  is also definable. 
 \end{lemma}
\begin{proof}
 Let $\bar{m} \in A_{n-1}$, then $(\bar{m}, 0) \in M$ thus $\pi^{n-1}(\bar{m},0)=\bar{m} \in B_{n-1}$. We conclude that $A_{n-1}$ is a submodule of $B_{n-1}$. Likewise $A_{1} \leq B_{1}$. For the second part of the statement, it is a straightforward computation to verify that the map $\phi: B_{n-1} \rightarrow B_{1}/A_{1}$( defined as in the statement) , is a well defined surjective homomorphism of $\mathcal{O}$-modules whose kernel is $A_{n-1}$. Lastly, the definability of $\phi$ follows immediately by the definability of $M$. 
 \end{proof}

  \section{The Stabilizer sorts}\label{stabilizer}

\subsection{An abstract criterion to eliminate imaginaries}\label{criterion}
We start by recalling Hurshovski's criterion, The following is \cite[Lemma 1.17]{criteria}.
\begin{theorem}\label{weakEI}Let $T$ be a first order theory with home sort $K$ (meaning that $\mathfrak{M}^{eq}=dcl^{eq}(K)$). Let $\mathcal{G}$ be some collection of sorts. If the following conditions all hold, then $T$ has weak elimination of imaginaries in the sorts $\mathcal{G}$. 
\begin{enumerate}
\item \emph{Density of definable types: } for every non-empty definable set $X \subseteq K$ there is an $acl^{eq}(\ulcorner X \urcorner )$-definable type in $X$.
\item \emph{Coding definable types:}  every definable type in $K^{n}$ has a code in $\mathcal{G}$ (possibly infinite). This is, if $p$ is any (global) definable type in $K^{n}$, then the set $\ulcorner p \urcorner$ of codes of the definitions of $p$ is interdefinable with some (possibly infinite) tuple from $\mathcal{G}$.
\end{enumerate}
\end{theorem}
\begin{proof}
 A very detailed proof can be found in \cite[Theorem 6.3]{will}. The first part of the proof shows weak elimination of imaginaries as it is shown that for any imaginary element $e$ we can find a tuple $a \in \mathcal{G}$ such that $e \in dcl^{eq}(a)$ and $a \in acl^{eq}(e)$. $\square$  
\end{proof}
We start by describing the sorts that are required to be added to apply this criterion and show that any valued field of equicharacteristic zero, with residue field algebraically closed and value group with bounded regular rank admits weak elimination of imaginaries.
\begin{definition}
 For each $n \in \mathbb{N}$,  let $\{ e_{1},\dots, e_{n}\}$ be the canonical basis of $K^{n}$ and  $(I_{1},\dots, I_{n})\in \mathcal{F}^{n}$. 
\begin{enumerate}
\item Let  $\displaystyle{C_{(I_{1}, \dots, I_{n})}= \{ \sum_{1 \leq i \leq n} x_{i}e_{i} \ | \ x_{i} \in I_{i}\}}$, we refer to this module as \emph{the canonical $\mathcal{O}$-submodule of $K^{n}$ of type $(I_{1},\dots,I_{n})$}. 
\item We denote as $B_{n}(K)$ the multiplicative group of $n\times n$-upper triangular and invertible matrices.
\item  We define the subgroup $Stab_{(I_{1}, \dots, I_{n})}= \{ A \in B_{n}(K) \ | \ AC_{(I_{1}, \dots, I_{n})}= C_{(I_{1}, \dots, I_{n})}\}$.
\item Let $ \Lambda_{(I_{1},\dots, I_{n})}  := \{  M  \ | \ M \subseteq K^{n}$ is an $\mathcal{O}$-module of type $(I_{1},\dots,I_{n})\}$. 
\item Let  $\mathcal{U}_{n} \subseteq (K^{n})^{n}$ be the set of $n$-tuples $(\bar{b}_{1}, \dots,\bar{b}_{n})$, such that $B=[\bar{b}_{1},\dots,\bar{b}_{n}]$ is an invertible upper triangular matrix. We define the equivalence relation $E_{(I_{1},\dots, I_{n})}$ on $\mathcal{U}_{n}$ as: 
\begin{equation*}
E_{(I_{1},\dots, I_{n})} \big(\bar{a}_{1},\dots,\bar{a}_{n} ; \bar{b}_{1},\dots, \bar{b}_{n} \big) \ \text{holds if and only if}
\end{equation*}
 $(\bar{a}_{1}, \dots, \bar{a}_{n})$ and $(\bar{b}_{1}, \dots, \bar{b}_{n})$ generate the same $\mathcal{O}$-module of type $(I_{1},\dots, I_{n})$, i.e.
\begin{equation*}
  \{ \sum_{1 \leq i \leq n} x_{i}\bar{a}_{i} \ | \ x_{i} \in I_{i}\}= \{ \sum_{1 \leq i \leq n} x_{i}\bar{b}_{i} \ | \ x_{i} \in I_{i}\}. \\
\end{equation*}
\item We denote as $\bar{\rho}_{(I_{1},\dots, I_{n})}$ the canonical projection map:
\begin{center}
$\bar{\rho}_{(I_{1},\dots, I_{n})}: \begin{cases}
\mathcal{U}_{n} & \rightarrow \mathcal{U}_{n}/ E_{(I_{1},\dots,I_{n})}\\
 (\bar{a}_{1},\dots, \bar{a}_{n}) &\mapsto [(\bar{a}_{1}, \dots, \bar{a}_{n} )]_{E_{(I_{1},\dots,I_{n})}}
\end{cases}$
\end{center}
\end{enumerate}
\end{definition}
\begin{remark}\label{identifications}
\begin{enumerate}
\item The set $\{ \ulcorner M \urcorner \ | \ M \in \Lambda_{(I_{1},\dots, I_{n})}\}$ can be canonically identified with $B_{n}(K)/Stab_{(I_{1},\dots, I_{n})}$. Indeed, by Corollary \ref{basismodule} given any $\mathcal{O}$-module $M$ of type $(I_{1},\dots,I_{n})$ we can find an upper triangular basis $\{ \bar{a}_{1},\dots, \bar{a}_{n}\}$ of $K^{n}$  such that $[a_{1}, \dots, a_{n}]$ is a matrix representation of $M$. The code $\ulcorner M \urcorner$ is interdefinable with the coset $[a_{1}, \dots, a_{n}] Stab_{(I_{1},\dots, I_{n})}$. 
\item Fix some $n \in \mathbb{N}_{\geq 2}$ and let  $(I_{1},\dots, I_{n})$ be a fixed tuple. The sort $B_{n}(K)/ Stab_{(I_{1}, \dots, I_{n})}$ is in definable bijection with the equivalence classes of $\mathcal{U}_{n}/ E_{(I_{1},\dots,I_{n})}$. In fact we can consider the $\emptyset$-definable map:
\begin{center}
$f: \begin{cases}
\mathcal{U}_{n}/ E_{(I_{1},\dots,I_{n})} &\rightarrow B_{n}(K)/ Stab_{(I_{1},\dots,I_{n})}\\
[(\bar{a}_{1},\dots, \bar{a}_{n})]_{E_{(I_{1},\dots, I_{n})}} & \mapsto [\bar{a}_{1},\dots,\bar{a}_{n}] Stab_{(I_{1},\dots, I_{n})}.
\end{cases}$
\end{center}
We denote as $\rho_{(I_{1},\dots,I_{n})}= \mathcal{U}_{n} \rightarrow B_{n}(K)/Stab_{(I_{1,\dots,I_{n}})}$  the composition maps $\rho_{(I_{1},\dots,I_{n})}=f \circ \bar{\rho}_{(I_{1},\dots,I_{n})}$.
\end{enumerate}
\end{remark}
%\textcolor{red}{mejorar la descripcion, considero que se puede escribir mejor.}
\begin{definition}\label{defstab}[The stabilizer sorts] We consider the language $\mathcal{L}_{\mathcal{G}}$ extending the three sorted language $\mathcal{L}$ (defined in Subsection \ref{language}), where: 
\begin{enumerate}
\item  We equipped the value group with the multi-sorted language $\mathcal{L}_{bq}$ introduced in Subsection \ref{OAGBR}. 
\item For each $n \in \mathbb{N}$ we consider the parametrized family of sorts $B_{n}(K)/Stab_{(I_{1},\dots,I_{n})}$ and  maps
\begin{equation*}\rho_{(I_{1},\dots,I_{n})}: \mathcal{U}_{n} \rightarrow B_{n}(K)/Stab_{(I_{1},\dots,I_{n})}
\end{equation*}
where $(I_{1},\dots,I_{n}) \in \mathcal{F}^{n}$. 
\end{enumerate}
We refer to the sorts in the language $\mathcal{L}_{\mathcal{G}}$ as the \emph{stabilizer sorts}. We denote as $\mathcal{G}$ their union, i.e.
\begin{align*}
K \cup k \cup \Gamma \cup \{ \Gamma/\Delta \ | \ \Delta \in RJ(\Gamma)\} \cup \{\Gamma/\Delta+n\Gamma \ | \ \Delta \in RJ(\Gamma), n \in \mathbb{N}_{\geq 2}\}  \cup \{ B_{n}(K)/Stab_{(I_{1},\dots,I_{n})} \ | \ n \in \mathbb{N}, (I_{1}, \dots, I_{n}) \in \mathcal{F}^{n}\}.
\end{align*}
\end{definition}
\begin{remark}\label{ACVF}
The geometric sorts for the case of $ACVF$ are a particular instance of the stabilizer sorts. Let $S_{n}$ denotes the set of $\mathcal{O}$-lattices of $K^{n}$ of rank $n$, these are simply the $\mathcal{O}$-modules of type $\underbrace{(\mathcal{O}, \dots, \mathcal{O})}_{n- \text{times}}$. For each $\Lambda \in S_{n}$, let $res(\Lambda)=\Lambda \otimes_{\mathcal{O}}k= \Lambda/ \mathcal{M} \Lambda$, which is a $k$ -vector space of dimension $n$.\\
 Let $\displaystyle{T_{n}=\bigcup_{\Lambda \in S_{n}} res(\Lambda)=\{(\Lambda,x) \ | \ \Lambda \in S_{n}, x \in res(\Lambda)\}}$.  Each of these torsors is considered in the stabilizer sorts as the code of an $\mathcal{O}$ module of type $( \mathcal{M}, \dots, \mathcal{M},\mathcal{O})$, because  any torsor of the form $a + \mathcal{M} \Lambda$ for some $\Lambda \in S_{n}$ can be identified with an $\mathcal{O}$-module of type $(\underbrace{\mathcal{M}, \dots, \mathcal{M}}_{n- \text{times}},\mathcal{O})$ (see Proposition \ref{torsormodule}).
\end{remark}

\subsection{An explicit description of the stabilizer sorts}
In this subsection we state an explicit description of the subgroups $Stab_{(I_{1},\dots,I_{n})}$. 

\begin{notation} For each $\Delta \in RJ(\Gamma)$  we denote as $\mathcal{O}_{\Delta}$ the valuation ring of $K$ of the coarsened valuation $v_{\Delta}: K^{\times}\rightarrow \Gamma/\Delta$ induced by $\Delta$.
\end{notation}
\begin{fact} \label{stabideals} Let $I \in \mathcal{F}$ and let $S_{I}=\{v(x) \ | \ x \in I\}$. Then $Stab(I)=\mathcal{O}_{\Delta_{S_{I}}}^{\times}=\{x \in K \ | \ v(x) \in \Delta_{S_{I}}\}$. 
\end{fact}
\begin{proof}
This is an immediate consequence of Fact \ref{stab1}.
\end{proof}
\begin{proposition}\label{stab}
 Let $n \in \mathbb{N}$, and $(I_1,\dots, I_{n})  \in \mathcal{F}^{n}$.
Then 
\begin{align*}
Stab_{(I_{1}, \dots, I_{n})}= \{ ((a_{i,j})_{1 \leq i,j \leq n}  \in B_{n}(K) \ | \ a_{ii} \in \mathcal{O}_{\Delta_{S_{I_{i}}}}^{\times}\ \wedge a_{ij} \in Col(I_{i},I_{j})  \   \text{for each $1 \leq i < j \leq n$ }\} 
\end{align*}
\end{proposition}
\begin{proof}
This is a straightforward computation and it is left to the reader. 
\end{proof}

    \section{Weak Elimination of imaginaries for henselian valued field with value group with bounded regular rank}\label{density}
Let $(K,v)$ be a henselian valued field of equicharacteristic zero, with residue field algebraically closed and value group with bounded regular rank. Let $T$ be its complete $\mathcal{L}_{\mathcal{G}}$-first order theory and $\mathfrak{M}$ its monster model. \\
In this section we show that both conditions required by Hrushovski's criterion to obtain weak elimination of imaginaries down to the stabilizer sorts hold.  
\subsection{Density of definable types}
In this section we prove density of definable types for $1$-definable sets $X \subseteq \mathfrak{M}$. There are two ways to tackle this problem. One can either use the quantifer elimination (see Corollary \ref{QEregular}) and obtain a \emph{canonical} decomposition of $X$ into nice sets $T_{i} \in \acl^{eq}(\ulcorner X \urcorner)$ and then build a global type $p(x) \in x \in T_{i}$ which is $\acl^{eq}(\ulcorner T_{i}\urcorner)$. This approach was successfully achieved by Holly in \cite{Holly} for the case of $ACVF$ and real closed valued fields, and her work essentially gives a way to code one-definable sets in the main field down to the geometric sorts. It is worth pointing out, that finding \emph{a canonical decomposition} is often a detailed technical work. Instead of following this strategy, we follow a different approach that exploits the power of generic types, which are definable partial types.
\begin{definition}\label{generico} Let $U\subseteq \mathfrak{M}$ be a definable $1$-torsor, let
\begin{equation*}
    \Sigma_{U}^{gen}(x)=\{ x \in U\} \cup \{ x \notin B \ | \ B \subsetneq U \ | \ \text{$B$ is a proper sub-torsor of $U$}\}.
\end{equation*}
This is a $\ulcorner U \urcorner$-definable partial type.
\end{definition}
\begin{proposition}\label{uniqueclosed} Let $U \subseteq \mathfrak{M}$ be a definable closed $1$-torsor. Then there is a \emph{unique} complete global type $p(x)$ extending $\Sigma_{U}^{gen}(x)$ i.e. $\Sigma_{U}^{gen}(x) \subseteq p(x)$. Moreover, $p(x)$ is $\ulcorner U \urcorner$-definable.
\end{proposition}
\begin{proof}
Let $a \in U$ and $Y_{a}=\{ v(x-a)\ | \ x \in U\} \subseteq \Gamma$. As $U$ is a closed $\mathcal{O}$-module then $Y_{a}$ has a minimum element $\gamma$. For any other element $b \in U$, we have that $\gamma=\min(Y_{a})=\min(Y_{b})$, thus $\gamma \in \dcl^{eq}(\ulcorner U \urcorner)$.\\
By quantifier elimination (see Corollary \ref{QEregular}) it is sufficient to show that $\Sigma_{U}^{gen}(x)$ determines also the congruence and coset formulas. Let $c$ be a realization of $\Sigma^{gen}_{U}(x)$ and then for any $a \in U(\mathfrak{M})$ we have that $v(c-a)=\gamma$. Let $p(x)=\tp(c/\mathfrak{M})$, $\Delta \in RJ(\Gamma)$, $\ell \in \mathbb{N}$ and $\beta \in \Gamma$. If $a \in U(\mathfrak{M})$ then:
\begin{align*}
    \vDash v_{\Delta}(c-a)-\rho_{\Delta}(\beta)+k^{\Delta} \in \Delta  \ &\text{if and only if} \ \vDash \phi_{\Delta}^{k}(\beta):=\gamma-\rho_{\Delta}(\beta)+k^{\Delta}, \ \text{and}\\
     \vDash v_{\Delta}(c-a)-\rho_{\Delta}(\beta)+k^{\Delta} \in \ell \big(\Gamma/\Delta) \ &\text{if and only if} \ \vDash \psi_{\Delta}^{k}(\beta):=\gamma-\rho_{\Delta}(\beta)+k^{\Delta} \in \ell(\Gamma/\Delta).
\end{align*}
We observe that $\psi_{\Delta}^{k}(\beta)$ and $\phi_{\Delta}^{k}(\beta)$ are  $\mathcal{L}(\dcl^{eq}(\ulcorner U\urcorner))$-formulas, and their definition is completely  independent from the choice of $c$.\\
If $a \notin U(\mathfrak{M})$, then for any $b \in U(\mathfrak{M})$ we have that $v(c-a)=v(b-a)$. Therefore
\begin{align*}
    \vDash v_{\Delta}(c-a)-\rho_{\Delta}(\beta)+k^{\Delta} \in \Delta  \ &\text{if and only if} \ \vDash \epsilon_{\Delta}^{k}(a,\beta):= \exists b \in U \big( v_{\Delta}(b-a)-\rho_{\Delta}(\beta)+k^{\Delta}\big), \ \text{and}\\
     \vDash v_{\Delta}(c-a)-\rho_{\Delta}(\beta)+k^{\Delta} \in \ell \big(\Gamma/\Delta) \ &\text{if and only if} \ \vDash \eta_{\Delta}^{k}(a, \beta):=\exists b \in U \big( v_{\Delta}(b-a)-\rho_{\Delta}(\beta)+k^{\Delta} \in \ell(\Gamma/\Delta)\big).
\end{align*}
Both formulas $\epsilon_{\Delta}^{k}(a,\beta)$ and $\eta_{\Delta}^{k}(a, \beta)$ are $\mathcal{L}(\ulcorner U \urcorner)$-definable and completely independent from the choice of $c$.\\
We conclude that $p(x)$ is a $\ulcorner U \urcorner$- definable type. Furthermore, for any possible realization $c \vDash \Sigma_{U}^{gen}(x)$ we obtain the same scheme of definition. Hence, there is a \emph{unique} extension $p(x)$ of $\Sigma_{U}^{gen}(x)$. 
\end{proof}

%\textcolor{red}{Technically I am being sloppy, the congruence imposes some interesting restruction only if $a$ lies in the torsor...but this is somehow clear for an expert. Should I make it explicit in the explanation? }
\begin{theorem}\label{densitydef} For every non-empty definable set $X \subseteq \mathfrak{M}$, there is a $acl^{eq}(\ulcorner X \urcorner)$-definable global type $p(x) \vdash x \in X$.
\end{theorem} 
\begin{proof}
Let $X \subseteq \mathfrak{M}$ be a $1$-definable set.\\
\begin{claim} There is a $1$-torsor $U$ such that $\ulcorner U \urcorner \in \acl^{eq}(\ulcorner X \urcorner)$ and the partial type:
\begin{equation*}
\Sigma^{gen}_{U}(x) \cup \{ x \in X\} \ \text{is consistent.}
\end{equation*}
\end{claim}
\begin{proof}
Let $\mathcal{F}$ be the family of closed balls $B$ such that $B\cap X \neq \emptyset$. We say that $B_{1} \sim B_{2}$ if and only if $B_{1}\cap X=B_{2} \cap X$. This is a $\ulcorner X \urcorner$-definable equivalence relation over $\mathcal{F}$.\\ 
Let $\pi: \mathcal{F} \rightarrow \mathcal{F}/\sim$, the natural $\ulcorner X \urcorner$-definable map sending a closed ball to its class $[B]_{\sim}$. For each class $\mu \in \mathcal{F}/\sim$ the set $\displaystyle{U_{\mu}=\bigcap_{B \in \mathcal{F}, \pi(B)=\mu}B}$ is a $\mu$-definable $1$-torsor. Moreover, for any $B \in \mathcal{F}$, $B \cap X=U_{\mu}\cap X$ if and only if $\pi(B)=\mu$. In particular, if $B$ is a proper closed subball of $U_{\mu}$, then $\pi(B)\neq \mu$.\\ 
Then set $\mathcal{F}/\sim$ admits a partial $\ulcorner X\urcorner$-definable order defined as:
\begin{equation*}
    \mu_{1} \triangleleft \mu_{2} \ \text{if and only} \ B_{1}\cap X \subsetneq B_{2}\cap X \ \text{where} \ \pi(B_{1})=\mu_{1} \ \text{and} \ \pi(B_{2})=\mu_{2}.
\end{equation*}
 The set $\big(\mathcal{F}/\sim, \triangleleft\big)$ is a tree with a maximal element $\mu_{0} \in \acl^{eq}(\ulcorner X \urcorner)$. This class is obtained by taking the projection of a ball $B_{0}$ such that $B_{0}\cap X=X$. By quantifier elimination (see Corollary \ref{QEregular}) $X$ is a finite union of nice sets, thus such ball $B_{0}$ exists.\\

For each $\mu \in \mathcal{F}/\sim$, we write $P(\mu)$ to denote the set of immediate predecessors of $\mu$ (if they exists). This is 
\begin{equation*}
\displaystyle{P(\mu):=\{ \beta \in \mathcal{F}/\sim \ | \ \beta \triangleleft \mu \ \text{and} \neg \exists z (\beta \triangleleft z \triangleleft \mu)\}}.
\end{equation*}
If $\Sigma_{U_{\mu}}^{gen}(x)\cup \{ x \in X\}$ is inconsistent then $P(\mu)$ is finite and has size at least $2$. Indeed,  $\Sigma_{U_{\mu}}^{gen}(x)\cup \{ x \in X\}$ is consistent if and only if 
\begin{equation*}
    \{ x \in U_{\mu} \} \cup \{ x \notin B \ | \ B\subseteq U \ \text{is a closed ball}\} \cup \{ x \in X\} \ \text{is consistent}. 
\end{equation*}
Hence, if $\Sigma_{U_{\mu}}^{gen}(x)\cup \{ x \in X\}$ is inconsistent, by compactness we can find finitely many disjoint closed balls $B_{1}, \dots, B_{k}$ such that $B_{i}\cap X \neq \emptyset$  and
\begin{equation*}
    U_{\mu}\cap X \subseteq \bigcup_{i\leq k} B_{i} \ 
\end{equation*}
Let $\beta_{i}=\pi(B_{i}) \triangleleft \mu$. Then $P(\mu)=\{\beta_{i} \ | \ i \leq k\} \subseteq \acl^{eq}(\ulcorner X \urcorner, \mu)$.\\

%\textcolor{red}{escribir mejor esto}
We now start looking for the $1$-torsor $U \in \acl^{eq}(\ulcorner X \urcorner)$ such that $\Sigma_{U}^{gen}(x) \cup \{ x \in X\}$ is consistent.\\
Let $\mu_{0}\in \acl^{eq}(\ulcorner X \urcorner)$ be the maximal element of $(\mathcal{F}/\sim, \triangleleft)$, if $\Sigma_{U_{\mu_{0}}}^{gen}(x) \cup \{ x \in X \}$ is consistent, the torsor $U_{\mu_{0}}$ satisfies the required conditions.  We may assume that $\Sigma_{U_{\mu_{0}}}^{gen}(x) \cup \{ x \in X \}$ is inconsistent, thus it has finitely many predecessors $P(\mu_{0})\subseteq \acl^{eq}(\ulcorner X \urcorner)$. For each $\beta \in P(\mu_{0})$ exactly one of the following cases hold:
\begin{enumerate}
    \item $\Sigma_{U_{\beta}}^{gen}(x) \cup \{ x \in X\}$ is consistent, then the torsor $U_{\beta}$ satisfies the required conditions of the claim; or
    \item $\Sigma_{U_{\beta}}^{gen}(x) \cup \{ x \in X\}$ is inconsistent, and $\beta$ has finitely many predecessors $P(\beta) \subseteq \acl^{eq}(\ulcorner X \urcorner, \beta)\subseteq \acl^{eq}(\ulcorner X \urcorner)$.
\end{enumerate}
By iterating this process for each of the predecessors, we build a discrete tree $\mathcal{T} \subseteq \mathcal{F}/\sim$ of finite ramification.\\
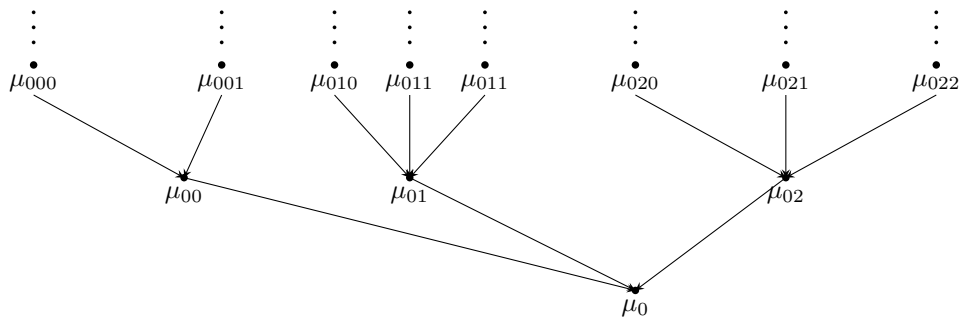
\begin{center}
\begin{tikzpicture}
%nodos
\coordinate[label=below:$\mu_{0}$] (A) at (7,-1);
\node at (A)[circle,fill,inner sep=1pt]{};
\coordinate[label=below:$\mu_{00}$] (B) at (1,0.5);
\node at (B)[circle,fill,inner sep=1pt]{};
\coordinate[label=below:$\mu_{01}$] (C) at (4,0.5);
\node at (C)[circle,fill,inner sep=1pt]{};
\coordinate[label=below:$\mu_{02}$] (D) at (9,0.5);
\node at (D)[circle,fill,inner sep=1pt]{};
%hijos B
\coordinate[label=below:$\mu_{000}$] (B1) at (-1,2);
\node at (B1)[circle,fill,inner sep=1pt]{};
\coordinate[label=below:$\mu_{001}$] (B2) at (1.5,2);
\node at (B2)[circle,fill,inner sep=1pt]{};
% sigue en B
\coordinate (E1) at (-1,2.3);
\node at (E1)[circle,fill,inner sep=0.5pt]{};
\coordinate (E2) at (-1,2.5);
\node at (E2)[circle,fill,inner sep=0.5pt]{};
\coordinate (E3) at (-1,2.7);
\node at (E3)[circle,fill,inner sep=0.5pt]{};
%end first line
\coordinate (F1) at (1.5,2.3);
\node at (F1)[circle,fill,inner sep=0.5pt]{};
\coordinate (F2) at (1.5,2.5);
\node at (F2)[circle,fill,inner sep=0.5pt]{};
\coordinate (F3) at (1.5,2.7);
\node at (F3)[circle,fill,inner sep=0.5pt]{};
%\end second line
%Hijos C
\coordinate[label=below:$\mu_{010}$] (C1) at (3,2);
\node at (C1)[circle,fill,inner sep=1pt]{};
\coordinate[label=below:$\mu_{011}$] (C2) at (4,2);
\node at (C2)[circle,fill,inner sep=1pt]{};
\coordinate[label=below:$\mu_{011}$] (C3) at (5,2);
\node at (C3)[circle,fill,inner sep=1pt]{};
%Sigue en C
%primer hijo
\coordinate (C11) at (3,2.3);
\node at (C11)[circle,fill,inner sep=0.5pt]{};
\coordinate (C22) at (3,2.5);
\node at (C22)[circle,fill,inner sep=0.5pt]{};
\coordinate (C33) at (3,2.7);
\node at (C33)[circle,fill,inner sep=0.5pt]{};
%segundo hijo
\coordinate (C10) at (4,2.3);
\node at (C10)[circle,fill,inner sep=0.5pt]{};
\coordinate (C20) at (4,2.5);
\node at (C20)[circle,fill,inner sep=0.5pt]{};
\coordinate (C30) at (4,2.7);
\node at (C30)[circle,fill,inner sep=0.5pt]{};
%tercer hijo
\coordinate (M1) at (5,2.3);
\node at (M1)[circle,fill,inner sep=0.5pt]{};
\coordinate (M2) at (5,2.5);
\node at (M2)[circle,fill,inner sep=0.5pt]{};
\coordinate (M3) at (5,2.7);
\node at (M3)[circle,fill,inner sep=0.5pt]{};
%Hijos D
\coordinate[label=below:$\mu_{020}$] (D1) at (7,2);
\node at (D1)[circle,fill,inner sep=1pt]{};
\coordinate[label=below:$\mu_{021}$] (D2) at (9,2);
\node at (D2)[circle,fill,inner sep=1pt]{};
\coordinate[label=below:$\mu_{022}$] (D3) at (11,2);
\node at (D3)[circle,fill,inner sep=1pt]{};
%Sigue en d
%primer hijo
\coordinate (C111) at (7,2.3);
\node at (C111)[circle,fill,inner sep=0.5pt]{};
\coordinate (C221) at (7,2.5);
\node at (C221)[circle,fill,inner sep=0.5pt]{};
\coordinate (C331) at (7,2.7);
\node at (C331)[circle,fill,inner sep=0.5pt]{};
%segundo hijo
\coordinate (C101) at (9,2.3);
\node at (C101)[circle,fill,inner sep=0.5pt]{};
\coordinate (C201) at (9,2.5);
\node at (C201)[circle,fill,inner sep=0.5pt]{};
\coordinate (C301) at (9,2.7);
\node at (C301)[circle,fill,inner sep=0.5pt]{};
%tercer hijo
\coordinate (M11) at (11,2.3);
\node at (M11)[circle,fill,inner sep=0.5pt]{};
\coordinate (M21) at (11,2.5);
\node at (M21)[circle,fill,inner sep=0.5pt]{};
\coordinate (M31) at (11,2.7);
\node at (M31)[circle,fill,inner sep=0.5pt]{};
%lineas

\draw[-stealth] (1,0.5)--(7,-1);
\draw[-stealth] (4,0.5)--(7,-1);
\draw[-stealth] (9,0.5)--(7,-1);
%lineas hijos de B
\draw[-stealth] (-1,1.6)--(1,0.5);
\draw[-stealth] (1.5,1.6)--(1,0.5);
%lineas hijos de C
\draw[-stealth] (3,1.6)--(4,0.5);
\draw[-stealth] (4,1.6)--(4,0.5);
\draw[-stealth] (5,1.6)--(4,0.5);
%lineas hijos D
\draw[-stealth] (7,1.6)--(9,0.5);
\draw[-stealth] (9,1.6)--(9,0.5);
\draw[-stealth] (11,1.6)--(9,0.5);
\end{tikzpicture}
\end{center}
Hence, it is sufficient to argue that every path in this tree is finite. Suppose by contradiction that a path is infinite, then we can find an infinite decreasing sequence $<\gamma_{i} \ | i \in \mathbb{N}>$ of elements in $\mathcal{F}/\sim$ such that $U_{\gamma_{0}}=U_{\mu_{0}}$, and:  
\begin{enumerate}\label{yuju}
   
    \item for each $i \in \mathbb{N}$, $P(\gamma_{i})$ is finite and of size at least $2$. Given $\eta_{1}\neq \eta_{2} \in P(\gamma_{i})$ we have that $U_{\eta_{1}}\cap U_{\eta_{2}}=\emptyset$. And
 $U_{\eta_{1}}$ is a proper subtorsor of $U_{\gamma_{i}}$.
 \item For each $\mu \in \mathcal{F}/\sim$, $U_{\mu}\subseteq U_{\gamma_{i}}$ for some $i \in \mathbb{N}$, or there is some $i \in \mathbb{N}$ such that $U_{\mu}\subseteq U_{\eta}$ for some $\eta \in P(\gamma_{i})\backslash \{ \gamma_{i+1}\}$. 
\end{enumerate}
%\textcolor{red}{incluir dibujo}\\
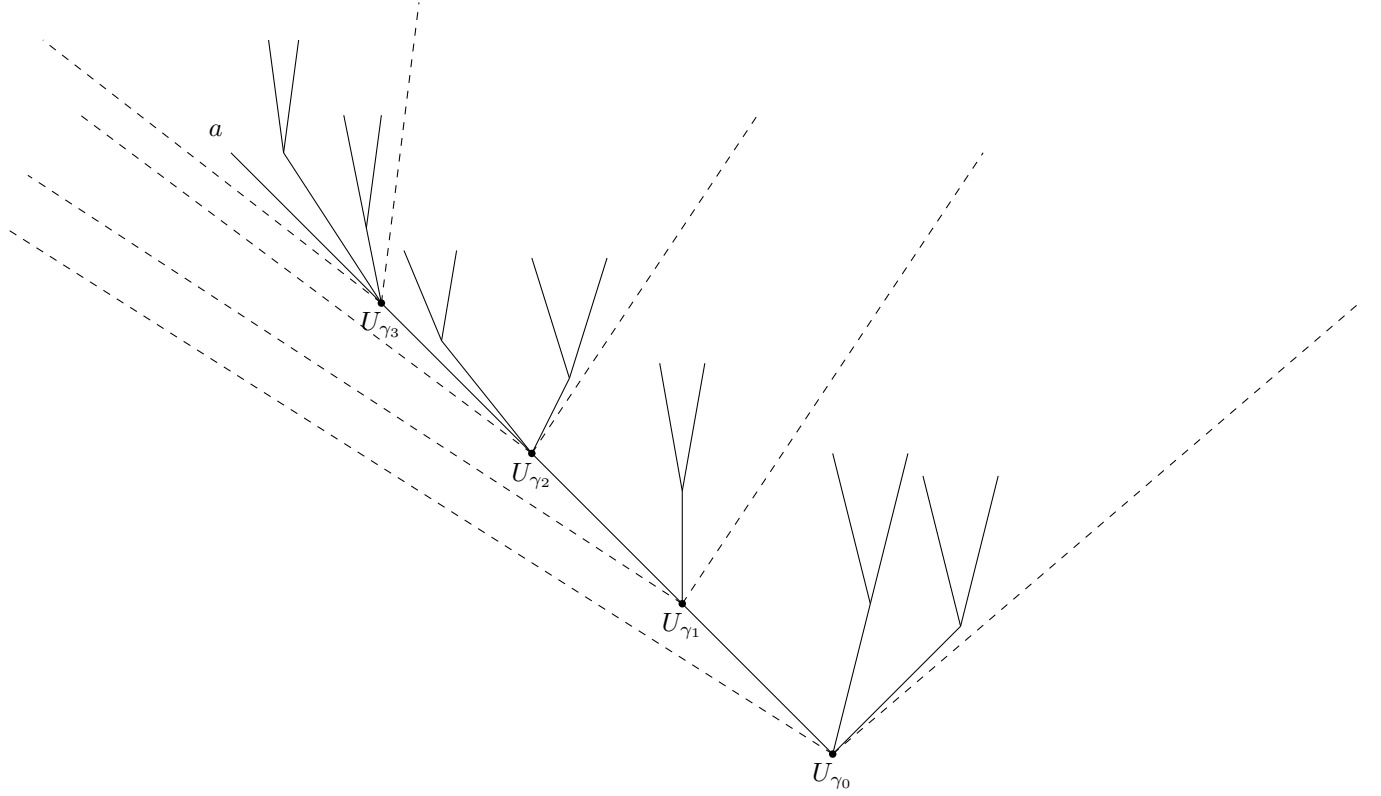
\begin{center}
\begin{tikzpicture}
\coordinate[label=below:$U_{\gamma_{0}}$] (A) at (5,-1);
\node at (A)[circle,fill,inner sep=1pt]{};
%cone
\draw[dashed](5,-1)--(-6,6);
\draw[dashed] (5,-1)--(12,5);
%subcones
\draw (5,-1)--(5.5,1);
\draw (5.5,1)--(5,3);
\draw (5.5,1)--(6,3);
\draw (5,-1)--(6.7,0.7);
\draw (6.7,0.7)--(7.2,2.7);
\draw (6.7,0.7)--(6.2,2.7);
\coordinate[label=below:$U_{\gamma_{1}}$] (B) at (3,1);
\node at (B)[circle,fill,inner sep=1pt]{};
%cone
\draw[dashed](3,1)--(-5.7,6.7);
\draw[dashed] (3,1)--(7,7);
\draw (3,1)--(3,2.5);
\draw (3,2.5)--(2.7,4.2);
\draw (3,2.5)--(3.3,4.2);
%\draw (3,1)--(6.5,3);
%\draw (7,1)--(7.5,3);
%\draw (5,-1)--(9,0.5);
%\draw (9,0.5)--(8.5,2.5);
%\draw (9,0.5)--(9.5,2.5);
\coordinate[label=below:$U_{\gamma_{2}}$] (C) at (1,3);
\node at (C)[circle,fill,inner sep=1pt]{};
%cone
\draw[dashed](1,3)--(-5,7.5);
\draw[dashed] (1,3)--(4,7.5);
%subcones
%subcone1
\draw(1,3)--(-0.2,4.5);
\draw(-0.2,4.5)--(-0.7,5.7);
\draw(-0.2,4.5)--(0,5.7);
%subcone2
\draw(1,3)--(1.5,4);
\draw(1.5,4)--(1,5.6);
\draw(1.5,4)--(2,5.6);
\coordinate[label=below:$U_{\gamma_{3}}$] (D) at (-1,5);
\node at (D)[circle,fill,inner sep=1pt]{};
%cone
\draw[dashed](-1,5)--(-5.5,8.5);
\draw[dashed] (-1,5)--(-0.5,9);
%subcone3
\draw(-1,5)--(-2.3,7);
\draw(-2.3,7)--(-2.5,8.5);
\draw(-2.3,7)--(-2.1,8.5);
%subcone 2
\draw(-1,5)--(-1.2,6);
\draw(-1.2,6)--(-1.5,7.5);
\draw(-1.2,6)--(-1,7.5);
%line for a
\draw (5,-1)--(-3,7);
\coordinate[label=below:$a$] (Z) at (-3.2,7.5);
\node at (Z){};
\end{tikzpicture}
\end{center}
By compactness we can find an element $a \in \mathfrak{M}$ such that $\displaystyle{a \in \bigcap_{i \in \mathbb{N}}U_{\gamma_{i}}}$. We note that $\displaystyle{\{ U_{\mu}\ | \ \mu \in \mathcal{F}/\sim\}}$ is a uniform definable family of $1$-torsors. Then we can define the set 
\begin{equation*}
D=\{ x \in K \ | \ \exists \mu \in \mathcal{F}/\sim  \ x\in U_{\mu} \ \text{and} \ a \notin U_{\mu}\}, 
\end{equation*}
but this set is not a finite union of nice sets (by the conditions in \ref{yuju}), which leads us to a contradiction.
\end{proof}
\begin{proof}
If $U$ is a closed $1$-torsor, we let $c$ be a realization of $\Sigma^{gen}_{U}(x) \cup \{ x \in X\}$. By Proposition \ref{uniqueclosed} the type $p(x)=\tp(c/\mathfrak{M})\vdash x \in X$ is $\ulcorner U \urcorner$-definable. The statement follows as $\ulcorner U \urcorner \in \acl^{eq}(\ulcorner X \urcorner)$.\\
We may assume that $U$ is an open torsor. We observe that for any realization $c \vDash \Sigma_{U}^{gen}(x)$ given $a\neq a' \in U(\mathfrak(M))$ we have that $v(c-a)=v(c-a')$.
Let $\pi:=\mathbb{N}\rightarrow \mathbb{N}\times \mathbb{N}_{\geq 1}$ be a fixed bijection. We build an increasing sequence of partial consistent types $\big(\Sigma_{k}(x) \ | \ k \in \mathbb{N})$ by induction:
\begin{itemize}
    \item \textbf{Stage $0$:} Let $\Sigma_{0}(x):=\Sigma^{gen}_{U}(x) \cup \{ x \in X\}$. 
    \item \textbf{Stage $k+1$:} Let $\pi(k)=(n,\ell)$. At this stage we decide the congruence modulo $\Delta_{n}+\ell\Gamma$. To simplify the notation we will assume that $\ell \geq 2$, otherwise the argument will follow in a similar manner (instead of working with $\ell (\Gamma/\Delta_{n})$ we argue with $\Gamma/\Delta_{n}$). Let
    \begin{equation*}
        \Lambda_{k}(x):=\Sigma_{k}(x) \cup \{ v_{\Delta_{n}}(x-a)-\rho_{\Delta_{n}}(\beta) \notin \ell(\Gamma/\Delta_{n}) \ | \ a \in U(\mathfrak{M}), \beta \in \Gamma\}.
    \end{equation*}
   If the partial type $\Lambda_{k}(x)$ is consistent, then we set $\Sigma_{k+1}(x)=\Lambda(x)$. Otherwise, let
    \begin{equation*}
    \displaystyle{A_{i}=\{\mu \in \Gamma/(\Delta_{n}+\ell\Gamma) \ | \ \Sigma_{k}(x) \cup \{ \pi_{\Delta_{n}}^{\ell}(v(x-a))=\mu \ | \ a \in U(\mathfrak{M})\} \ | \ \text{is consistent}\}}.
    \end{equation*}
     $A_{i}$ is a finite set.  We set
    \begin{equation*}\Sigma_{k+1}(x):= \Sigma_{k}(x) \cup \{ \pi_{\Delta_{n}}^{\ell}(v(x-a))=\mu \ | \ a \in U(\mathfrak{M})\}.
    \end{equation*}
\end{itemize}
Let $\mathcal{J}=\{ k \in \mathbb{N}_{\geq 1}\ | \ \Lambda_{k}(x)$ is inconsistent $\}$. 
\begin{claim}  For all $k \in \mathbb{N}$ we have that for any automorphism $\sigma \in Aut(\mathfrak{M}/\acl^{eq}(\ulcorner X \urcorner)$, $\sigma(\Sigma_{k}(x))=\Sigma_{k}(x)$ and if $k \in \mathcal{J}$ then $\sigma(A_{k})=A_{k}$. In particular, $A_{k} \subseteq \acl^{eq}(\ulcorner X \urcorner)$ for all $k \in \mathcal{J}$.
\end{claim}
\begin{proof} 
We proceed by induction, for the base case $k=0$ the statement follows because $\ulcorner U \urcorner \in \acl^{eq}(\ulcorner X \urcorner)$. We assume that for any $\sigma \in Aut(\mathfrak{M}/\ulcorner X \urcorner)$ we have that $\sigma(\Sigma_{k}(x))=\Sigma_{k}(x)$. We fix $\tau \in Aut(\mathfrak{M}/ \ulcorner X \urcorner)$ and we aim to show that $\tau(\Sigma_{k+1}(x))=\Sigma(\Sigma_{k+1}(x))$. If $\Lambda_{k}(x)$, then:
\begin{align*}
    \tau \big( \Sigma_{k+1}(x)\big)&=\tau \big( \Sigma_{k}(x) \cup \{ v_{\Delta_{n}}(x-a)-\rho_{\Delta_{n}}(\beta) \notin \ell(\Gamma/\Delta_{n}) \ | \ a \in U(\mathfrak{M}), \beta \in \Gamma\}\big)\\
    &=\Sigma_{k}(x) \cup \{ v_{\Delta_{n}}(x-\tau(a))-\rho_{\Delta_{n}}(\tau(\beta)) \notin \ell(\Gamma/\Delta_{n}) \ | \ a \in U(\mathfrak{M}), \beta \in \Gamma\}= \Sigma_{k+1}(x).
\end{align*}
If $\Lambda_{k}(x)$ is not consistent then $k \in \mathcal{J}$. And we first argue that $\tau(A_{k})=A_{k}$. By definition of $A_{k}$, given $\mu \in A_{k}$ then
\begin{equation*}
\Sigma_{k}(x) \cup \{ \pi_{\Delta_{n}}^{\ell}(v(x-a))=\mu \ | \ a \in U(\mathfrak{M})\} \ \text{is consistent},
\end{equation*}
because $\tau$ is an isomorphism,
\begin{align*}
\tau\big(\Sigma_{k}(x) \cup \{ \pi_{\Delta_{n}}^{\ell}(v(x-a))&=\mu \ | \ a \in U(\mathfrak{M})\})\\
&= \Sigma_{k}(x) \cup \{ \pi_{\Delta_{n}}^{\ell}(v(x-\tau(a)))=\tau(\mu) \ | \ a \in U(\mathfrak{M})\}\\
&=  \Sigma_{k}(x) \cup \{ \pi_{\Delta_{n}}^{\ell}(v(x-a))=\tau(\mu) \ | \ a \in U(\mathfrak{M})\}\ 
\text{is consistent},
\end{align*}
hence  $\tau(\mu) \in A_{i}$. We conclude that $\tau(A_{i})= A_{i}$, and because $\tau$ is an arbitrary element in $Aut(\mathfrak{M}/\acl^{eq}(\ulcorner X \urcorner))$ we conclude that $A_{i} \subseteq \acl^{eq}\big(\acl^{eq}(\ulcorner X \urcorner)\big)=\acl^{eq}(\ulcorner X \urcorner)$. In particular, for any $\mu \in A_{i}$, $\tau(\mu)=\mu$. Consequently,
\begin{align*}
    \tau \big( \Sigma_{k+1}(x)\big)&= \Sigma_{k}(x) \cup \{ \pi_{\Delta_{n}}^{\ell}(v(x-\tau(a)))=\mu \ | \ a \in U(\mathfrak{M})\}\\
    &=
    \Sigma_{k}(x) \cup \{ \pi_{\Delta_{n}}^{\ell}(v(x-a))=\mu \ | \ a \in U(\mathfrak{M})\}=\Sigma_{k+1}(x), \ \text{as required.}
\end{align*}
\end{proof}
Let $\displaystyle{\Sigma_{\infty}(x):= \bigcup_{k \in \mathbb{N}}\Sigma_{k}(x)}$, by construction this is a consistent partial type $\acl^{eq}(\ulcorner X \urcorner)$-definable and $\Sigma_{\infty}(x) \vdash x \in X$. By quantifier elimination, $\Sigma_{\infty}(x)$ determines a complete global type $p(x) \vdash x \in X$. This type $p(x)$ is $\acl^{eq}(\ulcorner X \urcorner)$-definable as $\Sigma_{\infty}(x)$ is.  
 \end{proof}
\end{proof}
 \subsection{Coding definable types}

 In this subsection we prove that any definable type can be coded in the stabilizer sorts $\mathcal{G}$. Let $\mathbf{x}=(x_{1},\dots,x_{k})$ be a tuple of variables in the main field sort. By quantifier elimination any definable type $p(\mathbf{x})$ over a model $K$ is completely determined by the boolean combinations formulas of the form: 
 \begin{enumerate}
 \item $Q_{1}(\mathbf{x})=0$,
 \item $v_{\Delta}(Q_{1}(\mathbf{x})) \leq v_{\Delta}(Q_{2}(\mathbf{x}))$, 
 \item $v_{\Delta}\left(\frac{Q_{1}(\mathbf{x})}{Q_{2}(\mathbf{x})}\right)-k_{\Delta} \in n (\Gamma/\Delta)$, 
 \item $v_{\Delta} \left(\frac{Q_{1}(\mathbf{x})}{Q_{2}(\mathbf{x})}\right) = k_{\Delta}$. 
 \end{enumerate}
 where $ Q_{1}(\mathbf{x}), Q_{2}(\mathbf{x}), \in K[X_{1},\dots,X_{k}]$, $n \in \mathbb{N}_{\geq 2}$, $\Delta \in RJ(\Gamma)$, $k \in \mathbb{Z}$ and $k_{\Delta}=k \cdot 1_{\Delta}$ where $1_{\Delta}$ is the minimal element of $\Gamma/\Delta$ if it exists. We will approximate such a type by considering for each $l \in \mathbb{N}$ the definable vector space  $D_{l}/ I_{l}$, where $D_{l}$ is the set of polynomials of degree at most $l$ and $I_{l}$ is the subspace of $D_{l}$ of polynomials $Q(\mathbf{x})$ such that $Q(\mathbf{x})=0$ is a formula in $p(\mathbf{x})$. The formulas of the second kind, essentially give $D_{l}/ I_{l}$ a valued vector space structure with all the coarsened valuations, while the formulas of the third and forth kind simply impose some binary relations in the linear order $\Gamma(D_{l}/ I_{l})$. This philosophy reduces the problem of coding definable types into finding a way to code the possible valuations that could be induced over some power of $K$ while taking care as well for the congruences.\\
 
 The following is \cite[Lemma 3.3]{will}. 
  \begin{fact}\label{basis} Let $K$ be any field. Let $V$ be a subspace of $K^{n}$ then $V$ can be coded by a tuple of $K$, and $V$ and $K^{n}/V$ have a $\ulcorner V \urcorner$-definable basis. 
 \end{fact}
 We start by coding the $\mathcal{O}$-submodules of $K^{n}$.
 \begin{lemma}\label{modulesok} Let $K \vDash T$ and $M \subseteq K^{n}$ be a definable $\mathcal{O}$-submodule. Then the code $\ulcorner M \urcorner$ can be coded in the stabilizer sorts. 
 \end{lemma}
 \begin{proof}
 Let $V^{+}$ the span of $M$ and $V^{-}$ the maximal $K$-subspace of $K^{n}$ contained in $M$. By Fact \ref{basis} the subspaces $V^{+}$ and $V^{-}$ can be coded by a tuple $c$ in $K$, and the quotient vector space $V^{+}/V^{-}$ admits a $c$-definable basis. Hence $V^{+}/V^{-}$ can be identified over $c$ with some power $K^{m}$. And $\ulcorner M\urcorner$ is interdefinable over $c$ with the code of the image $M/V^{-}$ in $K^{m}$. But this image is an $\mathcal{O}$-submodule of $K^{m}$ of type $(I_{1},\dots,I_{m}) \in \mathcal{F}^{m}$ so it admits a code in $B_{m}(K)/ Stab_{(I_{1},\dots, I_{n})}$.  So $M$ admits a code in the stabilizer sorts, as required.\end{proof}

 \begin{definition}\label{valuedrelation} [Valued relation] Let $K\vDash T$, and $\Gamma$ be its value group.  Let $V$ be some finite dimensional $K$-vector space and $R \subseteq V \times V$ be a definable subset that defines a total pre-order.We say that $R$ is a \emph{valued relation} if there is an interpretable valued vector space structure  $(V, \Gamma(V), val, +)$ in $K$ such that $(v,w) \in R$ if and only if $val(v) \leq val(w)$. 
\end{definition}
In fact, given a relation $R \subseteq V \times V$ that defines a total pre-order satisfying that:
\begin{itemize}
 \item for all $v,w \in V \ $ $(v, v+w) \in R$ or $(w,v+w)\in R$,
 \item for all $v \in V \ $ $(v,v) \in R$,
 \item for all $v,w \in V$ and $\alpha \in K$, if $(v,w)\in R$ then $(\alpha v, \alpha w) \in R$.
 \end{itemize}
 We can define an equivalence relation $E_{R}$ over $V$ as $\displaystyle{E_{R}(v,w) \leftrightarrow (v,w) \in R \wedge (w,v) \in R}$. \\
The set $\Gamma(V)= V/E_{R}$ is therefore interpretable in $K$ and we call it as the \emph{linear order induced by $R$}. Let $val: V \rightarrow \Gamma(V)$ be the canonical projection map that sends each vector to its class. We can naturally define an action of  $\Gamma$ over $\Gamma(V)$ as:
\begin{center}
$+:\begin{cases}
\Gamma \times \Gamma(V) &\rightarrow \Gamma(V)\\
(\alpha, [v]_{E_{R}}) &\mapsto [a v]_{E_{R}}, \ \text{where $a \in K$ satisfying $v(a)=\alpha$.}
\end{cases}$

\end{center}
This is a well defined map by the third condition imposed over $R$. The structure $(V, \Gamma(V), val, +)$ is an interpretable valued vector space structure over $V$ and we refer to it as the \emph{valued vector space structure induced by $R$.}
 \begin{lemma}\label{basismax} Let $K$ be a model of $T$ and let $R \subset K^{n} \times K^{n}$ be a binary relation inducing a valued vector space structure  $(K^{n}, \Gamma(K^{n}), val, +)$ over $K^{n}$. Then we can find a basis $\{ v_{1},\dots, v_{n}\}$ of $K^{n}$ such that:
 \begin{enumerate}
 \item It is a separated basis for $val$, this is given any set of coefficients $\lambda_{1}, \dots, \lambda_{n} \in K$, 
 \begin{align*}
 val \big( \sum_{i \leq n } \lambda_{i}v_{i} \big)=\min \{ v(\lambda_{i})+val(v_{i}) \ | \ i \leq n \}. 
 \end{align*}
 \item For each $i \leq n$, $\gamma_{i}=val(v_{i}) \in dcl^{eq}(\ulcorner R \urcorner)$. 
 \end{enumerate}
 \end{lemma}
 \begin{proof}
 Because the statement we are proving is first order expressible, by Fact \ref{maxext} we may assume that $K$ is maximal.  We proceed by induction on $n$. For the base case, note that $K= span_{K}\{1\}$ then $\gamma= val(1) \in dcl^{eq}(\ulcorner R \urcorner)$. We assume the statement for $n$ and we want to prove it for $n+1$. Let $W=K^{n}\times \{0\}$, $val_{W}=v\upharpoonright_{W}$, $\Gamma(W)=\{ val(w) \ | \ w \in W\}$, and $R_{W}= R \cap (W \times W)$. Then $(W, \Gamma(W), val_{W}, +)$ is a valued vector space structure over $W$ and $\ulcorner R_{W} \urcorner \in \dcl^{eq}(\ulcorner R \urcorner)$. The subspace $W$ admits an $\emptyset$-definable basis, so it can be canonically identified with $K^{n}$. By the induction hypothesis we can find $\{ w_{1}, \dots, w_{n}\}$ a separated basis of $W$ such that $val_{W}(w_{i})\in \dcl^{eq}(\ulcorner R_{W} \urcorner) \subseteq  \dcl^{eq}(\ulcorner R \urcorner)$. As $W$ is finite dimensional it is maximal by Lemma \ref{product}. By Fact \ref{optclosed} $W$ has the optimal approximation property in $K^{n+1}$. We can therefore define the valuation over the quotient space $K^{n+1}/W$ as follows:
 \begin{center}
$val_{K^{n+1}/W}:\begin{cases}
 \big(K^{n+1}/W \big)&\rightarrow \Gamma(K^{n})\\
 v+W &\mapsto \max\{ val(v+w_{0}) \ | \ w_{0} \in W \}.
\end{cases}$
\end{center}
Define $R_{K^{n+1}/W}= \{ (w_{1}+W, w_{2}+W) \ | \ val_{K^{n+1}/W}(w_{1}+W) \leq val_{K^{n+1}/W}(w_{2}+W)\}$, which is a valued relation over the quotient space $K^{n+1}/W$. As $K^{n+1}/W= K^{n+1}/( K^{n} \times \{ 0 \})$ is definably isomorphic over the $\emptyset$-set to $K$,  we can find a non zero coset $v+W$ such that $val_{K^{n+1}/W}(v+W) \in dcl^{eq}(\ulcorner R_{K^{n+1}/W} \urcorner) \subseteq dcl^{eq}(\ulcorner R\urcorner)$. Let $w^{*} \in W$ be a vector where the maximum  of $\{ val_{K^{n+1}/W}(v+w) \ | \ w \in W\}$ is attained, i.e. $val_{K^{n+1}/W}(v+W)=val(v+w^{*})$. It is sufficient to show that $\{ w_{1}, \dots, w_{n}, v+w^{*}\}$ is a separated basis for $K^{n+1}$. \\
 Let $\alpha \in K$, we show that for any $w \in W$ $\displaystyle{val((v+w^{*})+ \alpha w)= \min \{ val(v+w^{*}), val(\alpha w)\}}$.  
If $val(v+w^{*}) \neq val(\alpha w)$ then $val((v+w^{*})+ \alpha w)= \min \{ val(v+w^{*}), val(\alpha w)\}$. So let's assume that $\gamma= val(v+w^{*})=val(\alpha w)$, by the ultrametric inequality $val((v+w^{*}) + \alpha w) \geq \gamma$. By the maximal choice of $w^{*}$, we have that $val((v+w^{*})+\alpha w) \leq val(v+w^{*})=\gamma$. So $val((v+w^{*})+ \alpha w)= \min \{ val(v+w^{*}), val(\alpha w)\}$ as required. \\
 \end{proof}
\begin{theorem}\label{codingvalued} Let $K$ be a model of $T$ and $\Gamma$ its value group. Let $R$ be a definable valued relation over $K^{n}$ and $(K^{n}, \Gamma(K^{n}), val, +)$  be the valued vector space structure induced by $R$. Then $\ulcorner R \urcorner$ is interdefinable with a tuple of elements in the stabilizer sorts and there is an $\ulcorner R \urcorner$- definable bijection $\Gamma(K^{n})$ and finitely many disjoint copies of $\Gamma$ (all contained in $\Gamma^{s}$, where $s$ is the number of $\Gamma$-orbits over $\Gamma(K^{n})$).
\end{theorem}
\begin{proof}
As the statement that we are trying to prove is first order expressible, without loss of generality we may assume that $K$ is maximal. Let $R$ be a valued relation over $K^{n}$ and let $(K^{n}, \Gamma(K^{n}), val, +)$  be the valued vector space structure induced by $R$. By Lemma \ref{basismax}, we can find a separated basis $\{ v_{1}, \dots, v_{n}\}$ of $K^{n}$, such that for each $i \leq n$, $val(v_{i}) \in dcl^{eq}(\ulcorner R \urcorner)$.  Let $\{\gamma_{1}, \dots, \gamma_{s}\} \subseteq \{ val(v_{i}) \ | \ i \leq n\}$ be a complete set of representatives of the orbits of $\Gamma$ over the linear order $\Gamma(K^{n})$, this is:
\begin{align*}
 \Gamma(K^{n})= \dot{\bigcup_{i \leq s}} \Gamma+\gamma_{i}.
 \end{align*}
  For each $i \leq s$, we define $B_{i}:= \{ x \in K^{n} \ | \ val(x) \geq \gamma_{i}\}$. Each $B_{i}$ is an $\mathcal{O}$-submodule of $K^{n}$, so by Lemma \ref{modulesok}  $\ulcorner B_{i}\urcorner$ is interdefinable with a tuple in the stabilizer sorts. The valued vector space structure over $K^{n}$ is completely determined by the closed balls containing $0$, and each of these ones is of the form $\alpha B_{i}$ for some $\alpha \in K$ and $i \leq s$.  Thus the code $\ulcorner R \urcorner$ is interdefinable with the tuple $(\ulcorner B_{1}\urcorner, \dots, \ulcorner B_{s} \urcorner)$. We conclude that $\ulcorner R \urcorner$ can be coded in the stabilizer sorts.\\
For the second part of the statement, consider the map:
\begin{center}
$\begin{cases}
f:\dot{\bigcup}_{i \leq s} \Gamma+\gamma_{i} &\rightarrow \Gamma^{s}\\
\alpha+ \gamma_{i} &\mapsto (0, \dots, 0, \underbrace{\alpha}_{i-\text{th coordinate}},0, \dots,0).\\
\end{cases}$
\end{center}
As $\{\gamma_{1},\dots,\gamma_{s}\} \subseteq dcl^{eq}(\ulcorner R \urcorner)$ this is a  $\ulcorner R \urcorner$-definable bijection between $\Gamma(K^{n})$ to finitely many disjoint copies of $\Gamma$, contained in $\Gamma^{s}$.
\end{proof}

 \begin{theorem}\label{codingdeftypes} Let $p(\mathbf{x})$ be a definable global type in $\mathfrak{M}^{n}$. Then $p(\mathbf{x})$ can be coded in $\mathcal{G} \cup \Gamma^{eq}$. 
 \end{theorem}
 \begin{proof}
 Let $p(\mathbf{x})$ be a definable global type, and let $K$ be a small model where $p(\mathbf{x})$ is defined. Let $q(\mathbf{x})= p(\mathbf{x})\upharpoonright_{K}$ it is sufficient to code $q(\mathbf{x})$. \\
 
 For each $\ell \in \mathbb{N}$ let $D_{\ell}$ be the space of polynomials in $K[X_{1},\dots, X_{n}]$ of degree less or equal than $\ell$. This is a finite dimensional $K$-vector space with an $\emptyset$-definable basis. Let $I_{\ell}:= \{ Q(\bar{x}) \in D_{\ell} \ | \ Q(\bar{x})=0 \in q(\bar{x})\}$, this  is a subspace of $D_{\ell}$.  Let $R_{\ell}:= \{(Q_{1}(\mathbf{x}), Q_{2}(\mathbf{x}))\in D_{\ell} \times D_{\ell} \ | \ v(Q_{1}(\mathbf{x})) \leq v(Q_{2}(\mathbf{x})) \in q(\mathbf{x})\}$, this relation induces a valued vector space structure on the quotient space $V_{\ell}=D_{\ell}/I_{\ell}$. Let $( V_{\ell}, \Gamma(V_{\ell}), val_{\ell}, +_{\ell})$ be the valued vector space structure induced by $R_{\ell}$ over $V_{\ell}$.\\
For each $\Delta \in RJ(\Gamma)$ and $k \in \mathbb{Z}$, a formula of the form $v_{\Delta}(Q_{1}(\mathbf{x}))=v_{\Delta}(Q_{2}(\mathbf{x}))+k_{\Delta}$ determines a definable relation  $\phi_{\Delta}^{k} \subseteq \Gamma(V_{\ell})^{2}$, defined as: 
\begin{align*}
(val_{\ell}(Q_{1}(\mathbf{x})), val_{\ell}(Q_{2}(\mathbf{x})) \in \phi_{\Delta}^{k} \ \text{if and only if} \ v_{\Delta}(Q_{1}(\mathbf{x}))=v_{\Delta}(Q_{2}(\mathbf{x}))+k_{\Delta} \in q(\mathbf{x}). 
\end{align*}
Similarly, for each $\Delta \in RJ(\Gamma)$, $k \in \mathbb{Z}$ and $n \in \mathbb{N}_{\geq 2}$ we consider the definable binary relation $\psi_{\Delta}^{n,k} \subseteq \Gamma(V_{\ell})^{2}$ determined as:
\begin{align*}
(val_{\ell}(Q_{1}(\mathbf{x})), val_{\ell}(Q_{2}(\mathbf{x}))) \in \psi_{\Delta}^{k,n} \ \text{if and only if} \  v_{\Delta}(Q_{1}(\mathbf{x}))-v_{\Delta}(Q_{2}(\mathbf{x}))+k_{\Delta}\in n(\Gamma/\Delta) \in q(\mathbf{x}).
\end{align*}
Likewise, for each $\Delta \in RJ(\Gamma)$  and $k \in \mathbb{Z}$ we consider the definable binary relations $\theta^{k}_{\Delta} \subseteq \Gamma(V_{\ell})^{2}$ defined as:
\begin{align*}
(val_{\ell}(Q_{1}(\mathbf{x})), val_{\ell}(Q_{2}(\mathbf{x}))) \in \theta^{k}_{\Delta} \ \text{if and only if} \  v_{\Delta}(Q_{1}(\mathbf{x}))< v_{\Delta}(Q_{2}(\mathbf{x}))+k_{\Delta} \in q(\mathbf{x}).
\end{align*}
Let
\begin{equation*}
 \mathcal{S}_{\ell}=\{ \phi_{\Delta}^{k} \ | \ \Delta \in RJ(\Gamma), k \in \mathbb{Z}\} \cup \{\psi_{\Delta}^{k,n} \ | \ \Delta \in RJ(\Gamma), k \in \mathbb{Z}, n \in \mathbb{N}_{\geq 2}\} \cup \{ \theta_{\Delta}^{k} \ | \ \Delta \in RJ(\Gamma), k \in \mathbb{Z} \}
 \end{equation*}
   We denote as $\mathcal{V}_{\ell}=(V_{\ell}, \Gamma(V_{\ell}), val_{\ell}, +_{\ell}, \mathcal{S}_{\ell})$ the valued vector space over $V_{\ell}$ with the enriched structure over the linear order $\Gamma(V_{\ell})$. By quantifier elimination (see Corollary \ref{QEregular}), the type $q(\mathbf{x})$ is completely determined by boolean combinations of formulas of the form:
 \begin{itemize}
 \item $Q_{1}(x)=0$,
 \item $v_{\Delta}(Q_{1}(\mathbf{x})) < v_{\Delta}(Q_{2}(\mathbf{x}))$, 
 \item $v_{\Delta}\left(\frac{Q_{1}(\mathbf{x})}{Q_{2}(\mathbf{x})}\right)-k_{\Delta} \in n (\Gamma/\Delta)$,
 \item $v_{\Delta} \left(\frac{Q_{1}(\mathbf{x})}{Q_{2}(\mathbf{x})}\right) = k_{\Delta}$. 
 \end{itemize}
 where $ Q_{1}(\mathbf{x}), Q_{2}(\mathbf{x}), \in K[X_{1},\dots,X_{k}]$, $n \in \mathbb{N}_{\geq 2}$, $\Delta \in RJ(\Gamma)$, $k \in \mathbb{Z}$ and $k_{\Delta}=k \cdot 1_{\Delta}$ where $1_{\Delta}$ is the minimum positive element of $\Gamma/\Delta$ if it exists. 
Hence the type $p(\mathbf{x})$ is entirely determined (and determines completely) by the sequence of valued vector spaces with enriched structure over the linear order $(\mathcal{V}_{\ell} \ | \  \ell \in \mathbb{N})$. \\
 By Fact \ref{basis} for each $\ell \in \mathbb{N}$ we can find codes $\ulcorner I_{\ell} \urcorner$ in the home sort for the $I_{\ell}'s$. After naming these codes, each quotient space $V_{\ell}=D_{\ell}/I_{\ell}$ has a definable basis, so it can be definably identified with some power of $K$. Therefore, without loss of generality we may assume that the underlying set of the valued vector space with enriched structure $V_{\ell}$ is some power of $K$. By Theorem \ref{codingvalued}, the relation $R_{\ell}$ admits a code $\ulcorner R_{\ell} \urcorner$ in the stabilizer sorts. Moreover, there is a  $\ulcorner R_{\ell} \urcorner$ definable bijection $f: \Gamma(V_{\ell}) \rightarrow \Gamma^{s}$, where  $s \in \mathbb{N}_{\geq 2}$ is the number of $\Gamma$-orbits over $\Gamma(V_{\ell})$.\\
 In particular, for each $\Delta \in RJ(\Gamma)$, $n \in \mathbb{N}$ and $k \in \mathbb{Z}$ the definable relations $\phi_{\Delta}^{k}$, $\psi_{\Delta}^{k,n}$ and $\theta_{\Delta}^{k}$ are interdefinable over $\ulcorner R \urcorner$ with  $f(\phi_{\Delta}^{k})$, $f(\psi_{\Delta}^{k,n})$ and $f(\theta_{\Delta}^{k})$, all subsets of $\Gamma^{2s}$. Consequently, the type $q(\mathbf{x})$ can be coded in the sorts $\Gamma \cup \Gamma^{eq}$, as every definable subset $D$ in some power of $\Gamma$ admits a code in $\Gamma^{eq}$.
 \end{proof}
 \begin{theorem}\label{weakEI1} Let $K$ be a valued field of equicharacteristic zero, residue field algebraically closed and value group with bounded regular rank. Then $K$ admits weak elimination of imaginaries in the language $\mathcal{L}_{\mathcal{G}}$, where the stabilizer sorts are added. 
 \end{theorem}
 \begin{proof}
 By Theorem \ref{weakEI}, $K$ admits weak elimination of imaginaries down to the sorts $\mathcal{G} \cup \Gamma^{eq}$, where $\mathcal{G}$ are the stabilizer sorts. In fact, Hrushovski's criterion requires us to verify the following two conditions:
 \begin{enumerate}
 \item the density of definable types, this is Theorem \ref{densitydef}, and
 \item the coding of definable types, this is Theorem \ref{codingdeftypes}. 
 \end{enumerate}
 By Corollary \ref{stableembeddedness} the value group $\Gamma$ is stably embedded. By Theorem \ref{WEIbounded},  the ordered abelian group with bounded regular rank $\Gamma$ admits weak elimination of imaginaries once one adds the quotient sorts,
 \begin{equation*}
 \{ \Gamma/\Delta \ | \ \Delta \in RJ(\Gamma)\} \cup \{ \Gamma/\Delta+ n\Gamma \ | \ \Delta \in RJ(\Gamma), n \in \mathbb{N}_{\geq 2}\}.
 \end{equation*}
 We conclude that $K$ admits weak elimination of imaginaries down to the stabilizer sorts $\mathcal{G}$. 
 \end{proof}

  \section{ Elimination of imaginaries for henselian valued field with dp-minimal value group}\label{dpminimal}

 Let $(K,v)$ be a henselian valued field of equicharacteristic zero, residue field algebraically closed and dp-minimal value group. We see $K$ as a multisorted structure in the language $\hat{\mathcal{L}}$ extending the language $\mathcal{L}_{\mathcal{G}}$ (described in Definition \ref{defstab}), where the value group is equipped with the language $\mathcal{L}_{dp}$ described in Subsection \ref{language}. Let $\mathcal{I}'$ be the complete family of $\mathcal{O}$-submodules of $K$ described in Fact \ref{completemodules}. From now on we fix a complete family $\displaystyle{\mathcal{F}=\mathcal{I}' \backslash \{ 0, K\}}$.\\

We refer to these sorts as the \emph{stabilizer sorts} and we denote their union 
\begin{equation*}
\mathcal{G}= K \cup k \cup \Gamma \cup \{\Gamma/\Delta \ | \ \Delta \in RJ(\Gamma)\} \cup \{ B_{n}(K)/ \Stab(I_{1},\dots,I_{n}) \ | \ (I_{1},\dots, I_{n}) \in \mathcal{I}^{n}\} 
\end{equation*}
\begin{remark} If we work with the complete family $\mathcal{I}$ of end-segments given by Remark \ref{completedp}, each of $\mathcal{O}$-modules in $\mathcal{I}$ is definable over the empty set. In this setting we are adding a finite set of constants $\Omega_{n}$ in $\Gamma$ choosing representatives of $n \Gamma$ in $\Gamma$ for each $n \in \mathbb{N}$. The results we obtain in this section will hold in the same manner if we work with this language instead.
\end{remark}
Our main goal is the following Theorem.
\begin{theorem} Let $K$ be a henselian valued field of equicharacteristic zero, residue field algebraically closed and dp-minimal value group. Then $K$ eliminates imaginaries in the language $\hat{\mathcal{L}}$, where the stabilizer sorts are added. 
\end{theorem}

 \begin{definition} We say that a multi-sorted first order theory $T$ \emph{codes finite sets}  if for every model $M \vDash T$, and every finite subset $S \subseteq M$, the code $\ulcorner S \urcorner$ is interdefinable with a tuple of elements in $M$.
\end{definition}
The following is a folklore fact (see for example \cite{Poizat}).
\begin{fact} \label{all} Let $T$ be a complete multi-sorted theory.  If $T$ has weak elimination of imaginaries and codes finite sets then $T$ eliminates imaginaries.
\end{fact}

In view of Theorem \ref{weakEI1} and Fact \ref{all} it is only left to show that any finite set can be coded in $\mathcal{G}$. 

 \begin{definition} 
 \begin{enumerate}
 \item An equivalence relation $E$ on a set $X$ is said to be \emph{proper} if it has at least two different equivalence classes. It is said to be \emph{trivial} if for any $x,y \in X$ we have $E(x,y)$ if and only if $x=y$. 
\item  A finite set $F$ is \emph{primitive over $A$} if there is no proper non-trivial $( \ulcorner F \urcorner \cup A)$-definable equivalence relation on $F$. If $F$ is primitive over $\emptyset$ we just say that it is \emph{primitive}. 
\end{enumerate} 
\end{definition} 
To code finite sets we need numerous smaller results. This section is organized as follows:
\begin{enumerate}
\item Subsection \ref{residual}: we analyze the stable and stably embedded multi-sorted structure $VS_{k,C}$, consisting of the $k$-vector spaces $\red(s)$, where $s$ is some $\mathcal{O}$-lattice definable over $C$, an arbitrary imaginary set of parameters. This structure has elimination of imaginaries by results of Hurshovski in \cite{Grupoids}.
\item Subsection \ref{germen}: we introduce the notion of germ of a definable function $f$ over a definable type $p$. We prove that germs can be coded in the stabilizer sorts. 
 \item Subsection \ref{lemas1}: later we show that the code of any $\mathcal{O}$-submodule $M \subseteq K^{n}$ is interdefinable with the code of its projection to the last coordinate and the germ of the function describing each of the fibers. We show that the same statement holds for torsors. 
 \item Subsection \ref{lemas2}: we prove several results on coding finite sets in the one-dimensional case, e.g. if $F$ is a primitive finite set of $1$-torsors then it can be coded in $\mathcal{G}$. 
 \item Subsection \ref{primitive}: we carry a simultaneous induction to prove that any finite set $F \subseteq \mathcal{G}^{r}$ can be coded in the stabilizer sorts, and any definable function $f: F \rightarrow \mathcal{G}$ admits a code in the stabilizer sorts. 
 \item Subsection \ref{todo}: We state the result on full elimination of imaginaries down to the stabilizer sorts. 
 \end{enumerate}
\subsection{The multi-sorted structure of $k$-vector spaces}\label{residual}

By Corollary \ref{stableembeddedness} the residue field $k$ is stably embedded and it is a strongly minimal structure, because it is an algebraically closed field. This enables us to construct, over any imaginary base set of parameters $C$, a part of the structure that naturally inherits stability-theoretic properties from the residue field. Given a $\mathcal{O}$-lattice $s \subseteq K^{n}$ we have $\red(s)=s/\mathcal{M}s$  is a $k$-vector space.
\begin{definition} For any imaginary set of parameters $C$, we let $VS_{k,C}$ be the many-sorted structure whose sorts are the $k$ vector spaces $\red(s)$ where $s\subseteq K^{n}$ is an $\mathcal{O}$-lattice of rank $n$ definable over $C$. Each sort $\red(s)$ is equipped with its $k$-vector space structure. In addition, $VS_{k,C}$ has any $C$-definable relation on products of the sorts. 
\end{definition}
\begin{definition} 
%\begin{enumerate}
 A definable set $D$ is said to be \emph{internal to the residue field} if there is a finite set of parameters $F \subseteq \mathcal{G}$ such that $D \subseteq \dcl^{eq}(kF)$. 
%\item Given a set of parameters $C$, let $VS_{k,C}$ be the multi-sorted structure of $C$- definable sets internal to the residue field.
%\end{enumerate}
\end{definition}
Each of the structures $\red(s)$ is internal to the residue field, and the parameters needed to witness the internality lie in $\red(s)$, so in particular each of the $k$-vector spaces $\red(s)$ is stably embedded. The entire multi-sorted structure $VS_{k,C}$ is also stably embedded and stable, and in this subsection we will prove that it eliminates imaginaries. 
\begin{notation} We recall that given an $\mathcal{O}$-submodule $M$ of $K$, we write $S_{M}$ to denote the end-segment induced by $M$, i.e.$\displaystyle{\{ v(x) \ | \ x \in M\}}$.
\end{notation}

We recall some definitions from \cite{Grupoids} to show that $VS_{k,C}$ eliminates imaginaries. 
\begin{definition} Let $t$ be a theory of fields (possibly with additional structure). A \emph{$t$-linear structure} $\mathcal{A}$ is a structure with a sort $k$ for a model of $t$, and addional sorts $(V_{i} \ | \ i \in I)$ denoting finite-dimensional vector spaces. Each $V_{i}$ has (at least) a $k$-vector space structure, and $dimV_{i}< \infty$. We assume that:
\begin{enumerate}
\item $k$ is stably embedded,
\item the induced structure on $k$ is precisely given by $t$,
\item The $V_{i}$ are closed under tensor products and duals. 
\end{enumerate}
Moreover, we say it is \emph{flagged} if for any finite dimensional vector space $V$ there is a flitration $V_{1} \subseteq V_{2} \subseteq \dots \subseteq V_{n}=V$ by subspaces, with $dimV_{i}=i$ and $V_{i}$ is one of the distinguished sorts. 
\end{definition}
The following is \cite[Lemma 5.2]{Grupoids}.
\begin{lemma}\label{Superudi} If $k$ is an algebraically closed field and $\mathcal{A}$ is a flagged $k$-linear structure, then $\mathcal{A}$ admits elimination of imaginaries.
\end{lemma}
\begin{notation} Let $A$ be an $\mathcal{O}$-module. Let $\mathcal{M}A=\{ xa \ | \ x \in \mathcal{M}, \  a \in A\}$ we denote as $\red(A)$ the quotient $\mathcal{O}$-module $A/\mathcal{M} A$. 
\end{notation}
We observe that $\red(A)=A/\mathcal{M}A$ is canonically isomorphic to $A \otimes_{\mathcal{O}} k$. 

\begin{fact}\label{tensoriso} Let $A \subseteq K^{n}$ and $B \subseteq K^{m}$ be $\mathcal{O}$-lattices. Then $\red(A) \otimes_{k} \red(B)$ can be canonically identified with $\red(A \otimes_{\mathcal{O}} B)$. 
\end{fact}
\begin{proof}
This is a straightforward computation and it is left to the reader. 
\end{proof}
\begin{remark}\label{tensor}Given $A \subseteq K^{n}$ and $B \subseteq K^{m}$ $\mathcal{O}$-lattices, there is some $\mathcal{O}$-lattice $C \subseteq K^{mn}$ such that $A \otimes_{\mathcal{O}} B$ is canonically identified with $C$. This isomorphism induces as well a one to one correspondence between $\red(A \otimes_{\mathcal{O}} B)$ and $\red(C)$. 
\end{remark}
\begin{proof}
Given $K^{n}$ and $K^{m}$ two vector spaces, the tensor product $K^{n} \otimes K^{m}$ is a $K$ vector space whose basis is $\{ e_{i} \otimes e_{j} \ | \ i \leq n, j \leq m \}$ and it is canonically identified with $K^{nm}$, via a linear map $\phi$ that extends the bijection between the basis sending $e_{i} \otimes e_{j}$ to $ e_{ij}$. Given $A \subseteq K^{n}$ and $B \subseteq K^{m}$ $\mathcal{O}$-lattices, then $A \otimes_{\mathcal{O}} B$ is an $\mathcal{O}$-lattice of $K^{n} \otimes K^{m}$ and we denote as $C= \phi(A \otimes_{\mathcal{O}} B)$. This map induces as well an identification between $\red(A \otimes_{\mathcal{O}} B)$ and $\red(C)$ such that the following map commutes:
\begin{center}
\begin{tikzcd}
A \otimes_{\mathcal{O}} B\arrow[r, "\phi"] \arrow[d, "\red"]
& C \arrow[d, "\red" ] \\ \red(A \otimes_{\mathcal{O}}B) \arrow[r, "i " ]
&\red(C) \end{tikzcd}
\end{center}
\end{proof}
\begin{fact}\label{isores} Let $A \subseteq K^{n}$ and $B \subseteq K^{m}$ be $\mathcal{O}$-lattices. Then there is an isomorphism 
\begin{equation*}
\phi: \red(Hom_{\mathcal{O}}(A,B))\rightarrow Hom_{k}(\red(A), \red(B)),
\end{equation*}
where for any $f \in Hom_{\mathcal{O}}(A,B)$ and $a \in A$:
\begin{center}
$  \phi(f+ \mathcal{M} Hom_{\mathcal{O}}(A,B)):\begin{cases}
\red(A)&\rightarrow \red(B)\\
  a+\mathcal{M} A &\mapsto f(a)+\mathcal{M}B.
 \end{cases}$
 \end{center}
\end{fact}
\begin{proof}
This is a straightforward computation and it is left to the reader. 
\end{proof}
\begin{remark}\label{dual} Given an $\mathcal{O}$-lattice $A \subseteq K^{n}$, then $Hom_{\mathcal{O}}(A, \mathcal{O})$ can be canonically identified with some $\mathcal{O}$-module $C$ of $K^{n}$. So there is a correspondence between $\red(Hom_{\mathcal{O}}(A,\mathcal{O}))$ and $\red(C)$.
\end{remark}
\begin{proof} Let $A$ be an $\mathcal{O}$-lattice of $K^{n}$. By linear algebra $K^{n}$ can be identified with its dual space $(K^{n})^{*}$. Let
\begin{align*}
A^{*}=\{ T \in (K^{n})^{*} \ | \ \text{for all } \ a \in A, \ T(a) \in \mathcal{O} \}.
\end{align*}
$A^{*}$ is canonically identified with $Hom_{\mathcal{O}}(A, \mathcal{O})$ via the map that sends a transformation $T$ to $T\upharpoonright_{A}$. Also $A^{*}$ is isomorphic to some $\mathcal{O}$-lattice $C$ of $K^{n}$, as there is a canonical isomorphism between $K^{n}$ and its dual space. So we have a definable $\mathcal{O}$-isomorphism $\phi$ between $Hom_{\mathcal{O}}(A, \mathcal{O})$ and $C$, and
%\begin{proof}
%Let $T \in A^{*}$, by linearity   must be an $\mathcal{O}$- homomorphism from $A$ to $\mathcal{O}$. Now let $f \in Hom_{\mathcal{O}}(A, \mathcal{O})$, we claim the existence of some element $g \in A^{*}$ such that $g\upharpoonright_{A}=f$. Since $A$ is an $\mathcal{O}$-lattice of rank $n$, there is some basis $\{ b_{1},\dots,b_{n} \}$ of $K^{n}$ such that $A=b_{1}\mathcal{O}+ \dots + b_{n}\mathcal{O}$. Let $g:= K^{n} \rightarrow K^{n}$ be the linear transformation obtained by extending the bijection $f:= \{ b_{1}, \dots, b_{n} \} \rightarrow \{ f(b_{1}), \dots, f(b_{n})\}$. Clearly $g$ is a linear transformation extending $f$.
%\end{proof}
this correspondence induces an identification $\hat{\phi}$ between $\red(Hom_{\mathcal{O}}(A,O))$ and $\red(C)$ making the following diagram commute:
\begin{center}
\begin{tikzcd}
Hom_{\mathcal{O}}(A,\mathcal{O})\arrow[r, "\phi"] \arrow[d, "\red"]
& C \arrow[d, "\red" ] \\ \red(Hom_{\mathcal{O}}(A,\mathcal{O})) \arrow[r, "\bar{\phi} " ]
&\red(C) \end{tikzcd}
\end{center}
\end{proof}

%\textcolor{red}{Hay que revisar si esto funciona para Latices o para modulos}
\begin{remark}\label{flagged} Let $A \subseteq K^{n}$ be an $\mathcal{O}$-lattice. There is a sequence of $\mathcal{O}$-lattices  $<A_{i} \ | \ i \leq n>$ such that $<\red(A_{i}) \ | \ i \leq n>$ is a flag of $\red(A)$ and for each $i \leq n$, $\ \ulcorner A_{i} \urcorner \in \dcl^{eq}(\ulcorner A \urcorner).$
\end{remark}
\begin{proof}
We proceed by induction on $n$, the base case is trivial. Let $A \subseteq K^{n+1}$, and $\pi_{n+1}: K^{n+1} \rightarrow K$ be the projection into the last coordinate. Let $B\subseteq K^{n}$ be the $\mathcal{O}$-lattice such that 
\begin{equation*}
\ker(\pi_{n+1})= B \times \{0\}= A \cap (K^{n} \times \{0\}).  
\end{equation*}
We observe that $\ulcorner B \urcorner, \ulcorner \pi_{n+1}(A) \urcorner \in \dcl^{eq}(\ulcorner A \urcorner)$. By Corollary \ref{basismodule} $B$ is a direct summand of $A$, so we have the exact splitting sequence
\begin{equation*}
0 \rightarrow B \rightarrow A \rightarrow \pi_{n+1}(A) \rightarrow 0.
\end{equation*}
Consequently, 
\begin{align*}
0 &\rightarrow \mathcal{M} B \rightarrow \mathcal{M} A \rightarrow \mathcal{M}  \pi_{n+1}(A) \rightarrow 0 \ \text{and}\\
0 &\rightarrow \red(B) \rightarrow \red(A) \rightarrow \red( \pi_{n+1}(A) )\rightarrow 0
\end{align*}
are exact sequences that split. By the induction hypothesis, there is a sequence $\{0\} \leq A_{1} \leq \dots \leq A_{n}=B$ such that $<\red(A_{i}) \ | \ i \leq n>$ is a flag of $\red(B)$, $dim(\red(A_{i}))=i$ and $\ulcorner A_{i} \urcorner \in \dcl^{eq}(\ulcorner B\urcorner) \subseteq \dcl^{eq}(\ulcorner A\urcorner)$. Let $A_{n+1}= A$, the sequence $<A_{i} \ | \ i \leq n+1>$ satisfies the required conditions. 
\end{proof}

\begin{theorem}\label{EIinternalresidue}Let $C \subseteq K^{eq}$, then $VS_{k,C}$ has elimination of imaginaries. 
\end{theorem}
\begin{proof}
The sorts $\red(s)$ where $s$ is a $\mathcal{O}$-lattice of $K^{n}$ and $\dcl^{eq}(C)$-definable form the multi-sorted structure $\VS_{k,C}$. Each $\red(s)$ carries a $k$-vector space structure. $VS_{k,C}$ is closed under tensor operation by Remark \ref{tensor} and Fact \ref{tensoriso}. It is closed under duals by Remark \ref{dual} and Fact \ref{isores}. By Remark \ref{flagged} each sort $\red(s)$ where $s$ is an $\mathcal{O}$-lattice admits a complete filtration by $C$-definable vector spaces. Therefore, $\VS_{k,C}$ is a  flagged $k$-linear structure, so the statement is immediate consequence of \ref{Superudi}.
\end{proof}

 \subsection{Germs of functions}\label{germen}

 In this subsection we show how to code the germ of a definable function $f$ over a definable type $p(\mathbf{x})$ in the stabilizer sorts. 
\begin{definition}
Let $T$ be a complete first order theory and $M \vDash T$. Let $B \subseteq M$ and $p$ be 
 a $B$-definable type whose solution set is $P$.
 Let $f$ be an $M$-definable function whose domain contains $P$. 
 Suppose that $f=f_{c}$ is defined by the formula $\phi(x,y,c)$
  (so $f_{c}(x)=y$). We say that $f_{c}$ and $f_{c'}$ \emph{have the same germ on $P$} 
  if the formula $f_{c}(x)=f_{c'}(x)$ lies in $p$. By the definability of $p$ 
  the equivalence relation $E_{\phi}(c, c')$ that states $f_{c}$ and $f_{c'}$ 
  have the same germ on $P$ is definable over $B$.The \emph{germ of $f_{c}$ on $P$} is defined to be the class of $c$ 
under the equivalence relation $E_{\phi}(y,z)$, which is an element in
 $M^{eq}$. We write $germ (f, p)$ to denote the code
  for this equivalence class. 
\end{definition}
\begin{definition} Let $p$ be a global type definable over $B$ and let $C$ a set of parameters. We say that a realization $a$ of $p$ is sufficiently generic over $BC$ if $a \vDash p\upharpoonright_{BC}$. 
\end{definition}

We start proving some results that will be required to show how to code the germs of a definable function $f$ over a definable type $p$ in the stabilizer sorts.\\
Let $U \subseteq K$ be a $1$-torsor, we recall Definition \ref{generico}, where we defined the $\ulcorner U \urcorner$-definable partial type.
\begin{align*}
\Sigma_{U}^{gen}(x) = \{ x\in U \} \cup \{ x \notin B \ | \ B \subsetneq U \ \text{is a proper subtorsor of $U$}\}.
\end{align*}
We refer to this type as \emph{the generic type of $U$}.\\
When considering complete extensions of $\Sigma_{U}^{gen}(x)$ one finds an important distinction between the closed and the open case. In Proposition \ref{uniqueclosed} we proved that whenever $U$ is a closed $1$-torsor then $\Sigma_{U}^{gen}(x)$ admits a unique complete extension. The open case inherits a higher level of complexity.
\begin{proposition}\label{open} Let $U$ be an open $1$-torsor, then any completion of the generic type of $U$ is $\ulcorner U \urcorner$-definable. 
\end{proposition}
\begin{proof}
Let $c \vDash \Sigma_{U}^{gen}(x)$, it is sufficient to prove that $p(x)=\tp(c/\mathfrak{M})$ is $\ulcorner U \urcorner$-definable.  Let $\Delta \in RJ(\Gamma)$, $\ell \in \mathbb{N}$, $\beta \in \Gamma$, $k \in \mathbb{Z}$. First, we observe that for any $a,a' \in U(\mathfrak{M})$ we have that $v(c-a)=v(c-a')$, because $c$ realizes the generic type of $U$. In particular, 
\begin{equation*}
    v_{\Delta}(x-a)- \rho_{\Delta}(\beta)+k^{\Delta} \in \ell (\Gamma/\Delta) \in p(x) \ \text{if and only if} v_{\Delta}(x-a')- \rho_{\Delta}(\beta)+k^{\Delta} \in \ell (\Gamma/\Delta) \in p(x).
\end{equation*}
Pick some element $\delta \in \Gamma$ such that $\rho_{\Delta}(\delta)=k^{\Delta}$, and let $\mu=\pi_{\Delta}^{\ell}(v(c-a)+\delta) \in \dcl^{eq}(\emptyset)$. Then, for any $a \in U(\mathfrak{M})$ we have:
\begin{equation*}
    v_{\Delta}(x-a)- \rho_{\Delta}(\beta)+k^{\Delta} \in \ell (\Gamma/\Delta) \in p(x) \ \text{if and only if} \ \pi_{\Delta}^{\ell}(\beta)=\mu.
\end{equation*}
If $a \notin U(\mathfrak{M})$, then $v(c-a)=v(b-a)$ for any $b \in U(\mathfrak{M})$. Hence:
\begin{equation*}
    v_{\Delta}(x-a)- \rho_{\Delta}(\beta)+k^{\Delta} \in \ell (\Gamma/\Delta) \in p(x) \ \text{if and only if} \ \underbrace{\exists b \in U \big( v_{\Delta}(b-a)- \rho_{\Delta}(\beta)+k^{\Delta} \in \ell (\Gamma/\Delta)\big)}_{\phi(a,\beta)}.
\end{equation*}
Let
\begin{equation*}
\psi(a,\beta):= \big(a \in U \wedge \pi_{\Delta}^{\ell}(\beta)=\mu\big) \vee \big( a \notin U \wedge \phi(a,\beta)).
\end{equation*}
Hence,
\begin{equation*}
   v_{\Delta}(x-a)- \rho_{\Delta}(\beta)+k^{\Delta} \in \ell (\Gamma/\Delta) \in p(x) \ \text{if and only if} \ \psi(a,\beta).  
\end{equation*}
Note that $\ulcorner \psi(x,z)\urcorner \in \dcl^{eq}(\ulcorner U \urcorner)$.\\
We continue showing a $\dcl^{eq}(\ulcorner U \urcorner)$-definable scheme for the coset formulas. Let $a \in U(\mathfrak{M})$ and consider the definable end-segment $S_{a}=\{ v(x-a) \ | \ x \in U\}$. For any $a \neq a'$, $S_{a}=S_{a'}$, thus we write $S$ to denote this set. Note that $\ulcorner S \urcorner \in \dcl^{eq}(\ulcorner U\urcorner)$ and let $\Delta \in RJ(\Gamma)$. We recall that we denote by $S_{\Delta}$ the set $\rho_{\Delta}(S)$, which is a definable end-segment in $\Gamma/\Delta$. If $S_{\Delta}$ has a minimum element $\gamma \in \dcl^{eq}(S_{_{\Delta}})\subseteq \dcl^{eq}(\ulcorner S \urcorner)$ then for any $a \in U(\mathfrak{M})$ we have:
\begin{equation*}
    v_{\Delta}(x-a)=\rho_{\Delta}(\beta)+k^{\Delta} \in p(x) \ \text{if and only if} 
   \ \rho_{\Delta}(\beta)+k^{\Delta}=\gamma.
\end{equation*}
If $S_{\Delta}$ does not have a minimum, then for any $a \in U(\mathfrak{M})$,
\begin{equation*}
    v_{\Delta}(x-a)=\rho_{\Delta}(\beta)+k^{\Delta} \in p(x) \ \text{if and only if} 
    \ \beta \neq \beta.
\end{equation*}
Finally, for $a \notin U(\mathfrak{M})$ we have that $v(c-a)=v(b-a)$ for any $b \in U(\mathfrak{M})$, therefore for $\Delta \in RJ(\Gamma)$ and $k \in \mathbb{Z}$ we have that:
\begin{equation*}
  v_{\Delta}(x-a)=\rho_{\Delta}(\beta)+k^{\Delta} \in p(x) \ \text{if and only if} \ \exists b \in U \big( v_{\Delta}(b-a)=\rho_{\Delta}(\beta)+k^{\Delta}\big).
\end{equation*}
Consequently, for each quantifier free formula $\phi(x,y)$, we have shown the existence of a formula $d\phi(y)$ such that $\ulcorner d\phi(y) \urcorner \in \dcl^{eq}(\ulcorner U \urcorner)$ and $\phi(x,b) \in p(x)$ if and only if $\vDash d\phi(b)$. By quantifier elimination (see Corollary \ref{QEregular}), the type $p(x)$ is completely determined by the quantifer free formulas, we conclude that $p(x)$ is $\ulcorner U \urcorner$-definable. 
\end{proof}
\begin{corollary}\label{alldef} Let $U$ be a definable $1$-torsor, then each completion $p(x)$ of $\Sigma_{U}^{gen}(x)$ is $\ulcorner U \urcorner$-definable.
\end{corollary}
\begin{proof}
This follows immediately by combining Proposition \ref{uniqueclosed} and Proposition \ref{open}.
\end{proof}
\begin{proposition}\label{stabadd} Let $M\subseteq \mathfrak{M}$ be a definable $\mathcal{O}$-module and let $p(x)$ be a global type containing the generic type of $M$. Then $p(x)$ is stabilized additively by $M(\mathfrak{M})$ i.e. if $c$ is a realization of $p(x)$ and $a \in M(\mathfrak{M})$ then $a+c$ is a realization of $p(x)$. 
\end{proposition}
\begin{proof}
Let $c$ be a realization of the type $p(x)$, $a \in M(\mathfrak{M})$ and $d=c+a$. As $\Sigma_{M}^{gen}(x) \subseteq p(x)$, $c \in A$ and $c \notin U$ for any proper subtorsor $U \subseteq A$. First we argue that $d \vDash \Sigma_{M}^{gen}(x)$. Because $A$ is a $\mathcal{O}$-submodule of $K$ (in particular closed under addition) we have $d \in A$. And if there is a subtorsor $U \subsetneq A$ such that $d \in U$, then $c \in -a+U \subsetneq A$ contradicting that $c \vDash \Sigma_{A}^{gen}(x)$. For any $\Delta \in RJ(\Gamma)$, element $z \in M(\mathfrak{M})$  and  realization $b \vDash \Sigma_{A}^{gen}(x)$ we have  $v_{\Delta}(z-b)=v_{\Delta}(b)$. Thus for any $n \in \mathbb{N}$ and $\beta \in \Gamma/\Delta$:
\begin{align*}
v_{\Delta}(b-z)-\beta \in n(\Gamma/\Delta) \ \text{if and only if} \ v_{\Delta}(b)-\beta \in n (\Gamma/\Delta).
\end{align*}
We conclude that $d$ and $c$ must satisfy the same congruence and coset formulas, because $c$ is a realization of the generic type of $M$, $a \in M(\mathfrak{M})$ and  $v_{\Delta}(d)= v_{\Delta}(c+a)= v_{\Delta}(c-(-a))=v_{\Delta}(c)$. 
\end{proof}
\begin{corollary}\label{diff} Let $M \subseteq \mathfrak{M}$ be a definable $\mathcal{O}$-module. Let $p(x)$ be a global type containing the generic type of $M$.
Let $a \in M(\mathfrak{M})$, then $a$ is the difference of two realizations of $p(x)$ i.e. we can find $c,d \vDash p(x)$ such that $a=c-d$.
\end{corollary}
\begin{proof}
Let $c$ be a realization of $p(x)$ and fix $a \in M(\mathfrak{M})$. By Proposition \ref{stabadd}, $d=c-a$ is also a realization of $p(x)$. The statement now follows because $a=c-d$.
\end{proof}

\begin{proposition}\label{nicebasis}Let $(I_{1},\dots, I_{n}) \in \mathcal{I}^{n}$, for every $\mathcal{O}$-module $M$ of type $(I_{1},\dots, I_{n})$ we can find a type $p_{M}(\bar{x}_{1},\dots,\bar{x}_{n})\in S_{n\times n}(K)$ such that:
\begin{enumerate}
\item $p_{M}(\mathbf{x})$ is definable over $\ulcorner M \urcorner$, 
\item A realization of $p_{M}(\mathbf{x})$ is a matrix representation of $M$. This is if $(\bar{d}_{1},\dots, \bar{d}_{n}) \vDash p_{M}(\mathbf{x})$ then $[\bar{d}_{1},\dots,\bar{d}_{n}]$ is a representation matrix for $M$. 
\end{enumerate}
\end{proposition}
\begin{proof}
Let $\mathfrak{M}$ be the monster model. \\

\textit{Step $1$: We define a partial type $\Sigma_{(I_{1},\dots, I_{n})}$ satisfying condition $i)$ and $ii)$ for the canonical module $C_{(I_{1},\dots,I_{n})}$. Such a type is  left-invariant under the action of $\Stab(I_{1},\dots, I_{n})$.}\\

We consider the set $\mathcal{J}=\{ (i,j) \ | \ 1 \leq i,j \leq n\}$ , and we equip it with a linear order defined as:
\begin{align*}
(i,j)<(i',j') \ \text{if and only if} \ j<j' \ \text{or} \ j=j' \wedge i'>i.
\end{align*}
And we consider an enumeration of $\mathcal{J}=\{ v_{1},\dots,v_{n^{2}}\}$ such that $v_{1}<v_{2}<\dots<v_{n^{2}}$. By Proposition \ref{stab}
\begin{align*}
\Stab_{(I_{1}, \dots, I_{n})}= \{ ((a_{i,j})_{1 \leq i,j \leq n}  \in B_{n}(K) \ | \ a_{ii} \in \mathcal{O}_{\Delta_{S_{I_{i}}}}^{\times}\ \wedge a_{ij} \in Col(I_{i},I_{j})  \   \text{for each $1 \leq i < j \leq n$ }\}.
\end{align*}
Hence, for each $1 \leq m \leq n^{2}$ let:
\begin{center}
$
p_{v_{m}}(x)=
\begin{cases}
tp(0) \ &\text{if $v_{m}=(i,j)$ where $1\leq j < i \leq n$,}\\
\Sigma_{\mathcal{O}_{\Delta_{S_{I_{i}}}}}^{gen}(x) \ &\text{if $v_{m}=(i,i)$ for some $1 \leq i \leq n$},\\
\Sigma^{gen}_{Col(I_{i},I_{j})}(x) \ &\text{if $v_{m}=(i,j)$ where $1 \leq i < j \leq n$. }
\end{cases}
$
\end{center}
Consider the partial definable type $\displaystyle{\Sigma_{C_{(I_{1},\dots,I_{n})}}= p_{v_{n^{2}}} \otimes \dots \otimes p_{v_{1}}}$. Given a realization of this type $(b_{v_{n^{2}}},\dots,b_{v_{1}}) \vDash \Sigma_{C_{(I_{1},\dots,I_{n})}}$ let
\begin{center}
$
B= \begin{bmatrix}
b_{v_{n}} & b_{v_{2n}} & \dots & b_{v_{n^{2}}} \\
\vdots & \vdots& \vdots& \vdots\\
b_{v_{1}} & b_{v_{n+1}}& \dots &b_{v_{n(n-1)+1}}
\end{bmatrix}
$
\end{center}
 By construction $B$ is an upper triangular matrix such that $(B)_{i,j} \in Col(I_{i},I_{j})$ for $1 \leq i < j \leq n$ and $(B)_{ii} \in \mathcal{O}_{\Delta_{S_{I_{i}}}}^{\times}$, thus its column vectors constitute a basis for the canonical module. 
To check left invariance, it is sufficient to take $A \in \Stab_{(I_{1},\dots,I_{n})}(\mathfrak{M})$ and argue that for each $1 \leq m \leq n^{2}$ the element $(AB)_{v_{m}}$ is a realization of generic type $p_{v_{m}}$ sufficiently generic over $\mathfrak{M} \cup \{ (AB)_{v_{k}} \ | \ k < m \}$. 
Suppose that $v_{m}=(i,j)$, then $\displaystyle{(AB)_{v_{m}}=(AB)_{ij}= \sum_{k=i}^{j}a_{ik}b_{kj}=a_{ii}b_{ij}+\dots+a_{ij}b_{jj}}$. \\

In the fixed enumeration we guarantee that $b_{ij}$ is chosen sufficiently generic over $\mathfrak{M} \cup \{ b_{kj} \ | \ i < k \leq j\}$. For each $i< k\leq j$, $a_{ik}b_{kj}\in Col(I_{i},I_{j})$, thus we have that $v(a_{ii}b_{ij})=v((AB)_{ij})$.
Consequently, $(AB)_{ij}$ is a realization of $p_{v_{m}}$ generic over $\mathfrak{M}$ together with all the elements $b_{kl}$ where $(k,l)$ appears earlier in the enumeration than $(i.j)$.\\
\textit{Step $2$: For any $\mathcal{O}$-module $M \subseteq K^{n}$ of type $(I_{1},\dots,I_{n})$ there is an $\ulcorner M \urcorner$-definable type $p_{M}$, such that any realization of $p_{M}$ is a representation matrix for $M$.}\\
Let $T= \mathfrak{M}^{n} \rightarrow \mathfrak{M}^{n}$
 be a linear transformation whose representation matrix is upper triangular and $T$ sends the canonical module $C_{(I_{1},\dots,I_{n})}$ to $M$. And let $\Sigma_{M}=T(\Sigma_{C_{(I_{1},\dots,I_{n})}})$, its definition is independent from the choice of $T$, because given two linear transformations with upper triangular representation matrices $[T]$ and $[T']$ which send the canonical $\mathcal{O}$-module of type $(I_{1},\dots,I_{n})$ to $M$, we have that $[T']^{-1}[T] \in \Stab_{(I_{1},\dots,I_{n})}$ and the type $\Sigma_{C_{(I_{1},\dots,I_{n})}}$ is left invariant under the action of such group. Thus, $\Sigma_{M}$ is $\ulcorner M \urcorner$-definable and given $B \vDash \Sigma_{M}$, the type $\tp(B/ \mathfrak{M})$ is still $\ulcorner M \urcorner$-definable by Corollary \ref{alldef}. 
\end{proof}
%\textcolor{blue}{Tom: Review this Thereom: something key could be going on with strong germs.}
\begin{theorem}\label{codinggerm} Let $X$ be a definable subset of $K^{n}$ and let $p(\mathbf{x}) \vdash \mathbf{x} \in X$ be a global type definable over $\ulcorner X \urcorner$. Let 
$f= X \rightarrow \mathcal{G}$ be a definable function. Then the $p$-germ of $f$  is coded in $\mathcal{G}$ over $\ulcorner X \urcorner$. 
\end{theorem}
\begin{proof}
 We first assume that $f: X \rightarrow B_{n}(K)/ \Stab_{(I_{1}, \dots, I_{n})}$. Let 
 \begin{equation*}
 B= \dcl_{\mathcal{G}}(\germ(f, p), \ulcorner X \urcorner)= \dcl^{eq}(\germ(f, p), \ulcorner X \urcorner) \cap \mathcal{G}.
 \end{equation*}
 %(This is the set of elements in the stabilizer sorts that are definable over $\{ germ(f, p), \ulcorner X \urcorner \}$). \\
Suppose that $f$ is $c$-definable, and let $q=\tp(c/B)$ and $Q$ be its set of realizations. Fix some $c' \in Q$. We denote by $f'$ the function obtained by replacing the parameter $c$ by $c'$ in the formula defining $f$. Let $M$ be a small model containing $Bcc'$. \\

\textit{Step $1$: For any realization $\mathbf{a} \vDash p(\mathbf{x})\upharpoonright_{M}$ we have $f(\mathbf{a})=f'(\mathbf{a})$.}\\
Let $\mathbf{a}$ be a realization of $p(\mathbf{x})\upharpoonright_{M}$.
 Let $u_{f(\mathbf{a})}(\mathbf{y})$ be the definable type over $\ulcorner f(\mathbf{a}) \urcorner$ given by Proposition \ref{nicebasis}. Given any realization $\mathbf{d}=(\bar{d}_{1},\dots,\bar{d}_{n}) \vDash u_{f(\mathbf{a})}(\mathbf{y})$, $[\bar{d}_{1},\dots,\bar{d}_{n}]$ is a representation matrix for the module $f(\mathbf{a})$. In particular, $f(\mathbf{a})=\rho_{(I_{1},\dots,I_{k})}(\bar{d}_{1},\dots,\bar{d}_{k}))\in \dcl^{eq}(\mathbf{d})$. Let $\mathbf{d}$ be a realization of $u_{f(\mathbf{a})}(\mathbf{y}) \upharpoonright_{M}$  and let $r(\mathbf{x},\mathbf{y})= \tp(\mathbf{a},\mathbf{d}/ M)$, then the type $r(\mathbf{x}, \mathbf{y})$ is $B$-definable, and therefore $B$-invariant.\\

Because $\tp(c/B)=\tp(c'/B)$ we can find an automorphism $\sigma \in Aut(\mathfrak{M}/B)$ sending $c$ to $c'$. Then $u_{f'(\mathbf{a})}= \sigma(u_{f(\mathbf{a})})$, which is a definable type over $\ulcorner f'(\mathbf{a}) \urcorner$. Let $\mathbf{d}'$ be a realization of $u_{f'(\mathbf{a})}\upharpoonright_{M}$. Let $r'(\mathbf{x},\mathbf{y})=\tp(\mathbf{a},\mathbf{d}'/M)$, then $\sigma(r(\mathbf{x},\mathbf{y}))=r'(\mathbf{x},\mathbf{y})=r(\mathbf{x},\mathbf{y})$ by the $B$-invariance of $r(\mathbf{x},\mathbf{y})$, so $\tp(\mathbf{a},\mathbf{d}/M)=\tp(\mathbf{a},\mathbf{d}'/M)$.  Since $f(\mathbf{a}) \in \dcl^{eq}(\mathbf{d})$ and $f'(\mathbf{a}) \in \dcl^{eq}(\mathbf{d}')$, we must have that $\tp(\mathbf{a},f(\mathbf{a})/M)=\tp(\mathbf{a},f'(\mathbf{a})/M)$ and since $f$ and $f'$ are both definable over $M$ this implies that $f(\mathbf{a})=f'(\mathbf{a})$.\\

\textit{Step $2$: The  $\germ(f,p)$ is coded in the stabilizer sorts $\mathcal{G}$ over $\ulcorner X \urcorner$.  }\\
Firstly, note that for any $\mathbf{a} \vDash p(\mathbf{x}) \upharpoonright_{Bcc'}$ it is the case that $f(\mathbf{a})=f'(\mathbf{a})$. In fact, by Step $1$ $f(x)=f'(x) \in \tp(a/M)$ and $f(x)=f'(x)$ is a formula in $\tp(a/Bcc')$. Then $f$ and $f'$ both have the same $p$-germ. Since $p(\mathbf{x})$ is definable over $B=\dcl_{\mathcal{G}}(B)$ the equivalence relation  $E$ stating that $f$ and $f'$ have both the same $p$-germ is $B$-definable. Since for any realization  $\mathbf{a} \vDash p(\mathbf{x}) \upharpoonright_{Bcc'}$ it is the case that $f(\mathbf{a})=f'(\mathbf{a})$, the class $E(x, c)$ is $B$-invariant, therefore $\germ(f,p)$ is definable over $B=\dcl_{\mathcal{G}}(B)$. \\

We continue arguing that the statement for $f= X \rightarrow B_{n}(K)/ \Stab_{(I_{1},\dots, I_{n})}$ is sufficient to conclude the entire result. For each $\Delta \in RJ(\Gamma)$ there is a canonical isomorphism $\Gamma/\Delta \cong K^{\times}/ \mathcal{O}_{\Delta}^{\times}$, where $\mathcal{O}_{\Delta}$ is the valuation ring of the coarsened valuation $v_{\Delta}$ induced by $\Delta$. The functions whose image lie in $\Gamma/\Delta$ are being considered in the previous case, because $\Stab(\mathcal{O}_{\Delta})=\mathcal{O}_{\Delta}^{\times}$. By Proposition \ref{torsormodule} any definable function $f= X \rightarrow k=\mathcal{O}/\mathcal{M}$ can be seen as a function whose image lies in $B_{2}(K)/ \Stab_{(\mathcal{M},\mathcal{O})}$.\\
It is only left to consider the case where the target set is $K$. The proof follows in a very similar manner as the case for $\displaystyle{f:X \rightarrow B_{n}(K)/\Stab_{(I_{1},\dots,I_{n})}}$. Let $\mathbf{a} \vDash p(x)$. We let $a\vDash p(x)$ and let $r(x,y):= \tp(\mathbf{a},f(\mathbf{a})/M)$, this is a $B$-definable type by Theorem \ref{codingdeftypes}, in particular $B$-invariant. Likewise, $r'(x,y):=\tp(\mathbf{a},f'(\mathbf{a})/M)$ is $B$-invariant, thus $\tp(\mathbf{a},f(\mathbf{a})/M)=\tp(\mathbf{a},f'(\mathbf{a})/M)$. Since $f$ and $f'$ are both definable over $M$, this implies that $f(\mathbf{a})=f'(\mathbf{a})$. The rest of the proof follows exactly as in the second step.
\end{proof}

\subsection{Some useful lemmas}\label{lemas1}
In this subsection we prove several lemmas that will be required to code finite sets.
\begin{notation} 
Let $U \subseteq K$ be a $1$- torsor, $U=a+bI$ where $I \in \mathcal{I}$. Let $A=bI$, we consider the definable equivalence relation over $U$ given by:
$\displaystyle{E(b,b')}$ if and only if $b-b' \in \mathcal{M}A$. We write $red(U)$ to denote the definable quotient $U/E$. \\
We write $p_{U}(x)$ to denote some type centered in $U$ extending  the generic type of $U$ $\Sigma_{U}^{gen}(x)$ which is $\ulcorner U \urcorner$-definable. If $U$ is closed such type is unique (see Proposition \ref{uniqueclosed}) and for the open case there are several choices for this type, but all of them are $\ulcorner U \urcorner$ -definable by Proposition \ref{open}.
\end{notation}

 \begin{lemma}\label{tipomaster} Let $F=\{B_{1}, \dots, B_{n}\}$ be a primitive finite set of $1$- torsors. Let $W=\{ \{ x_{1},\dots,x_{n}\} \ | \ x_{i} \in B_{i}\}$, and $W^{*}=\{ \ulcorner  \{ x_{1},\dots,x_{n}\} \urcorner \ | \   \{ x_{1},\dots,x_{n}\} \in W\}$  Then there is a  $\ulcorner W^{*} \urcorner$- definable type $q$ concentrated on $W^{*}$. Furthermore, given $b^{*}$ a realization of $q$ sufficiently generic over a set of parameters $C$, if we take $B$ the finite set coded by $b^{*}$, then if $b \in B$ is the element that belongs to $B_{i}$ then $b_{i}$ is a sufficiently generic realization of some type $p_{B_{i}}(x)$, which is $\ulcorner B_{i}\urcorner$-definable and extends the generic type of $B_{i}$. Lastly, the types $p_{B_{i}}(x)$ are compatible under the action of $Aut(\mathfrak{M}/\ulcorner F\urcorner)$ meaning that if $\sigma \in Aut(\mathfrak{M}/\ulcorner F \urcorner)$ and $\sigma(B_{i})=B_{j}$ then $\sigma(p_{B_{i}}(x))=p_{B_{j}}(x)$. 

 \end{lemma}
 \begin{proof} We focus first on the construction of the type $q$, and later we show that it satisfies the required conditions. Suppose that each $1$-torsor $B_{i}=c_{i}+b_{i}I_{i}$ for some $I_{i} \in \mathcal{I}$. By transitivity all the balls are of the same type $I \in \mathcal{I}$ and for all $1\leq i,j \leq n$ we have that $v(b_{i})=v(b_{j})$. Hence, we may assume that each $B_{i}$ is of the form $c_{i}+bI$ for some fixed $c_{i},b \in K$ and $I \in \mathcal{I}$. We argue by cases:
 \begin{enumerate}
 \item \emph{Case $1$: All the $1$-torsors $B_{i}$ are closed.}\\
  For each $i \leq n$, let $p_{B_{i}}(x)$ be the unique $\ulcorner B_{i} \urcorner$-definable type given by Proposition \ref{uniqueclosed}. Define $r(x_{1},\dots,x_{n})= p_{B_{1}}(x_{1})\otimes \dots \otimes p_{B_{n}}(x_{n})$, this  is $\ulcorner W^{*} \urcorner$-definable type. Let $(a_{1},\dots, a_{n}) \vDash r(x_{1},\dots,x_{n})$ and let $q=\tp(\ulcorner \{ a_{1},\dots, a_{n}\} \urcorner / \mathfrak{M})$. This type is well defined independently of the choice of the order, because each type $p_{B_{i}}(x)$ is generically stable, thus it commutes with any definable type by \cite{NIP}[Proposition 2.33]. The type $q$ is  $\ulcorner W^{*} \urcorner$-definable and centered at $W^{*}$. 
 \item \emph{Case $2$: All the $1$-torsors  $B_{i}$ are open, i.e.  $I \in \mathcal{I} \backslash \{ \mathcal{O}\}$.}\\
  Let $S_{bI}=v(b)+S_{I}= \{ v(b)+ v(x) \ | \ x \in I\}$, this is a definable end-segment of $\Gamma$ with no minimal element. Let $r(y)$ be the $\ulcorner S_{bI} \urcorner$ definable type given by Fact  \ref{genericodp}, extending the partial generic type $\Sigma^{gen}_{S_{bI}}(y)$. Fix elements  $a=\{ a_{1},\dots, a_{n} \} \in W(\mathfrak{M})$ and $\delta \vDash r(y)$, we define $\displaystyle{C(a, \delta)=\{ C_{1}(a), \dots, C_{n}(a)\}}$, where each $C_{i}(a)$ is the closed ball around $a_{i}$ of radius $\delta$. For each $i \leq n$ we take $p_{C_{i}(a)}(x)$ the unique extension of the generic type of $C_{i}(a)$ given by Proposition \ref{uniqueclosed}, this type is $\ulcorner C_{i}(a)\urcorner$-definable. 
  Let $q^{a}_{\delta}$ be the symmetrized generic type of $C_{1}(a)\times \dots \times C_{n}(a)$, i.e. we take $\tp(\ulcorner \{ b_{1}, \dots, b_{n}\} \urcorner / \mathfrak{M}\delta)$ where $(b_{1},\dots,b_{n})$ is a realization of the generically stable type $p_{C_{1}(a)}\otimes \dots \otimes p_{C_{n}(a)}$. Let $q^{a}$ be the definable global type satisfying that $d \vDash q^{a}$ if and only if there is some $\delta \vDash r(y)$ and $d \vDash q^{a}_{\delta}$. \\
  
 \begin{claim}{ The type $q^{a}$ does not depend on the choice of $a$.}\end{claim}
 \begin{proof}
  Let $a'=\{a'_{1}, \dots, a'_{n}\} \in W(\mathfrak{M})$ and $\delta \vDash r(y)$. For each $i \leq n$, $a_{i}, a_{i}' \in B_{i}$ meaning that $a_{i}-a_{i}^{'} \in bI$ i.e. $v(a_{i}-a_{i}') \in S_{bI}= v(b)+S_{I}$ and note that $v(a_{i}-a_{i}') \in \Gamma(\mathfrak{M})$.\\
  By construction, $\delta \in S_{bI}$ and $\delta < v(a_{i}-a'_{i})$, thus the closed ball of radius $\delta$ concentrated on $a_{i}$ is the same closed ball of radius $\delta$ concentrated on $a_{i}'$. As the set of closed balls $C(a,\delta)=C(a',\delta)$ we must have that $q^{a}_{\delta}=q^{a'}_{\delta}$, and since this holds for any $\delta \vDash r(y)$ we conclude that $q^{a}$ does not depend on the choice of $a$ and we simply denote it as $q$.
  \end{proof}
  This type $q$ is $\ulcorner W^{*} \urcorner$-definable  and it is centered in $W^{*}$. This finalizes the construction of the type $q$ that we are looking for.
   \end{enumerate}
We continue checking that the type $q$ that we have constructed satisfies the other properties that we want. In both cases, by construction, if $b^{*}$ is a sufficiently generic realization of $q$ over $C$ and $B$ is the finite set coded by $b^{*}$ if we take $b_{i}$ the unique element of $B$ that lies on $B_{i}$ then $b_{i}$ realizes the generic type $\Sigma_{B_{i}}^{gen}(x)$. By Corollary \ref{alldef}, the type $\tp(b_{i}/\mathfrak{M})$ is $\ulcorner B_{i}\urcorner$-definable. If the torsors are closed, then the types $p_{B_{i}}(x)$ are all compatible under the action of $Aut(\mathfrak{M}/\ulcorner F \urcorner)$ as there is a unique complete extension of the generic type of $B_{i}$, this is guaranteed by Proposition \ref{uniqueclosed}. We now work the details for the open case, let's fix $\sigma \in Aut(\mathfrak{M}/\ulcorner F \urcorner)$ and assume that $\sigma(B_{i})=B_{j}$. The type $r(y)$ is $\ulcorner F \urcorner$-definable, thus $\sigma(r(y))=r(y)$. By construction, for all $k \leq n$
the type $p_{B_{k}}(x)$ that we are fixing is the unique extension of generic type of some closed ball $C_{\delta}(a_{k})$ where $a_{k} \in B_{i}$ and $\delta \vDash r(y)$. And for any $a,a' \in B_{i}$ and $\delta, \delta' \vDash r(y)$, $C_{\delta}(a)=C_{\delta'}(a')$.
If $\sigma(B_{i})=B_{j}$, then $\sigma(b_{i}) \in B_{j}$ and
\begin{equation*}
\sigma(C_{\delta}(b_{i}))= C_{\sigma(\delta)}(\sigma(b_{i}))=C_{\delta}(b_{j}).
\end{equation*}
By Proposition \ref{uniqueclosed} there is a unique complete extension of the generic type of the closed ball $C_{\delta}(a_{k})$ for each $k \leq n$, thus $\sigma(p_{B_{i}}(x))=p_{B_{j}}(x)$ as desired. 
 \end{proof}

\begin{notation}\label{fibra} Let $M \subseteq K^{n}$ be a non-trivial definable $\mathcal{O}$-module and let $Z= \bar{d}+M$ be a torsor. Let $\pi_{n}= K^{n} \rightarrow K$ be the projection into the last coordinate. Consider the function that describes the fiber in $Z$ of each element at the projection, this is $h_{Z}(x)=\{ y \in K^{n-1} \ | \ (y,x) \in Z\}$. 
\end{notation}
\begin{fact}\label{suma}Let $M$ be a $\mathcal{O}$-submodule of $K^{n}$. Then for any $x,z \in \pi_{n}(M)$ we have that 
\begin{equation*}
h_{M}(x)+h_{M}(y)=h_{M}(x+y).
\end{equation*}
Furthermore, if $Z=\bar{b}+M \in K^{n}/M$ is a torsor, then for any $d_{1},d_{2} \in \pi_{n}(Z)$ we have that $d_{1}-d_{2} \in \pi_{n}(N)$ and:
\begin{equation*}
    h_{N}(d_{1}-d_{2})=h_{Z}(d_{1})-h_{Z}(d_{2}).
\end{equation*}
\end{fact}
\begin{proof}
This is a straightforward computation and it is left to the reader. 
\end{proof}
%\textcolor{red}{debemos especificar que entendemos por $p_{A}$ en esta seccion cuando $A$ es un $\mathcal{O}$-submodule}
%\textcolor{blue}{This uses generics, we went together through it. Mari: revisa esto con mas calma estoy muy cansada ahora. Debemos fijarnos continuamente en los parametros y tomar las realizaciones suff generic}
\begin{lemma}\label{codemodule} Let  $n \geq 2$ be a natural number and  $M \subseteq K^{n}$ be a definable $\mathcal{O}$-submodule. Then $\ulcorner M \urcorner$ is interdefinable with $\displaystyle{(\ulcorner \pi_{n}(M) \urcorner,\germ(h_{M}, p_{\pi_{n}(M)}))}$, where $p_{\pi_{n}}(M)$ is any complete extension of the generic type of $\pi_{n}(M)$.
\end{lemma}
\begin{proof}
Let $M_{1}$ and $M_{2}$ be $\mathcal{O}$-modules of the same type. Suppose that $A=\pi_{n}(M_{1})= \pi_{n}(M_{2})$, and $\germ(h_{M_{1}},p_{A})=\germ(h_{M_{2}},p_{A})$. We must show that $M_{1}$ and $M_{2}$ are the same $\mathcal{O}$-module.\\

Let $c$ be a realization of the type $p_{A}(x)$ sufficiently generic over $\ulcorner M_{1} \urcorner \ulcorner M_{2} \urcorner$ and $d=c-y$. By Proposition \ref{stabadd} $d$ is a realization of $p_{A}(x)$ sufficiently generic over $\ulcorner M_{1} \urcorner \ulcorner M_{2} \urcorner$, and $y=c-d$.  
 As $\germ(h_{M_{1}},p_{A})=\germ(h_{M_{2}},p_{A})$, we have that $h_{M_{1}}(c)=h_{M_{2}}(c)$ and $h_{M_{1}}(d)=h_{M_{2}}(d)$. By Fact \ref{suma}, $h_{M_{1}}(y)=h_{M_{1}}(c)-h_{M_{1}}(d)=h_{M_{2}}(c)-h_{M_{2}}(d)=h_{M_{2}}(y)$. Consequently, $M_{1}=M_{2}$ as desired. 
\end{proof}
%\textcolor{blue}{This uses generics, we went together through it. }
\begin{corollary}\label{torgen} Let $n \geq 2$ be a natural number and $N\subseteq K^{n}$ be a definable $\mathcal{O}$-submodule. Let $Z= \bar{b}+N$ be a torsor,  then $\ulcorner Z \urcorner$ is interdefinable with $\displaystyle{(\ulcorner \pi_{n}(Z) \urcorner, \germ(h_{Z},  p_{\pi_{n}(Z)}))}$, where $p_{\pi_{n}}(Z)$ is any global type containing the generic type of $\pi_{n}(Z)$.
\end{corollary}
\begin{proof}
We first show that  $\ulcorner Z \urcorner$ is interdefinable with $\big(\ulcorner \pi_{n}(Z) \urcorner, \ulcorner N \urcorner,\germ(h_{Z}, p_{\pi_{n}(Z) }) \big)$. Let $Z_{1}=\bar{b}_{1}+N$ and $Z_{2}=\bar{b}_{2}+N$ torsors, and suppose that $A=\pi_{n}(Z_{1})= \pi_{n}(Z_{2})$.  Let $c$ be a realization of the type $p_{A}(x)$ sufficiently generic over $\ulcorner Z_{1} \urcorner \ulcorner Z_{2} \urcorner$, then $h_{Z_{1}}(c)=h_{Z_{2}}(c)$. If $Z_{1}\neq Z_{2}$, then they must be disjoint because they are different cosets of $N$. But if $h_{Z_{1}}(c)=h_{Z_{2}}(c)$ then $Z_{1} \cap Z_{2} \neq \emptyset$, so  $Z_{1}= Z_{2}$.\\
We continue showing that $N$ is definable over $(\ulcorner \pi_{n}(Z) \urcorner, \germ(h_{Z}, p_{ \pi_{n}(Z)}))$. We will find a global type $p_{\pi_{n}(N)}(x)$ extending the generic type of $\pi_{n}(N)$ such that 
\begin{equation*}
(\ulcorner \pi_{n}(N) \urcorner, \germ(h_{N},p_{\pi_{n}(N)})) \in \dcl^{eq}(\ulcorner \pi_{n}(Z) \urcorner,  \germ(h_{Z}, p_{\pi_{n}(Z)})).
\end{equation*}
By Lemma \ref{codemodule}, this guarantees that $\ulcorner N \urcorner\in \dcl^{eq}(\ulcorner \pi_{n}(Z) \urcorner,  \germ(h_{Z}, p_{\pi_{n}(Z)}))$.\\
First, let $y'\in \pi_{n}(Z)$, then $\pi_{n}(N)=\{ y- y' \ | \ y \in \pi_{n}(Z)\}$. As this definition is independent from the choice of $y'$, we have $\ulcorner \pi_{n}(N) \urcorner \in \dcl^{eq}(\ulcorner \pi_{n}(Z) \urcorner)$.
\begin{claim} Let $q(x_{2},x_{1})=p_{\pi_{n}(Z)}(x_{2}) \otimes p_{\pi_{n}(Z)}(x_{1})$ and $(d_{2},d_{1})\vDash q(x_{2},x_{1})$. \\
Then $\Sigma_{\pi_{n}(N)}^{gen}(y) \subseteq \tp(d_{2}-d_{1}/\mathfrak{M})$.
\end{claim}
\begin{proof}
We proceed by contradiction and we assume the existence of some proper $\mathfrak{M}$-definable subtorsor $B\subseteq \pi_{n}(N)$ such that $d_{2}-d_{1} \in B$. Then $d_{2}\in d_{1}+B \subsetneq \pi_{n}(Z)$, and $d_{1}+B$ is a proper $\mathfrak{M}d_{1}$-definable torsor of $\pi_{n}(Z)$, but this contradicts that $d_{2}$ is a sufficiently generic realization of $\Sigma_{\pi_{n}(Z)}^{gen}(x)$ over $\mathfrak{M}d_{1}$. 
\end{proof}
Let $p_{\pi_{n}(N)}(x)=\tp(d_{2}-d_{1}/\mathfrak{M})$, we observe that such type is independent of the choices of $d_{1}$ and $d_{2}$ as the congruence and coset formulas are completely determined in the type $q(x_{2},x_{1})$.\\
It is only left to show that $\germ(h_{N},p_{\pi_{n}(N)}) \in \dcl^{eq}(\ulcorner \pi_{n}(Z) \urcorner,  \germ(h_{Z}, p_{\pi_{n}(Z)}))$. \\
Let $\sigma \in Aut(\mathfrak{M}/(\ulcorner \pi_{n}(Z) \urcorner,  \germ(h_{Z}, p_{\pi_{n}(Z)})))$, we will show that $h_{N}(x)=h_{\sigma(N)}(x) \in p_{\pi_{n}(N)}(x)$. Because $p_{\pi_{n}(N)}(x)$ is $\ulcorner\pi_{n}(N)\urcorner$-definable then $\sigma(p_{\pi_{n}(N)}(x))=p_{\pi_{n}(N)}(x)$. As $\sigma(\germ(h_{Z},p_{\pi_{n}(Z)}))=\germ(h_{Z},p_{\pi_{n}(Z)})$, then $h_{Z}(x)=h_{\sigma(Z)}(x) \in p_{\pi_{n}(Z)}$. Let $C=\{ \ulcorner Z \urcorner, \ulcorner \sigma(Z)\urcorner, \ulcorner N \urcorner, \ulcorner \sigma(N) \urcorner\}$.  In particular, if $d_{1}\vDash p_{\pi_{n}(Z)}\upharpoonright_{C}$ and $d_{2}\vDash p_{\pi_{n}(Z)}\upharpoonright_{{C}d_{1}}$ then $h_{Z}(d_{1})=h_{\sigma(Z)} (d_{1})$ and $h_{Z}(d_{2})=h_{\sigma(Z)}(d_{2})$. By Fact \ref{suma},
\begin{equation*}
    h_{N}(d_{2}-d_{1})=h_{Z}(d_{2})-h_{Z}(d_{1})=h_{\sigma(Z)}(d_{2})-h_{\sigma(Z)}(d_{1})=h_{\sigma(N)}(d_{2}-d_{1}).
\end{equation*}
Consequently $\sigma(\germ(h_{N},p_{\pi_{n}(N)}))=\germ(h_{N},p_{\pi_{n}(N)})$. Because $\sigma$ is arbitrary, we conclude that 
\begin{equation*}
\germ(h_{N},p_{\pi_{n}(N)})\in \dcl^{eq}(\ulcorner \pi_{n}(Z) \urcorner,  \germ(h_{Z}, p_{\pi_{n}(Z)})), \ \text{as required.}
\end{equation*}
\end{proof}
\subsection{Some coding lemmas}\label{lemas2}
%\textcolor{red}{Maybe I should include a fact for the primitivity assumption and a draw at the beginning of this section. }
%\begin{notation}  Let  $A$ be a definable $\mathcal{O}$-lattice in $K^{n}$ and  $U=\bar{d}+A \in K^{n}/A$ be a $1$-torsor. We write as $\red(U)=\{ \bar{d}+t \ | \ t \in \red(A)=A/\mathcal{M}A\} \subseteq K/\mathcal{M} A$. 
%\end{notation}
%\textcolor{red}{This proposition is an easy computation... whathever easy means. Should I take away the proof to make shorter the paper? }
\begin{lemma}\label{injectiveresidue} Let $A$ be a definable $\mathcal{O}$-lattice in $K^{n}$ and  $U \in K^{n}/A$ be a torsor. Let $B$ be the $\mathcal{O}$-lattice in $K^{n+1}$ that is interdefinable with $U$ (given by Proposition \ref{torsormodule}).Then there is a $\ulcorner U \urcorner$- definable injection :
\begin{center}
$f=\begin{cases}
 \red(U) &\rightarrow \red(B)\\
 b+\mathcal{M}A &\mapsto (b,1)+\mathcal{M}B.
\end{cases}$
\end{center}
\end{lemma}
\begin{proof}
We recall how the construction of $B$ was achieved. Given any $\bar{d} \in U$, we can represent $B= A_{2}+\begin{bmatrix}  \bar{d} \\ 1 \end{bmatrix} \mathcal{O}$, where $A_{2}=\{0\}\times A$. This definition is independent from the choice of $\bar{d}$. We consider the $\ulcorner U \urcorner$ definable injection $\phi= U \rightarrow B$ that sends each element $\bar{b}$ to $\begin{bmatrix} \bar{b} \\ 1 \end{bmatrix}$. The interpretable sets  $\red(U)$ and $\red(B)=B/\mathcal{M}B$ are both $\ulcorner U \urcorner$-definable. It follows by a standard computation that for any $b,b' \in U$, $b-b' \in \mathcal{M}A$ if and only if $ \begin{bmatrix} b \\ 1 \end{bmatrix}- \begin{bmatrix} b' \\ 1 \end{bmatrix} \in \mathcal{M}B$. This shows that the map $f$ is a  $\ulcorner U \urcorner$-definable injection. 
\end{proof}
%Let $b,b' \in U$ and suppose $b-b' \in \mathcal{M}A$ then there are  $x \in \mathcal{M}$ and  $d \in A$ such that $b-b'=xd$. In particular $ \begin{bmatrix} 1 \\ b \end{bmatrix}- \begin{bmatrix} 1 \\ b' \end{bmatrix} =\begin{bmatrix} 0 \\ b-b' \end{bmatrix}= x \begin{bmatrix} 0 \\  d\end{bmatrix} \in \mathcal{M}A_{2} \subseteq \mathcal{M}B$. For the converse, suppose  $\begin{bmatrix} 1 \\ b \end{bmatrix}- \begin{bmatrix} 1 \\ b' \end{bmatrix} \in \mathcal{M}B$, then
%$\begin{bmatrix} 0 \\b- b' \end{bmatrix} \in A_{2} \cap \mathcal{M}B=\mathcal{M}A_{2}$.
%So there is some $x \in \mathcal{M}$ and some $a \in A$ such that $\begin{bmatrix} 0 \\b- b' \end{bmatrix} =x\begin{bmatrix} 0 \\ a \end{bmatrix}=\begin{bmatrix} 0 \\ xa \end{bmatrix}$
%thus $b-b' \in \mathcal{M}A$.
%\end{proof}
%By the previous claim the map
%\begin{align*}
%f= res(U) &\rightarrow res(B)\\
%b+\mathcal{M}A &\rightarrow (1,b)+\mathcal{M}B, 
%\end{align*} 
%is an $\ulcorner U \urcorner$- definable injective map, as required. 

\begin{lemma}\label{codetorsors} Let $F$ be a primitive finite set of $1$-torsors,  then $F$ can be coded in $\mathcal{G}$. 
\end{lemma}
\begin{proof}
If $|F|=0$ or $|F|=1$ the statement follows clearly. So we may assume that $|F|> 1$. By primitivity all the torsors in $F$ are translates of the same $\mathcal{O}$-submodule of $K$.  Indeed, there are some $b \in K$ and $I \in \mathcal{I}$ that for any $t \in F$ there is some $a_{t} \in K$ satisfying $t=a_{t}+bI$. Moreover, there is some $\delta \notin v(b)+S_{I}$ such that for any two different torsors $t, t' \in F$ if $x \in t$ and $y \in t'$ then $v(x-y)=\delta$.  Let $\displaystyle{T=\bigcup_{t \in F} t}$. We define
 \begin{align*}
 J_{F}= \{ Q(x) \in K[x] \ | \ \text{$Q(x)$ has degree at most $|F|$ and for all $x \in T$,}  \ v(Q(x)) \in v(b)+(|F|-1)\delta+ S_{I}\}. 
 \end{align*}
 \textit{Step $1$: $\ulcorner J_{F} \urcorner$ is interdefinable with $\ulcorner F \urcorner$.}\\
 Observe that $J_{F}$ is definable over $\ulcorner F \urcorner$, because $v(b), \ulcorner T \urcorner, \delta$ lie in $\dcl^{eq}(\ulcorner F \urcorner)$. Hence, it is sufficient to prove that we can recover $F$ from $J_{F}$. For this we will show that given a monic polynomial $Q(x) \in K[x]$ with exactly $|F|$-different roots in $K$ each of multiplicity one, we have that $Q(x) \in J_{F}$ if and only if $Q(x)$ satisfies all the following condition:\\

    \emph{Condition}: Let $\{ \beta_{1},\dots,\beta_{|F|}\} \subseteq K$ be the set of all the roots of $Q(x)$ (note that all of them are different). For each $1\leq i \leq |F|$ there is some $t \in F$ such that $\beta_{i} \in t$. And all the roots of $Q(x)$ lie in different torsors, i.e. if $i\neq j$, take $t,t' \in F$ such that $\beta_{i} \in t$ and $\beta_{j} \in t'$ then $t \neq t'$.\\

%\textcolor{red}{add draw if there is time}.\\
We first show that a monic polynomial $Q(x)$ with exactly $|F|$-different roots in $K$ each of multiplicity one satisfying the condition above belongs to $J_{F}$. Let $R=\{ \beta_{1},\dots,\beta_{|F|}\}\subseteq K$ the set of all the (different) roots of $Q(x)$. Let $x \in T$, then there is some $t \in F$ such that $x \in t$. Let $\beta_{i}$ be the root of $Q(x)$ that belongs to $t$, then $x,\beta_{i} \in t$, so $v(x-\beta_{i}) \in v(b)+S_{I}$. For any other index $j \neq i$, let $t' \in F$ be such that $\beta_{j} \in t'$, because $t\neq t'$, $\ v(x-\beta_{j})=\delta$. Summarizing we have: 
 \begin{align*}
 v \big(Q(x) \big)= v \big( \prod_{k\leq |F|} (x-\beta_{k}) \big)=\underbrace{v(x-\beta_{i})}_{\in v(b)+S_{I}}+\underbrace{\sum_{j \neq i} v(x-\beta_{j})}_{=(|F|-1)\delta} \in v(b)+(|F|-1)\delta + S_{I}.  
 \end{align*}
 Consequently, $Q(x) \in J_{F}$.  For the converse, let $Q(x)\in J_{F}$ be a monic polynomial with exactly $|F|$-different roots $R=\{ \beta_{1},\dots,\beta_{|F|}\}\subseteq K$. We show that $Q(x)$ satisfies the condition, i.e. each root belongs to some torsor $t \in F$ and any two different roots belong to different torsors of $F$. \\
 \begin{claim}\label{delta} Given any torsor $t \in F$, there is a unique root $\beta \in R$ such that for all elements $x \in t$, $v(x-\beta)>\delta$.
 \end{claim}
\begin{proof}
 Let $t \in F$ be a fixed torsor. We first show the existence of some root $\beta \in R$ such that for any $x \in t$, we have $v(x-\beta)>\delta$. We argue by contradiction, so 
 let $t \in F$ and assume that there is no root $\beta \in B$ such that $v(x-\beta)> \delta$ for all $x \in t$. Then for each element $x \in t$ we have: 
 \begin{align*}
 v(Q(x))=v\big( \prod_{i \leq |F| } (x- \beta_{i})\big)= \sum_{i \leq |F|} v(x-\beta_{i}) \leq |F| \delta.
 \end{align*}
In this case  $Q(x) \notin J_{F}$,  because $|F|\delta  \notin v(b)+(|F|-1)\delta+S_{I}$ as $\delta \notin v(b)+S_{I}$. This concludes the proof for existence.\\
For uniqueness, let $\{ t_{1},\dots, t_{|F|}\}$ be some fixed enumeration of $F$. Let $\beta_{i} \in R$ be such that for all $x \in t_{i}$ we have $v(x-\beta_{i})>\delta$. We first argue that for any $i \neq j$, we must have that $\beta_{i}\neq \beta_{j}$. Suppose by contradiction that $\beta_{i}=\beta_{j}=\beta$, and let $x \in t_{i}$ and $y \in t_{j}$, then:
\begin{equation*}
   \delta= v(x-y)=v((x-\beta)+(\beta-y))\geq \min\{ v(x-\beta), v(y-\beta)\}>\delta.
\end{equation*}
The uniqueness now follows because $|F|=|R|$. 
\end{proof}

By Claim \ref{delta}, we can fix an enumeration $\{ t_{i} \ | \ i \leq |F|\}$ of $F$  such that for any $x \in t_{i}$, $v(x-\beta_{i})>\delta$. We note that if $j \neq i$, then for any $x \in t_{i}$ we have that $v(x-\beta_{j})=\delta$. In fact, fix some $y \in t_{j}$, as $v(y-\beta_{j})>\delta$ we have:
\begin{equation*}
    v(x-\beta_{j})= v((x-y)+(y-\beta_{j}))=\min\{ \underbrace{v(x-y)}_{=\delta}, \underbrace{v(y-\beta_{j})}_{>\delta}\}=\delta.
\end{equation*}
\begin{claim} For each $ i \leq |F|$ we have that $\beta_{i} \in t_{i}$. 
\end{claim}
\begin{proof}
We fix some $i \leq |F|$. Thus, for any $x \in t_{i}$: 
 \begin{align*}
 v(Q(x))= v \big( \prod_{k \leq |F|} v(x- \beta_{k})\big)= v(x-\beta_{i})+ \sum_{j \neq i} v(x-\beta_{j})=v(x-\beta_{i})+ (|F|-1)\delta.
 \end{align*}
  Because $Q(x) \in J_{F}$, we must have that $v(x-\beta_{i}) \in v(b)+S_{I}$. Thus $\beta_{i} \in t_{i}$. Moreover, by construction, if $i \neq j$ then $t_{i} \neq t_{j}$.
  \end{proof}
 %\textcolor{red}{should I include details on how to define $F$?}
 \textit{Step $2$: $F$ admits a code in the geometric sorts.}\\
 By the first step $F$ is interdefinable with $J_{F}$. The latter one is an $\mathcal{O}$-module, so by Lemma \ref{modulesok} it admits a code in the stabilizer sorts $\mathcal{G}$. 
\end{proof}

\begin{lemma}\label{tricktorsors} Let $F$ be a primitive finite set of  $1$-torsors such that $|F|>1$. There is a $\ulcorner F \urcorner$- definable $\mathcal{O}$-lattice $s \subseteq K^{2}$  and an $\ulcorner F \urcorner$-definable injective map $g= F \rightarrow \VS_{k, \ulcorner s \urcorner}$.
\end{lemma}
\begin{proof}
Let $F$ be a primitive finite set of $1$-torsors. By primitivity, there is some $d \in K$ and $I \in \mathcal{I}$ such that for any $t \in F$ there is some $a_{t}\in K$ satisfying $t=a_{t}+dI$.  Moreover, there is some $\delta \in \Gamma \backslash (v(d)+S_{I})$ such that for any pair of different torsors $t, t' \in F$, and $x \in t$, $y \in t'$ we have $v(x-y)=\delta$. Let  $\displaystyle{T = \bigcup_{t \in F} t}$, and take elements $c \in T$ and $b \in K$ such that $v(b)=\delta$. Let $U=c+b \mathcal{O}$. Then $U$ is the smallest closed $1$-torsor that contains all the elements of $F$. Note that $U$ is definable over $\ulcorner F \urcorner$. Let $h$ be the map sending each element of $F$ to the unique class that contains it in $\red(U)$. By construction, such a map must be injective and it is $\ulcorner F \urcorner$-definable. Let $s$ be the $\mathcal{O}$-lattice in $K^{2}$, whose code is interdefinable with $\ulcorner U \urcorner$ (given by Proposition \ref{torsormodule}). By Lemma \ref{injectiveresidue} there is a $\ulcorner s\urcorner$-definable injection $\displaystyle{f:\red(U) \rightarrow \red(s)}$. Let $g= f\circ h$, the composition map $g= F \rightarrow \VS_{k, \ulcorner s \urcorner}$ satisfies the required conditions.\end{proof}
%\textcolor{red}{actually we may assume the third condition to be simply $\mathcal{O}$-modules... but ok, it is harmless. }
\begin{lemma} \label{codingeasy}Let $F$ be a finite set of $1$-torsors  and let $f: F \rightarrow \mathcal{G}$ be a definable function. Suppose that $F$ is primitive over $\ulcorner f \urcorner$, then:
\begin{enumerate}
\item for any set of parameters $C$ if $f(F) \subseteq VS_{k,C}$ then $f$ is coded in $\mathcal{G}$ over $C$,
\item if $f(F) \subseteq K$ then $f$ is coded in $\mathcal{G}$,
\item if $f(F)$ is a finite set of $1$-torsors of the same type $I \in \mathcal{I}$. Then $f$ is coded in $\mathcal{G}$.
\end{enumerate}
\end{lemma}
\begin{proof} 
In all the three cases, we may assume that $|F|>1$, otherwise the statement clearly follows. Also, by primitivity of $F$ over $\ulcorner f \urcorner$, $f$ is either constant or injective. If it is constant and equal to $c$, the tuple $(\ulcorner F \urcorner, c)$ is a code for $f$. By Lemma \ref{codetorsors} $\ulcorner F \urcorner$ admits a code in $\mathcal{G}$, so $(\ulcorner F \urcorner, c)$ is interdefinable with a tuple in the stabilizer sorts. In the following arguments we assume that $f$ is an injective function and that $|F| \geq 2$.
\begin{enumerate}
\item By Lemma \ref{codetorsors}  $\ulcorner F \urcorner \in \mathcal{G}$.  Let $s$ be the $\mathcal{O}$-lattice of $K^{2}$ and $g: F \rightarrow VS_{k, C\ulcorner s \urcorner}$ the injective map given by Lemma \ref{tricktorsors}. Both $s$ and $g$ are $\ulcorner F \urcorner$-definable. Let $F^{*}=g(F)\subseteq VS_{k, C\ulcorner s \urcorner}$, the map $f \circ g^{-1}: F^{*}\rightarrow VS_{k, C\ulcorner s \urcorner}$ can be coded in $\mathcal{G}$ by Theorem \ref{EIinternalresidue}. Hence, the tuple $(\ulcorner f \circ g^{-1} \urcorner, \ulcorner F \urcorner)$ is a code of $f$ over $C$, because $g$ is a $\ulcorner F \urcorner$-definable bijection, and $(\ulcorner f \circ g^{-1} \urcorner, \ulcorner F \urcorner)$ is interdefinable with a tuple of elements in $\mathcal{G}$.
\item Let $D=f(F) \subseteq K$, this is a finite set in the main field so it can be coded by a tuple $d$ of elements in $K$. Because $F$ is primitive over $\ulcorner f\urcorner$, then $D$ is a primitive set. Thus, there is some $\delta \in \Gamma$ such that for any pair of different elements $x, y \in D$ $v(x-y)=\delta$. Let $b \in K$  be such that $v(b)= \delta$, take $x \in D$ and let $U= x +b \mathcal{O}$, this is the smallest closed $1$-torsor containing $D$. The elements of $D$ all lie in different classes of $\red(U)$ and let $g: D \rightarrow \red(U)$ be the definable map sending each element $x \in D$ to the unique element in $\red(U)$ that contains $x$. Both $U$ and $g$ are $\ulcorner D \urcorner$-definable, and therefore $d$-definable. By Proposition \ref{torsormodule}, there is  an $\mathcal{O}$-lattice $s \subseteq K^{2}$, whose code is interdefinable with $\ulcorner U\urcorner$. Let  $h:\red(U) \rightarrow \red(s)$ the $\ulcorner U \urcorner$- definable injective map given by Lemma \ref{injectiveresidue}. Both $U$ and $h$ are $d$-definable. By $(1)$ of this statement the function 
$h \circ g \circ f: F \rightarrow VS_{k,  \ulcorner s \urcorner}$
can be coded in $\mathcal{G}$. Since $f$ and $h \circ g \circ f$ are interdefinable over  $d$, the statement follows. %(In more detail, to recover $f: F \rightarrow K$ simply send $t$ to the unique element in the set coded by $d$ such that $g \circ f(x)=h \circ g \circ f(t)$. To describe $h \circ g \circ f$, simply note that the construction of $U, g$ and $h$ is $d$-definable.)
\item  Let $D= f(F)$ then $D$ must be a primitive set of $1$-torsors because $F$ is primitive over $\ulcorner f \urcorner$ (in particular, there are $I\in \mathcal{I}$ and $b \in K$ and elements $a_{t} \in K$ such that for each $t \in f(F)$ we have $t=a_{t}+bI$). By Lemma \ref{codetorsors}, we may assume $\ulcorner D\urcorner$ is a tuple in the stabilizer sorts.  Let $s \subseteq K^{2}$ and $g: D \rightarrow \red(s) \subseteq VS_{k, \ulcorner s \urcorner}$ the injective map given by Lemma \ref{tricktorsors}. Both $s$ and $g$ are $\ulcorner D \urcorner$-definable. By part $(1)$ of this statement the composition $g \circ f$ can be coded in $\mathcal{G}$, and as $g$ is a $\ulcorner D \urcorner$-definable bijection the tuple $(\ulcorner g \circ f \urcorner, \ulcorner D \urcorner)$ is interdefinable with $\ulcorner f \urcorner$. 
\end{enumerate}
\end{proof}
%\textcolor{red}{1. perhaps is better to across all this section fix the notation at the beginning of $p_{A}$ being the definable $1$-dimensional type\\
%2. Coding functions germs of functions is an argument that should naturally (by the same proof) generalized to higher dimensional settings... i.e. not only to $1$-torsors....but ok}

\subsection{Coding of finite sets of tuples in the stabilizer sorts}\label{primitive}
We start by recalling some terminology from previous sections for sake of clarity.
\begin{notation} Let $M \subseteq K^{n}$ be an $\mathcal{O}$-module, and $(I_{1},\dots, I_{n}) \in \mathcal{I}$ be such that $\displaystyle{M \cong \oplus_{i \leq n}I_{i}}$. For any torsor $Z= \bar{d}+ M \in K^{n}/M$ we say that $Z$ is of type $(I_{1},\dots, I_{n})$ and it has \emph{complexity $n$}. We denote by $\pi_{n}: K^{n} \rightarrow K$ the projection to the last coordinate and for a torsor $Z =\bar{d}+ M \in K^{n}/M$ we write as $A_{Z}= \pi_{n}(Z)$. We recall as well the notation introduced in \ref{fibra} for the function that describes the fiber in $Z$ of each element at the projection, this is $h_{Z}(x)=\{ y \in K^{n-1} \ | \ (y,x) \in Z\}$.
\end{notation}
\begin{definition} Let $F$ be a finite set of torsors, the complexity of $F$ is the maximum complexity of the torsors $t \in F$.
\end{definition}
The following is a very useful fact that we will use repeatedly.
\begin{fact}\label{inter} Let $F$ be a finite set of torsors, then there is a finite set $F' \subseteq \mathcal{G}$ such that $\ulcorner F \urcorner$ and $\ulcorner F' \urcorner$ are both interdefinable. In particular, any definable function $f: F \rightarrow P$, where $P$ is a finite set of torsors or $P \subseteq \mathcal{G}$, is interdefinable with a fuction $g: F' \rightarrow \mathcal{G}$, where $F' \subseteq \mathcal{G}$.
\end{fact}
\begin{proof}
The statement follows immediately by Proposition \ref{torsormodule}.
\end{proof}
The main goal of this section is the following theorem. 
\begin{theorem} For every $m \in \mathbb{N}_{\geq 1}$ the following hold:
\begin{itemize}
    \item $I_{m}$: For every $r>0$ and finite set $F \subseteq \mathcal{G}^{r}$ of size  $m$ then $\ulcorner F \urcorner$ is interdefinable with a tuple of elements in $\mathcal{G}$.
    \item $II_{m}$: For every $F \subseteq \mathcal{G}$ of size $m$ and $f: F \rightarrow \mathcal{G}$ a definable function, then $\ulcorner f \urcorner$ is interdefinable with a tuple in $\mathcal{G}$.
    \end{itemize}
\end{theorem}
We will prove this statement by induction on $m$, we note that for $m=1$ the statements $I_{m}$ and $II_{m}$ follow trivially. We now assume that $I_{k}$ and $II_{k}$ hold for each $k \leq m$ and we want to show $I_{m+1}$ and $II_{m+1}$. In order to keep the steps of the proof easier to follow we break the proof into some smaller steps. We write each step as a proposition to make the document more readable.\\
\begin{proposition}\label{coding1} Let $F$ be finite set of torsors of size at most $m+1$, then $\ulcorner F \urcorner$ is interdefinable with a tuple of elements in $\mathcal{G}$. Furthermore, any definable function $f: S \rightarrow F$, where $S$ is a finite set of at most $m+1$-torsors and $F$ is a finite set of torsors can be coded in $\mathcal{G}$.
\end{proposition}
\begin{proof}
We will start by proving the following statements by a simultaneous induction on $n$:
\begin{itemize}
    \item $A_{n}$: Any set $F$ of torsors of size at most $m+1$ of complexity at most $n$ can be coded in $\mathcal{G}$.
    \item $B_{n}:$ Every definable function $f: S \rightarrow F$, where $S$ is a finite set of at most $m+1$-torsors and $F$ is a finite set of torsors of complexity at most $n$ can be coded in $\mathcal{G}$.
\end{itemize}
 We observe first that we may assume in $A_{n}$ that $F$ is a primitive set of size $m+1$. If $|F|\leq m$ the statement follows immediately by Fact \ref{inter} combined with $I_{k}$ for each $k \leq m$. So we may assume that $F$ has $m+1$ elements. If $F$ is not primitive, then we can find a non trivial equivalence $E$ relation definable over $\ulcorner F \urcorner$, and let $C_{1},\dots, C_{l}$ be the equivalence classes. For each $i \leq l$ $\ |C_{i}|\leq m$, by Fact \ref{inter}
 and because $I_{k}$ holds for each $k \leq m$
 $\ulcorner C_{i}\urcorner$ is interdefinable with a tuple $c_{i}$ of elements in $\mathcal{G}$. Because $l\leq m$ and $I_{l}$ holds, 
 we can find a code $c$ in the stabilizer sorts of the set $\{ c_{1},\dots, c_{l}\}$. The code $\ulcorner F \urcorner$ is interdefinable with $c \in \mathcal{G}$.\\
 Likewise, for $B_{n}$ we may assume that $S$ is primitive over $\ulcorner f \urcorner$. Otherwise, there is a $\big(\ulcorner f \urcorner \cup \ulcorner S \urcorner\big)$-definable equivalence relation $E$ on $S$ and let $C_{1},\dots, C_{l}$ be the equavalence classes of this relation. For each $i \leq l$, $|C_{i}| \leq m$ and let $f_{i}=f\upharpoonright_{C_{i}}$. By Fact \ref{inter}, for each $i \leq l$  $\ulcorner f_{i}\urcorner$ is interdefinable with a map $g_{i}: S_{i}\rightarrow \mathcal{G}$ where $S_{i} \subseteq \mathcal{G}$ and $|S_{i}|\leq m$. Because $II_{k}$ holds for each $k\leq m$,  $f_{i}$  admits a code $c_{i}$ in $\mathcal{G}$. Because $I_{l}$ holds, we can find a code $c$ for the finite set $\{c_{1},\dots,c_{l}\}$. The codes $\ulcorner f \urcorner$ and $c$ are interdefinable.\\
We continue arguing for the base case $n=1$. The statement $A_{1}$ holds by Lemma \ref{codetorsors}, while $B_{1}$ is given by (3) of Lemma \ref{codingeasy}. We now assume that $A_{n}$ and $B_{n}$ hold and we prove $A_{n+1}$ and $B_{n+1}$.\\
First we prove that $A_{n+1}$ holds. Let $F$ be a primitive finite set of torsors of size $m+1$. By primitivity all the torsors in $F$ are of the same type. For each $Z \in F$ we write $A_{Z}$ to denote the projection of $Z$ into the last coordinate. By primitivity of $F$ the projections to the last coordinate are either all equal or all different. We argue by cases:
\begin{enumerate}
\item \textit{Case $1$: All the projections are equal, i.e. $A=A_{Z}$ for all $Z \in F$. }
\begin{proof}
 For each $x \in A$, the set of fibers $\{ h_{Z}(x) \ | \ Z \in F \}$ is a  finite set of torsors of size at most $m+1$ of complexity at most $n$. By the induction hypothesis $A_{n}$ it admits a code in the stabilizer sorts. By compactness we can uniformize such codes, and we can define the function $g: A \rightarrow \mathcal{G}$ by sending  the element $x$ to the code $\ulcorner \{ \ulcorner h_{Z}(x) \urcorner \ | \ Z \in F\} \urcorner$. This is a $\ulcorner F \urcorner$-definable function. Let $p_{A}(x)$ be a global type extending the generic type of $A$, it is  $\ulcorner A \urcorner$-definable by Corollary \ref{alldef}. By Theorem \ref{codinggerm} the germ of $g$ over $p_{A}$ can be coded in $\mathcal{G}$ over $\ulcorner A \urcorner$. By Corollary \ref{torgen} for any $Z \in F$ the code $\ulcorner Z \urcorner $ is interdefinable with the tuple $(\ulcorner A \urcorner, \germ(h_{Z}, p_{A}))$, then $\ulcorner F \urcorner$ is interdefinable with $(\ulcorner A \urcorner, \ulcorner\{\germ(h_{Z}, p_{A}) \ | \ Z \in F \}\urcorner)$. \\
%\textcolor{blue}{Warning on the following Claim---this one I feel much more confident to recover it.}\\
\begin{claim}\label{claim1}{ $\germ(g,p_{A})$ is interdefinable with the code $\ulcorner\{ \germ(h_{Z}, p_{A}) \ | \ Z \in F\} \urcorner$ over $\ulcorner A \urcorner$.}\end{claim}
 \begin{proof}
We first prove that  $\germ(g,p_{A}) \in \dcl^{eq}(\ulcorner A \urcorner, \ulcorner\{\germ(h_{Z}, p_{A}) \ | \ Z \in F\}\urcorner)$. \\
 Let $\sigma \in Aut(\mathfrak{M}/\ulcorner  \{ \germ(h_{Z}, p_{A}) \ | \ Z \in F\} \urcorner, \ulcorner A \urcorner)$, we want to show that $\sigma(\germ(g,p_{A}))=\germ(\sigma(g),p_{A})=\germ(g,p_{A})$. Let $B$ the set of all the parameters required to define all the objects that have been mentioned so far. It is therefore sufficient to argue that for any realization $c$ of $p_{A}(x)$ sufficiently generic over $B$ %$B=\{\ulcorner A \urcorner\} \cup \{ \ulcorner Z \urcorner, \ulcorner \sigma(Z) \urcorner \ | \ Z \in F \}$ 
 we have $\sigma(g)(c)=g(c)$, where $\sigma(g):A \rightarrow \mathcal{G}$ is the function given by sending the element $x  $ to the code $\ulcorner \{ \ulcorner h_{\sigma(Z)}(x)\urcorner  \ | \ Z \in F\}\urcorner$. Note that 
 \begin{equation*}
 \sigma( \{ \germ(h_{Z}, p_{A}) \ | \ Z \in F\} )=\{ \germ(h_{\sigma(Z)}, p_{A}) \ | \ Z \in F \}=\{ \germ(h_{Z}, p_{A}) \ | \ Z \in F\},
 \end{equation*}
 because $\sigma(\ulcorner \{\germ(h_{Z}, p_{A}) \ | \ Z \in F \} \urcorner)=\ulcorner \{\germ(h_{Z}, p_{A}) \ | \ Z \in F \} \urcorner$. 
 As a result,  for any realization $c$ of $p_{A}(x)$ sufficiently generic over $B$ we must have that $ \{ \ulcorner h_{Z}(c) \urcorner \ | \ Z \in F\}= \{ \ulcorner h_{\sigma(Z)}(c) \urcorner \ | \ Z \in F\}$ so $g(c)=\sigma(g)(c)$, as desired.\\
For the converse, let $\sigma \in Aut(\mathfrak{M}/ \ulcorner A \urcorner, \germ(g,p_{A}))$ we want to show that $\sigma(\ulcorner\{ \germ(h_{Z}, p_{A}) \ | \ Z \in F\} \urcorner)=\ulcorner\{ \germ(h_{Z}, p_{A}) \ | \ Z \in F\} \urcorner$. Let $c$ be a realization of $p_{A}(x)$ sufficiently generic over $B$ %\{\ulcorner A \urcorner\} \cup \{ \ulcorner Z \urcorner, \ulcorner \sigma(Z) \urcorner \ | \ Z \in F \}$, 
by hypothesis $g(c)=\sigma(g)(c)$. Then:
\begin{align*}
g(c)=\ulcorner \{\ulcorner h_{Z}(c) \urcorner \ | \ Z \in F \} \urcorner = \ulcorner \{ \ulcorner h_{\sigma(Z)}(c) \urcorner \ | \ Z \in F\} \urcorner=\sigma(g)(c).
\end{align*}
Therefore,  for each $Z \in F$ there is some $Z' \in F$ such that $h_{Z}(c)=h_{\sigma(Z')}(c)$ and this implies that $\germ(h_{Z},p_{A})=\germ(h_{\sigma(Z')},p_{A})$. Thus 
\begin{align*}
\sigma \big( \{ \germ(h_{Z},p_{A}) \ | \ Z \in F\} \big)= \{ \germ(h_{\sigma(Z)},p_{A}) \ | \ Z \in F\}= \{ \germ(h_{Z},p_{A}) \ | \ Z \in F\}.
\end{align*}
We conclude that $\sigma(\ulcorner\{ \germ(h_{Z}, p_{A}) \ | \ Z \in F\} \urcorner)= \ulcorner\{ \germ(h_{Z}, p_{A}) \ | \ Z \in F\} \urcorner$, as desired.
 \end{proof}
Consequently, $F$ is coded by the tuple $(\ulcorner A \urcorner, \germ(g,p_{A}))$ which is a sequence of elements in $\mathcal{G}$.
\end{proof}

\item \textit{Case $2$: All the projections are different i.e. $A_{Z} \neq A_{Z'}$ for all $Z\neq Z' \in F$.}
\begin{proof}
 To simplify the notation fix some enumeration of the projections $\{ A_{Z} \ | \ Z \in F\}$ say $\{A_{1}, \dots, A_{n}\}$. Let $W=\{ \{ x_{1}, \dots, x_{n} \}  \ | \ x_{i} \in A_{i}\}$, such set is independent from the choice of the enumeration. Each set $\{ x_{1}, \dots, x_{n} \}\in W$ admits a code in the home sort $K$, because fields uniformly code finite sets. We denote by $W^{*}= \{ \ulcorner \{ x_{1},\dots, x_{n} \} \urcorner \ | \  \{ x_{1},\dots, x_{n} \} \in W \}$, i.e. the set of all these codes. \\
For each $x^{*} \in W^{*}$, we define the function $f_{x^{*}}: S \rightarrow K$  that sends $A_{Z} \mapsto x_{Z}$, where $x_{Z}$ is the unique element in the set coded by $x^{*}$ that belongs to $A_{Z}$. Let 
$l_{x^{*}}:S \rightarrow \mathcal{G}$ the function given by sending $A_{Z} \mapsto  \ulcorner h_{Z}(f_{x^{*}}(A_{Z})) \urcorner$.\\
This map sends the projection $A_{Z}$ to the code of the fiber in the module $Z$ at the point $x_{Z}$, which is the unique point in the set coded by $x^{*}$ that belongs to $A_{Z}$ [See Figure $2$]. 
\begin{center}
\begin{tikzpicture}
\draw (-2,0) --(0,0);
\draw(-2,1) rectangle (0,2.5);
%nombre Z
\coordinate[label=above right:\textbf{Z}] (Q) at (-2,2.5);
%nodo1
\coordinate[label=below:$x_{Z}$] (A) at (-1,0);
\node at (A)[circle,fill,inner sep=1pt]{};
\draw[-stealth] (-1.3,0.3) --(-1,0.8);
%fibra 1
\draw[dashed](-1,1)--(-1,2.5);
%punto A_{Z}

\coordinate[label=below:\small{$l_{x^{*}}(A_{Z})$}] (H) at (-1,3.3);
%\node at (B)[circle,fill,inner sep=1pt]{};
\draw[-stealth] (-1,2.8) --(-1,2.5);
\coordinate[label=above left :$A_{Z}$] (X) at (-1.2,0);
\node at (X){};
%nodo 2
\draw(0.5,0) --(2.5,0);
\draw(0.5,1) rectangle (2.5,2.5);

\coordinate[label=below:$x_{Z'}$] (B) at (0.7,0);
\node at (B)[circle,fill,inner sep=1pt]{};
\draw[-stealth] (1.3,0.3) --(0.7,0.9);
%fibra A_{Z'}
\draw[dashed](0.7,1)--(0.7,2.5);
\coordinate[label=below:\small{$l_{x^{*}}(A_{Z'})$}] (M) at (0.7,3.3);
\node at (M){};
\draw[-stealth] (0.7,2.8) --(0.7,2.5);
%punto A_{Z'}
%nombre del torsor
\coordinate[label=above right:\textbf{Z}'] (R) at (2,2.5);
\node at (R){};

\coordinate[label=above left:$A_{Z'}$] (Y) at (2,0);
\node at (Y){};
%nodo 3
\draw (3,0) --(5,0);
\draw(3,1) rectangle (5,2.5);
\coordinate[label=below:$x_{Z''}$] (B) at (4.8,0);
\node at (B)[circle,fill,inner sep=1pt]{};
\draw[-stealth] (3.3,0.4) --(4.8,0.9);
%fibra A_{z''}
\draw[dashed](4.8,1)--(4.8,2.5);
\coordinate[label=below:\small{$l_{x^{*}}(A_{Z''})$}] (N) at (4.8,3.3);
\node at (N){};
\draw[-stealth] (4.8,2.8) --(4.8,2.5);
\coordinate[label=above left :$A_{Z''}$] (Z) at (3.5,0);
\node at (Z){};
\coordinate[label=above right:\textbf{Z}''] (Q) at (3,2.5);
%\node at (R){};
\coordinate[label=left:\textbf{Figure 2}] (L) at (-2,3.5);
\node at (L){};
\end{tikzpicture}
\end{center}
\begin{claim}
For each $x^{*}\in W^{*}$  the functions $f_{x^{*}}$ and $l_{x^{*}}$ can be coded in $\mathcal{G}$.
\end{claim}
\begin{proof}
We argue first for the function $f_{x^{*}}$. If $S$ is primitive over $\ulcorner f_{x^{*}}\urcorner$ the statement follows by Lemma $(2)$ \ref{codingeasy}. If $S$ is not primitive over $\ulcorner f_{x^{*}}\urcorner$ then there is an equivalence relation $E$ definable over $\big( S \cup \ulcorner f_{x^{*}} \urcorner \big)$ and let $C_{1}, \dots, C_{l}$ be the equivalence classes of $E$. For each $i \leq l$,  $|C_{i}| \leq m$ and let $f_{x^{*}}^{i}=f_{x^{*}}\upharpoonright_{C_{i}}: C_{i} \rightarrow K$. For each $i \leq l$ $\ulcorner f_{i}\urcorner$ is interdefinable with a tuple $c_{i}$ of elements in $\mathcal{G}$, this follows by combining  Fact \ref{inter} and $II_{k}$  for each $k \leq m$. Because $I_{l}$ holds, the set $\{ c_{1},\dots, c_{l}\}$ admits a code $c$ in the stabilizer sorts. Then $\ulcorner f_{x^{*}}\urcorner$ and $c$ are interdefinable.\\
For the function $l_{x^{*}}$, the statement follows immediately by the induction hypothesis $B_{n}$. 
\end{proof}
 By compactness we can uniformize all such codes, so we can define the function $g: W^{*} \rightarrow \mathcal{G}$ by sending $x^{*}  \mapsto (\ulcorner f_{x^{*}}\urcorner, \ulcorner l_{x^{*}} \urcorner)$.\\

By Lemma \ref{tipomaster} there is some $\ulcorner W^{*}\urcorner$-definable type $q(x^{*}) \vdash x^{*} \in W^{*}$. The second part of Lemma  \ref{tipomaster} also guarantees that given $d^{*}$ a generic realization of $q$ over a set of parameters $B$, if we take $Y$ the set coded by $d^{*}$ and $b$ is the element in $Y$ that belongs to $A_{Z}$ then $b$ is a sufficiently generic realization over $B$ of some type $p_{A_{Z}}(x)$ which is $\ulcorner A_{Z}\urcorner$-definable and extends the generic type of $A_{Z}$.We recall as well that the types $p_{A_{Z}}(x)$ given by Lemma \ref{tipomaster} are all compatible under the action of $Aut(\mathfrak{M}/ \ulcorner F \urcorner)$, this is for any $\sigma \in Aut(\mathfrak{M}/ \ulcorner F \urcorner)$ if $\sigma(Z)=Z'$ then $\sigma(p_{A_{Z}}(x))=p_{A_{\sigma(Z)}}(x)$. By Theorem \ref{codinggerm} the germ of $g$ over $q$ can be coded in the stabilizer sorts $\mathcal{G}$ over $\ulcorner W^{*} \urcorner \in \dcl^{eq}(\ulcorner S \urcorner)$ . By Lemma \ref{tricktorsors} we may assume  $\ulcorner S \urcorner \in \mathcal{G}$.\\

\begin{claim} The tuple $(\germ(g, q), \ulcorner S \urcorner) \in \mathcal{G}$ is interdefinable with $\ulcorner F \urcorner$.
\end{claim}
\begin{proof}
It is clear that $(\germ(g,q),\ulcorner S \urcorner) \in \dcl^{eq}(\ulcorner F \urcorner)$. For the converse, let  $\sigma \in Aut(\mathfrak{M}/ \germ(g,q), \ulcorner S \urcorner)$ we want to show that $\sigma(F)=F$. By Corollary \ref{torgen} the code of each torsor  $Z \in F$ is interdefinable with the pair $(A_{Z}, \germ(h_{Z},p_{A_{Z}}))$. Hence it is sufficient to argue that:
\begin{align*}
\sigma(\{ (A_{Z}, \germ(h_{Z}, p_{A_{Z}})) \ | \ Z \in F \})=\{ (A_{Z}, \germ(h_{Z}, p_{A_{Z}})) \ | \ Z \in F \}.
\end{align*}
We have that $\sigma(\ulcorner W^{*} \urcorner)=\ulcorner W^{*}\urcorner$ because $\sigma(S)=S$. Therefore $\sigma(\germ(g,q))=\germ(\sigma(g),q)=\germ(g,q)$. Let $B$ be the set of parameters required to define all the objects that have been mentioned so far. For any realization $d^{*}$ of the type $q$ sufficiently generic over $B$ we have $g(d^{*})=\sigma(g)(d^{*})$, where $\sigma(g)$ is the function sending an element $x^{*}$ in $W^{*}$ to the tuple $(\sigma(f)_{x^{*}},  \sigma(l)_{x^{*}})$. As a result, $(\ulcorner f_{d^{*}} \urcorner, \ulcorner l_{d^{*}}\urcorner)= (\ulcorner \sigma(f)_{d^{*}} \urcorner, \ulcorner \sigma(l)_{d^{*}}\urcorner)$.  Let $D=\{ d_{t} \ | \ t \in S\}$ be the set of elements coded by $d^{*}$.  The action of $\sigma$ is just permuting the elements of the graph $f_{d^{*}}$, because $\ulcorner f_{d^{*}} \urcorner=\ulcorner \sigma(f)_{d^{*}} \urcorner$.\\
The function $f_{d^{*}}:S \rightarrow K$ sends a $1$-torsor $t$ to the unique element $d_{t} \in D$ such that $d_{t} \in t$, then $\sigma$ is sending the pair $(t, d_{t})$ to $(\sigma(t), d_{\sigma(t)})$, where  $d_{\sigma(t)}$ is a realization $p_{\sigma(t)}(x)$ sufficiently generic over $B$. By assumption, we also have that $\ulcorner l_{d^{*}} \urcorner= \ulcorner \sigma(l)_{d^{*}} \urcorner$, thus the action of $\sigma$ is a bijection among the elements of the graph of $l_{d^{*}}$. Consequently, for any $t \in S$ there is some unique $t' \in S$ such that $\sigma((t, h_{Z}(d_{t})))=(\sigma(t), h_{\sigma(Z)}(d_{\sigma(t)}))=(t', h_{\sigma(Z)}(d_{t'}))$. 
Thus $\sigma(Z)$ is a torsor whose projection is $t \in S$, and $d_{t}\in t$ is a realization of the type $p_{t}(x)$ sufficiently generic over $B$. As a result, $\sigma\big(t, \germ(h_{Z},p_{t}) \big)=\big(t', \germ(h_{\sigma(Z)}, p_{t'})\big)$. We conclude that:  
\begin{align*}
\sigma \big(& \{ (A_{Z}, \germ(h_{Z}, p_{A_{Z}})) \ | \ Z \in F \} \big)= \sigma(\big\{ (t, \germ(h_{Z},p_{t})) \ | \ t \in S\} \big)\\
&=  \{ (t', \germ(h_{\sigma(Z)},p_{t'})) \ | \ t' \in S\}=  \{ (A_{Z}, \germ(h_{Z}, p_{A_{Z}})) \ | \ Z \in F \}, \ \text{as desired.}
\end{align*}
\end{proof}
This finalizes the proof for \textit{Case $2$}.
\end{proof}
\end{enumerate}
Consequently $A_{n+1}$ holds. We prove $B_{n+1}$, i.e. every definable function $f: S \rightarrow F$ where $S$ is a finite set of at most $m+1$ torsors and $F$ is a finite set of torsors of complexity at most $n$ can be coded in $\mathcal{G}$. We recall that without loss of generality we may assume that $S$ is primitive over $\ulcorner f \urcorner$, so $F$ is also a primitive set. By primitivity $f$ is either constant or injective, if it is constant equal to $c$ then $\ulcorner f \urcorner$ is interdefinable with $(\ulcorner S \urcorner, c)$. By Proposition \ref{torsormodule} and Lemma \ref{codetorsors} this tuple is interdefinable with a tuple in $\mathcal{G}$. Thus we may assume that $f$ is an injective function.
By primitivity of $F$, all the torsors in $F$ are of the same type and the projections to the last coordinate are either all equal or all different. 
We proceed again by cases.
\begin{enumerate}
    \item \textit{Case $1$: The projections are all equal, i.e.  there is a torsor $A$ such that $A=A_{Z}$ for all $Z \in F$.}
    \begin{proof}
    We fix $p_{A}(x)$ be some global type extending the generic type of $A$, it is $\ulcorner A \urcorner$-definable by Corollary \ref{alldef}. Let $f:S \rightarrow F$ be a definable injective map. For each $x \in A$ we define the function $g_{x}: S \rightarrow  \mathcal{G}$ by sending $ t \mapsto \ulcorner h_{f(t)}(x) \urcorner$.

This is the function that sends each torsor $t \in S$ to the fiber at $x$ of the torsor $f(t) \in F$. [See Figure $1$].
\begin{center}
\begin{tikzpicture}
%\caption{M1} \label{fig:M1}
%\filldraw[black] (0,0) circle (1pt);
\draw (0,0) ellipse (1 and 0.3);
\draw(-1.2,1) rectangle (1,2.5);
\draw (2.5,0) ellipse (1 and 0.3);
\draw(1.5,1) rectangle (3.5,2.5);
\draw (5,0) ellipse (1 and 0.3);
\draw(4,1) rectangle (6,2.5);
%lineas modulos resaltadas
\draw[very thick](-1.2,1)--(1,1);
\draw[very thick](1.5,1)--(3.5,1);
\draw[very thick](4,1)--(6,1);
%lineas verticales
\draw[dashed](0.5,1)--(0.5,2.5);
\draw[dashed](3,1)--(3,2.5);
\draw[dashed](5.5,1)--(5.5,2.5);
%puntos distuinguidos x
\coordinate[label=below:$x$] (A) at (0.48,1);
\node at (A)[circle,fill,inner sep=1pt]{};
\coordinate[label=below:$x$] (B) at (3,1);
\node at (B)[circle,fill,inner sep=1pt]{};
\coordinate[label=below:$x$] (C) at (5.5,1);
\node at (C)[circle,fill,inner sep=1pt]{};
%puntos distinguidos de torsores
\coordinate[label=right :$t$] (D) at (-0.2,0);
\node at (D){};
\coordinate[label=right :$t'$] (E) at (2.3,0);
\node at (D){};
\coordinate[label=right :$t''$] (F) at (5,0);
\node at (F){};
%flechas function f
\draw[-stealth] (0,0.4) --(0.3,1.4);
\draw[-stealth] (2.5,0.4) --(2.8,1.4);
\draw[-stealth] (5,0.4) --(5.3,1.4);
%puntos distinguidos imagenes
\coordinate[label=above:\tiny{$h_{f(t)}(x)$}] (1) at (0.3,2.7);
\node at (1){};
\draw[-stealth] (0.5,2.7) --(0.5,2.5);
\coordinate[label=above:\tiny{$h_{f(t')}(x)$}] (2) at (2.8,2.7);
\node at (2){};
\draw[-stealth] (3,2.7) --(3,2.5);
\coordinate[label=above:\tiny{$h_{f(t'')}(x)$}] (3) at (5.3,2.7);
\node at (3){};
\draw[-stealth] (5.5,2.7) --(5.5,2.5);
%label imagenes// cordenadas
\coordinate[label=right:$f(t)$] (1) at (-1.2,2.2);
\node at (1){};
\coordinate[label=right:$f(t')$] (2) at (1.5,2.2);
\node at (2){};
\coordinate[label=right:$f(t'')$] (3) at (4.1,2.2);
\node at (3){};

%puntos distinguidos de torsores
\coordinate[label=right :$t$] (D) at (-0.2,0);
\node at (D){};
\coordinate[label=right :$t'$] (E) at (2.3,0);
\node at (E){};
\coordinate[label=right :$t''$] (F) at (5,0);
\node at (F){};
%puntos distinguidos de A
\coordinate[label=below left :\textbf{A}] (X) at (-0.5,0.9);
\node at (X){};
\coordinate[label=below left :\textbf{A}] (Y) at (2.2,0.9);
\node at (Y){};
\coordinate[label=below left:\textbf{A}] (Z) at (4.7,0.9);
\node at (Z){};
%punto de label del grafico
\coordinate[label=left:\textbf{Figure 1}] (L) at (-1.2,3);
\node at (L){};
\end{tikzpicture}
\end{center}
By the induction hypothesis $B_{n}$ for each $x \in A$, the function $g_{x}$ can be coded in $\mathcal{G}$ because its range is of lower complexity. By compactness we can uniformize such codes, so we can define the function $r: A \rightarrow \mathcal{G}$ by sending $x \mapsto\ulcorner g_{x} \urcorner$ .\\
By Theorem \ref{codinggerm} the germ of $r$ over $p_{A}(x)$ can be coded in $\mathcal{G}$ over $\ulcorner A \urcorner$. By Lemma \ref{codetorsors} the set $S$ admits a code $\ulcorner S \urcorner$ in the stabilizer sorts.\\
%\textcolor{blue}{Warning on the following Claim}\\
\begin{claim} {The code $\ulcorner f \urcorner$ is interdefinable with $(\ulcorner A \urcorner, \ulcorner S \urcorner, \germ(r,p_{A}))$, and the later is a sequence of elements in $\mathcal{G}$.}\end{claim}
\begin{proof}
It is clear that $(\ulcorner A \urcorner, \ulcorner S \urcorner, \germ(r,p_{A})) \in\dcl^{eq}(\ulcorner f \urcorner)$.% In fact given $f$, $A$ can be recovered by just taking the projection into the last coordinate of any  torsor $Z$ coded by an element $\ulcorner Z \urcorner  \in range(f)$. Once such set is recovered, we can easily defined $r$ using $f$ (as above). Furthermore, $d$ is just the code of $S$, which is the domain of $f$.\\
We want to show that $\ulcorner f \urcorner \in \dcl^{eq}( \ulcorner A \urcorner, \ulcorner S \urcorner, \germ(r,p_{A}))$. Let $\sigma \in Aut(\mathfrak{M}/  \ulcorner A \urcorner,\ulcorner S \urcorner, \germ(r,p_{A}))$. By Corollary \ref{torgen} for each torsor $Z \in F=\{ f(t) \ | \ t \in S\}$, the code $\ulcorner Z \urcorner$ is being identified with the tuple $(\ulcorner A \urcorner, \germ(h_{Z},p_{A}))$. Thus, the function $f$ is interdefinable over $\ulcorner A \urcorner$ with the function: $f':S \rightarrow \mathcal{G}$ that sends  $t\mapsto \germ(h_{f(t)},p_{A})$. So, it is sufficient to argue that $\sigma(\ulcorner f' \urcorner)= \ulcorner f' \urcorner$. Let $B$ be the set of parameters required to define all the objects that have been mentioned so far. For any realization $c$ of $p_{A}(x)$ sufficiently generic over $B$ %$\ulcorner A \urcorner, \ulcorner f \urcorner, \ulcorner \sigma(f) \urcorner$, 
we must have that $r(c)= \sigma(r)(c)$. Because $\germ(r,p_{A})=\sigma(\germ(r,p_{A}))=\germ(\sigma(r),p_{A})$. By definition, $r(c)=\ulcorner g_{c} \urcorner$ and  $\sigma(r)(c)= \ulcorner \sigma(g)_{c} \urcorner$, where
$ \sigma(g)_{c}: S \rightarrow \mathcal{G}$ is the function that sends  $t  \mapsto \ulcorner h_{\sigma(f)(t)}(c)\urcorner$.\\
For any torsor $t \in S$ there must be a unique element $t' \in S$ such that $\sigma(t')=t$ and $h_{f(t)}(c)=h_{\sigma(f)(\sigma(t'))}(c)$, as $g_{c}=\sigma(g)_{c}$. The later implies that $\germ(h_{f(t)}, p_{A})=\germ(h_{\sigma(f)(\sigma(t'))},p_{A})$. We conclude that $\sigma(t', \germ(h_{f(t')},p_{A}))=(t,\germ(h_{f(t)},p_{A}))$ meaning that $\sigma$ is acting as a bijection among the elements in the graph of $f'$. Therefore, $\sigma(\ulcorner f' \urcorner)=\ulcorner f'\urcorner$, as desired. 
\end{proof}
This completes the proof for the first case.
\end{proof}
  \item \textit{Case $2$: All the projections are different i.e. $A_{Z} \neq A_{Z'}$ for all $Z\neq Z' \in F$.}
\begin{proof}
Let $f: S \rightarrow F$ be a definable injective function where $S$ is a finite set of $1$-torsors primitive over $\ulcorner f \urcorner$. We consider the definable function that sends each torsor $t \in S$ to the code of the projection into the last coordinate of the torsor $f(t)\in F$, more explicitly: 
\begin{center}
$\pi_{n+1} \circ f:\begin{cases}
 S & \rightarrow \mathcal{G}\\
t& \mapsto \ulcorner \pi_{n+1}(f(t)) \urcorner.
\end{cases}$
\end{center}
%\textcolor{red}{include picture}
By Lemma (3) \ref{codingeasy}, $\pi \circ f$ can be coded in $\mathcal{G}$, and by $A_{n+1}$ the finite set $F$ is coded by a tuple in $\mathcal{G}$. 
 It is  sufficient to show the following claim:\\
 
\begin{claim}{ The code $\ulcorner f \urcorner$ is interdefinable with the tuple $(\ulcorner \pi \circ f \urcorner , \ulcorner F \urcorner)$, which is a tuple in the stabilizer sorts.} \end{claim}
\begin{proof}
Clearly $(\ulcorner \pi \circ f \urcorner, \ulcorner F \urcorner) \in \dcl^{eq}(\ulcorner f \urcorner)$. Note that $\ulcorner S \urcorner \in \dcl^{eq}(\ulcorner \pi \circ f \urcorner)$ because $S$ is the domain of the given function, we can describe the function $f: S  \rightarrow  F$ by sending $t  \mapsto \ulcorner Z_{t} \urcorner$, where $Z_{t}$ is the unique torsor in $F$ such that $\ulcorner \pi_{n+1}(Z_{t})\urcorner =(\pi \circ f ) (t)$, we conclude that $\ulcorner f \urcorner \in \dcl^{eq}(\ulcorner \pi \circ f \urcorner , \ulcorner F \urcorner)$. 
As a consequence, $f$ is coded in $\mathcal{G}$ by the tuple $(\ulcorner \pi \circ f \urcorner , \ulcorner F \urcorner)$. \end{proof}
This finalizes the proof for the second case. 
\end{proof}
\end{enumerate}
Consequently, $A_{n}$ and $B_{n}$ hold for all $n \in \mathbb{N}$. The statement follows.
\end{proof}
We continue arguing that $I_{m+1}$ holds for $r=1$.
\begin{proposition}\label{coding2} Let $F \subseteq \mathcal{G}$ be a finite set of size $m+1$ then $F$ admits a code in $\mathcal{G}$. 
\end{proposition}
\begin{proof}
 If $F$ is not primitive  we show that $\ulcorner F \urcorner$ can be coded in $\mathcal{G}$, by using Fact \ref{inter} and  the induction hypothesis $I_{k}$ for $k \leq m$. We may assume that $F$ is a primitive set,  so all the elements of $F$ lie in the same sort. If $F$ is either contained in the main field or the residue field, then $F$ is coded by a tuple of elements in the same field, because fields code uniformly finite sets. If $F \subseteq \Gamma/\Delta$ for some $\Delta \in RJ(\Gamma)$ the statement follows as there is a definable order over the elements of $F$. If $F \subseteq B_{n}(K)/\Stab_{(I_{1},\dots, I_{n})}$  for some $n\geq 2$, by Proposition \ref{coding1}  $F$ admits a code in $\mathcal{G}$. (Indeed, $\mathcal{O}$-modules are in particular torsors).
\end{proof}
We continue showing that $II_{m+1}$ holds, we first prove the following statement.
\begin{proposition}\label{coding3} Let $F$ be a finite set of torsors of size $m+1$ and $f: F \rightarrow P$ be a definable bijection, where $P$ is a finite set of torsors. Suppose that $F$ is primitive over $\ulcorner f\urcorner$, then $\ulcorner f \urcorner$ is interdefinable with a tuple of elements in $\mathcal{G}$. 
\end{proposition}
\begin{proof}
We proceed by induction on the complexity of the torsors in $F$. The base case follows directly by Proposition \ref{coding1}. We assume the statement for any set of torsors $F$ with complexity $n$ and we prove it for complexity $n+1$. By primitivity all the projections into the last coordinate are either equal or all distinct. For each torsor $Z \in F$ we denote as $A_{Z}$ the projection of $Z$ into the last coordinate. We argue by cases:
\begin{enumerate}
\item \emph{Case $1$: All the projections are equal  and let  $A=A_{Z}$ for all $Z \in F$. }\\
For each $x \in A$, let $\mathcal{I}_{x}=\{ \ulcorner h_{Z}(x) \urcorner \ | \ Z \in F\}$ which describes the set of fibers at $x$. We define $B= \{ x \in A \ | \ |\mathcal{I}_{x}| = | F | \}$ which is a $\ulcorner F \urcorner$-definable set. For each $y \in B$ we consider the map $g_{y}: \mathcal{I}_{y} \rightarrow P$ defined by sending $h_{Z}(y)  \mapsto  f(Z)$, which is the function that sends each fiber to the image of the torsor under $f$. By the induction hypothesis we can find a code $\ulcorner g_{y} \urcorner$ in $\mathcal{G}$, and by compactness we can uniformize such codes. Therefore we can define the function:
$r:B \rightarrow \mathcal{G}$ by sending $y  \mapsto \ulcorner g_{y} \urcorner$.\\
Let $p_{A}(x)$ be a global complete type containing the generic type of $A$, it is $\ulcorner A \urcorner$- definable by Corollary \ref{alldef}. By Corollary \ref{torgen}, $p_{A}(x) \vdash x \in B$. In fact, if we fix a realization of the generic type $c$ of $p_{A}(x)$ sufficiently generic over $\{ \ulcorner Z \urcorner \ | \ Z \in F\}$, and  $Z\neq Z' \in F$ then the fibers $h_{Z}(c)$ and $h_{Z'}(c)$ must be different. By Theorem \ref{codinggerm} the germ of $r$ over $p_{A}(x)$ can be coded in $\mathcal{G}$ over $\ulcorner A \urcorner$.  \\
%\textcolor{blue}{Warning on the following Claim}\\
\begin{claim}{The code $\ulcorner f \urcorner$ is interdefinable with $(\germ(r,p_{A}), \ulcorner F \urcorner)$ which is a tuple in the stabilizer sorts $\mathcal{G}$.}\end{claim}
\begin{proof}
Clearly $(\germ(r,p_{A}), \ulcorner F \urcorner) \in \dcl^{eq}(\ulcorner f \urcorner)$. We will argue that for any automorphism \\$\sigma \in Aut(\mathfrak{M}/ \ulcorner F \urcorner, \germ(r,p_{A}))$ we have $\sigma(\ulcorner f \urcorner)=\ulcorner f \urcorner$. As each torsor $Z \in F$ is being identified with the tuple $(\ulcorner A \urcorner, \germ(h_{Z},p_{A}))$, and $\ulcorner A \urcorner \in \dcl^{eq}(\ulcorner F \urcorner)$ then  it is sufficient to argue that:
\begin{align*}
\sigma \big( \{ (\germ(h_{Z},p_{A}), f(Z)) \ | \ Z \in F\}\big)=\{ (\germ(h_{Z},p_{A}), f(Z)) \ | \ Z \in F\}. 
\end{align*}
 For any $Z \in F$ there is a unique torsor $Z' \in F$ such that $\sigma(Z')=Z$, because $\sigma(\ulcorner F \urcorner)=\ulcorner F \urcorner$.  Let $D$ be the set of parameters required to define all the objects that have been mentioned so far. For any realization $c$ of the type $p_{A}(x)$ sufficiently generic over $D$ we have $r(c)=\sigma(r)(c)$, because $\sigma(\germ(r,p_{A}))=\germ(\sigma(r),p_{A})$. Consequently $r(c)=\ulcorner g_{c}\urcorner= \ulcorner \sigma(g)_{c}\urcorner =\sigma(r)(c)$. In particular, $h_{\sigma(Z')}(c)=h_{Z}(c)$ which implies that $\germ(h_{\sigma(Z')},p_{A})=\germ(h_{Z},p_{A})$. In addition, 
$\sigma(f)(\sigma(Z'))=\sigma(g)(h_{\sigma(Z')}(c))= g_{c}(h_{Z}(c))=f(Z)$. 
Therefore,
\begin{align*}
\sigma \big( \{ (\germ(h_{Z},p_{A}) f(Z)) \ | \ Z \in F \} \big)=\{ (\germ(h_{Z},p_{A}), f(Z)) \ | \ Z \in F \}, \text{as desired.}
\end{align*}
\end{proof}
%For the converse, if $F$ is being given we can define $A$ as the projection into the last coordinate of any element $Z \in F$. Let $c \vDash p_{A}(x)$, where $p_{A}(x)$ is the definable type over $\ulcorner A \urcorner$, as $\germ(r,p_{A})$ is being given we can compute $r(c)= \ulcorner g_{c} \urcorner$, thus we can recover the relation $g_{c}= \mathcal{I}_{c} \rightarrow \mathcal{G}$ that sends each fiber at the generic $h_{Z}(c)$ to the torsor $f(Z)$. We can therefore recover $f: F \rightarrow \mathcal{G}$ as the function that sends the torsor $Z$ to $g_{c}(h_{Z}(c))$. 

\item  \emph{Case $2$: All the projections are different. i.e. for all $Z \neq Z' \in F$ we have $A_{Z} \neq A_{Z'}$. }\\
By Proposition \ref{coding1} we can find a code in the stabilizer sorts for $F$, and $\ulcorner F \urcorner \in \dcl^{eq}(\ulcorner f \urcorner)$ as it is the domain of this function. Let $S=\{ A_{Z} \ | \ Z \in F \}$ and define the function $g: S \rightarrow F$ by sending $ A_{Z} \mapsto Z$, where $Z$ is the unique torsor in $F$ satisfying that $\pi_{n+1}(Z)=A_{Z}$. Clearly $g$ is a $\ulcorner F \urcorner$-definable bijection. We consider the map $f \circ g: S \rightarrow P$ that sends $A_{Z} \mapsto  f(Z)$.  By Proposition \ref{coding1}, the function $f \circ g$ admits a code in the stabilizer sorts.\\

\begin{claim} {The code $\ulcorner f \urcorner$ is interdefinable with the tuple $(\ulcorner f \circ g\urcorner, \ulcorner F \urcorner)$ which is a tuple in the stabilizer sorts.}\end{claim}
\begin{proof}
It is clear that $(\ulcorner f \circ g\urcorner, \ulcorner F \urcorner) \in \dcl^{eq}(\ulcorner f \urcorner)$. For the converse note that $S$ is definable over $\ulcorner f \circ g\urcorner$ as it is its domain. As $F$ is given, we can define the function $\pi: F \rightarrow S$ that sends $Z \mapsto A_{Z}$. This is the map that sends each torsor to its projection into the last coordinate. We observe that $f=( f \circ g)  \circ \pi$, in fact $f(Z)= (f \circ g)(A_{Z})$. So $\ulcorner f \urcorner \in \dcl^{eq}(\ulcorner f \circ g\urcorner, \ulcorner F \urcorner)$. \end{proof}
\end{enumerate}
\end{proof}

\begin{proposition}\label{coding4} For every $F \subseteq \mathcal{G}$ finite set of size $m+1$ and definable function $f:F \rightarrow \mathcal{G}$, the code $\ulcorner f \urcorner$ is interdefinable with a tuple of elements in $\mathcal{G}$. 
\end{proposition}
\begin{proof}
Without loss of generality we may assume that $F$ is primitive over $\ulcorner f \urcorner$. Otherwise, there is a $(\ulcorner F \urcorner \cup \ulcorner f \urcorner)$-definable equivalence relation on $F$ and we let $C_{1}, \dots, C_{l}$ be the equivalence classes. For each $i \leq l$ we have $|C_{i}|\leq m$ and let $f_{i}=f\upharpoonright_{C_{i}}$. By the induction hypothesis, for each $k \leq m$ $II_{k}$ the code $\ulcorner f_{i}\urcorner$ is interdefinable with a tuple $c_{i}\in \mathcal{G}$. Because $l \leq m$ and $I_{l}$ holds the set $\{ c_{1},\dots, c_{l}\}$ admits a code $c \in \mathcal{G}$. Then $\ulcorner f \urcorner$ and $c$ are interdefinable. Hence, we may assume that $F$ is primitive over $\ulcorner f \urcorner$. By primitivity $f$ is either constant or injective. If $f$ is constant equal to some $c$ then $\ulcorner f \urcorner$ is interdefinable with the tuple $(\ulcorner F \urcorner, c)$, which lies in the stabilizer sorts by Proposition \ref{coding2}. Summarizing, we may assume that $f$ is an injective function and $F$ is primitive over $\ulcorner f\urcorner$. By primitivity all the torsors of $F$ lie in the same sort. \\
If $F$ is contained in the residue field, then $F$ is interdefinable with the code of a finite set of $1$-torsors of type $\mathcal{M}$ and the statement follows by Proposition \ref{coding3}. If $F \subseteq B_{n}(K)/\Stab_{(I_{1},\dots, I_{n})}$ for some $n\geq 2$, the statement follows by Proposition \ref{coding3}, because $\mathcal{O}$-modules are torsors. If $F \subseteq \Gamma/\Delta$ for some $\Delta \in RJ(\Gamma)$, then we can list the elements of $F$ in increasing order $\gamma_{1}<\dots<\gamma_{m+1}$, and the tuple $(\gamma_{i},f(\gamma_{i}))_{1\leq i \leq m+1}$ lies in the stabilizer sorts and is interdefinable with the code of $f$. \\
It is therefore left to consider the case where $F \subseteq K$. We may assume that $\ulcorner F \urcorner$ is a tuple of elements in the main field, as fields code finite sets. 
 Let $U$ be the smallest closed torsor that contains all the elements of $F$, this is a $\ulcorner F \urcorner$-definable set. 
% Indeed, by primitivity there is some $\delta \in \Gamma$ such that for any two elements $x,y \in F$ we have $v(x-y)=\delta$. By taking any element $x \in F$ we have that $U=x+b\mathcal{O}$ where $b$ is any element whose valuation is $\delta$. Because any field eliminates finite imaginaries, we can find a tuple $c$ in the main field interdefinable with $\ulcorner F \urcorner$ and note that $U$ is $c$-definable. 
Let $g$ the function that sends each element $x \in F$ to the unique class of $\red(U)$ that contains such element. Let $s$ be the $\mathcal{O}$- lattice whose code is interdefinable with $\ulcorner U\urcorner$, and let $h= \red(U) \rightarrow \red(s)$ be the map given by Lemma \ref{injectiveresidue}. Let $D=h \circ g(F)$, which is an $\ulcorner F \urcorner$-definable finite subset of $\red(s)$. By Proposition \ref{coding3}, the composition $f \circ g^{-1} \circ h^{-1}= D \rightarrow \mathcal{G}$ can be coded in the stabilizer sorts $\mathcal{G}$.  As $h \circ g= F \rightarrow D$ is a $\ulcorner F \urcorner$-definable bijection, then $f$ is interdefinable with the tuple $(\ulcorner F \urcorner, \ulcorner f \circ g^{-1} \circ h^{-1} \urcorner)$ which is a sequence of elements in $\mathcal{G}$. 
\end{proof}
Finally, we conclude proving that $I_{m+1}$ holds for $r>0$.
\begin{proposition} For any $r >0$ let $F \subseteq \mathcal{G}^{r}$ be a finite set of size $m+1$. Then $F$ can be coded in $\mathcal{G}$. 
\end{proposition}
\begin{proof}
Let  $r>0$ and $F$ be a finite set of $\mathcal{G}^{r}$ of size $m+1$. Suppose that $F$ is not primitive, that means that we can find a non trivial equivalence $E$ relation definable over $\ulcorner F \urcorner$, and let $C_{1},\dots, C_{l}$ be such classes. For each $i \leq l$, $|C_{i}|\leq m$, because $I_{k}$ holds for each $k\leq m$ we can find a code $c_{i} \in \mathcal{G}$. As $l \leq m$ by $I_{l}$ holds,  we can find a code $c$ in the stabilizer sorts of the set $\{ c_{1},\dots, c_{l}\}$, because $l<m+1$. The code $\ulcorner F \urcorner$ is interdefinable with $c$.\\
We assume that $F$ is a primitive set. Let $\pi_{i}= \mathcal{G}^{r} \rightarrow \mathcal{G}$ be the projection into the $i-th$ coordinate. By primitivity of $F$ each projection $\pi_{i}$ es either constant or injective. As $|F|>1$ there must be an index $1 \leq i_{0}\leq r$ such that $\pi_{i_{0}}$ is injective and $F_{0}=\pi_{i_{0}}(F)$ is a primitive finite subset of $\mathcal{G}$. By Proposition \ref{coding2}  we can find a code $\ulcorner F_{0} \urcorner$ in $\mathcal{G}$. For each other index $i \neq i_{0}$, by Proposition \ref{coding4} we have that $\pi_{i} \circ \pi_{i_{0}}^{-1}= F_{0} \rightarrow \mathcal{G}$ can be coded in the stabilizer sorts. Then $\ulcorner F \urcorner$ is interdefinable with the tuple $(\ulcorner F_{0} \urcorner, (\ulcorner \pi_{i} \circ \pi_{i_{0}}^{-1} \urcorner)_{i \neq i_{0}})$ which is a tuple in the stabilizer sorts, as required. 
\end{proof}
This completes the induction on the cardinality of the set $F$. Because $I_{m}$ holds for each $m \in \mathbb{N}$ we can conclude with the following statement.
\begin{theorem}\label{codingfiniteset} Let $r>0$ and $F \subseteq \mathcal{G}^{r}$, then $\ulcorner F \urcorner$ is interdefinable with a tuple of elements in $\mathcal{G}$.
\end{theorem}

\subsection{Putting everything together}\label{todo}

We conclude this section with our main theorem. 
\begin{theorem}\label{EIdp} Let $K$ be a henselian valued field of equicharacteristic zero, residue field algebraically closed and dp-minimal value group. Then $K$ eliminates imaginaries in the language $\hat{\mathcal{L}}$, where the stabilizer sorts are added. 
\end{theorem}
\begin{proof}
By Theorem \ref{weakEI1}, $K$  has weak elimination of imaginaries down to the stabilizer sorts. By Fact \ref{all} it is sufficient to show that finite sets can be coded, this is guaranteed by Theorem \ref{codingfiniteset}. 
\end{proof}

\bibliographystyle{plain}

  \end{document}